\documentclass[bj,authoryear]{imsart}

\usepackage[ruled,linesnumbered]{algorithm2e}
\usepackage{booktabs}
\usepackage{ulem}
\usepackage{epsfig}
\usepackage{epstopdf}

\usepackage[utf8]{inputenc}
\usepackage[T1]{fontenc}

\usepackage{xr}

\startlocaldefs
\theoremstyle{plain}

\newtheorem{theorem}{Theorem}[section]
\newtheorem{lemma}[theorem]{Lemma}

\newtheorem{proposition}{Proposition}
\newtheorem{corollary}{Corollary} 
\theoremstyle{remark}
\newtheorem{definition}[theorem]{Definition}

\newtheorem{remark}{Remark} 

\numberwithin{equation}{section}
\DeclareMathOperator*{\argmin}{arg\,min}
\DeclareMathOperator*{\argmax}{arg\,max}

\def\E{{\mathbf E}}
\def\P{{\mathbf P}}
\def\Pa{{\mathcal P}}

\def\D{{\mathcal D}}
\def\O{{\mathcal O}}
\def\R{{\mathbb R}}
\def\S{{\cal S}}
\def\L{{\cal L}}

\def\A{{\small\mathcal{A}}}
\def\C{{\small\mathcal{C}}}

\def\I{\mbox{I}}
\def\red{\color{black}}

\endlocaldefs

\begin{document}

\begin{frontmatter}
	\title{Consistency of the oblique decision tree and its boosting and random forest}
	\runtitle{Consistency of Oblique Decision Tree}

\begin{aug}
	\author[AC]{\fnms{Haoran}~\snm{Zhan}\ead[label=e1]{haoran.zhan@u.nus.edu}}
	\author[B]{\fnms{Yu}~\snm{Liu}\ead[label=e2]{liuy8stat@sicnu.edu.cn}}
	\author[AC]{\fnms{Yingcun}~\snm{Xia}\ead[label=e3]{staxyc@nus.edu.sg}}


      \address[AC]{Department of Statistics and Data Science,
	National University of Singapore}
    
	\address[B]{School of Mathematics Science, Sichuan Normal University \\ \printead[presep={\ }]{e1,e2,e3}}


\end{aug}

\begin{abstract}
		Classification and Regression Tree (CART), Random Forest (RF) and Gradient Boosting Tree (GBT) are probably the most popular set of statistical learning methods. However, their statistical consistency can only be proved under very restrictive assumptions on the underlying regression function. As an extension to standard CART, the oblique decision tree (ODT), which uses linear combinations of predictors as partitioning variables, has received much attention. ODT tends to perform numerically better than CART and requires fewer partitions. In this paper, we show that ODT is consistent for very general regression functions as long as they are $L^2$ integrable. Then, we prove the consistency of the ODT-based random forest (ODRF), whether fully grown or not. Finally, we propose an ensemble of GBT for regression by borrowing the technique of orthogonal matching pursuit and study its consistency under very mild conditions on the tree structure. After refining existing computer packages according to the established theory, extensive experiments on real data sets show that both our ensemble boosting trees and ODRF have noticeable overall improvements over RF and other forests.

	\end{abstract}
	
	
\end{frontmatter}


\section{Introduction}

The classification and regression tree \cite[CART]{breiman1984classification} is one of the most popular machine learning algorithms. It has apparent simplicity, visibility and interpretability and is therefore widely used in data mining and analysis. The random forest \cite[RF]{breiman2001random}, as an ensemble method of CART, is arguably a very efficient method for tabular data and is therefore one of the most popular methods in machine learning. Meanwhile, another improvement of CART, called the gradient boosting tree \cite[GBT]{friedman2001greedy}, is also an efficient method in machine learning.

However, \cite{breiman1984classification} has noticed since the time when CART was first proposed that the use of marginal variables as splitting variables could cause problems both in theory and in its numerical performance in classification and prediction.  As a remedy for this disadvantage of CART,  \cite{breiman1984classification} suggested using linear combinations of the predictors as the splitting variables. Later, the method became known as the oblique decision tree \cite[ODT]{Heath93inductionof} and has received much attention. ODT was expected to perform better than CART, requiring fewer splits and therefore smaller trees than CART (see \cite{kim2001classification}), while the linear combination often preserves interpretability in data analysis.  There are also various oblique random forests (ODRF) in the literature, including Forest-RC of \cite{breiman2001random}, Random Rotation Random Forest (RR-RF) of \cite{blaser2016random}, Canonical Correlation Forests (CCF) of \cite{rainforth2015canonical},  Random Projection Forests (RPFs) of \cite{lee2015fast} and Sparse Projection Oblique Random Forests of \cite{tomita2020sparse}. For these forests, the difference is in how to find the linear combination and can be regarded as different implementations of ODT or their random forests. 

\subsection{A review of studies on statistical consistency of decision trees}

Despite the widespread popularity of CART and RF, their statistical consistency has long been a significant theoretical challenge for statisticians and is far from fully resolved. This challenge can be attributed to two primary issues. The first is an inherent limitation of the methods themselves, namely the reliance on marginal variables as partitioning criteria when constructing the trees in CART or RF. The second issue stems from the mathematical complexity involved in analyzing the relationships between different layers of the tree structure.

Early works on the consistency of CART or RF were mainly focused on the simplified version of CART or RF.  The most celebrated theoretical result is that of \cite{Breiman2001}, which offers an upper bound on the generalization error of forests in terms of correlation and strength of the individual trees. This was followed by \cite{breiman2004consistency} which focuses on a stylized version of the original algorithm.  \cite{lin2006random} established lower bounds for non-adaptive forests (that is, independent of the training set) via the connection between random forests and a particular
class of nearest-neighbor predictors; see also \cite{biau2010layered}. In the past twenty years, various theoretical developments have been made, for example, \cite{biau2008consistency}, \cite{ishwaran2010consistency}, \cite{biau2012analysis}, \cite{genuer2012variance}, and \cite{zhu2015reinforcement}, in analyzing the consistency of simplified models. Other efforts to bridge the gap between theory and practice include \cite{denil2013consistency}, which establishes the first consistency result for online random forests, as well as \cite{wager2014asymptotic} and \cite{mentch2014ensemble}, which investigate the asymptotic sampling distribution of forests.

The work of \cite{scornet2015consistency} was a milestone and proved that CART-based RF is consistent in the $L^2$ sense if the unknown regression function is additive of the marginal variables. Following their proofs, \cite{syrgkanis2020estimation}, \cite{klusowski2021universal} and \cite{chi2020asymptotic} showed that RF is also consistent in high-dimensional setting under different modelling assumptions. In particular, \cite{klusowski2021universal} found that there is a relationship between CART and greedy algorithms and also gave a consistency rate $(\ln n)^{-1}$, where $n$ denotes the sample size. \cite{chi2020asymptotic} improved this consistency rate under an additional assumption called sufficient impurity decrease (SID), which includes the additive model as a special case in a high-dimensional setting. However, all the above consistency results for CART or RF are based on very strong restrictions on $m(x)$, such as the additive model or the SID condition.

On the other hand, consistency results for gradient boosting trees (GBT) are relatively scarce in the literature, with most research focusing primarily on improving numerical performance. The early work of \cite{buhlmann2002consistency} and \cite{zhang2005boosting} demonstrated the consistency of $L^2$ boosting trees, but without incorporating the CART scheme. More recently, \cite{zhou2022decision} introduced the CART rule into the boosting process, but their consistency result was limited to a varying coefficient model. In this paper, our objective is to investigate the consistency of GBT for more general regression functions under the CART partitioning scheme. Furthermore, we employ orthogonal matching pursuit and bagging to improve the original boosting tree.

\subsection{Our contributions}

Note that the existing consistency results for decision trees or their corresponding random forests are either proved under very strong assumptions on the unknown regression functions or are only for simplified versions of decision trees that are not practically used.
In view of this, the consistency results in this paper are novel. Our contributions are summarized as follows.
\begin{itemize}
	\item We establish consistency results for ODT as well as ODRF, which is based on either fully grown trees or partially grown trees, for general regression functions as long as they are $L^2$ integrable. The results include those of CART or RF as their special cases or corollaries.

	\item We introduce two methods of feature bagging to improve prediction performance and establish the consistency of the methods.



	\item 
    We also refine the existing packages for ODRF according to the established theory.  Extensive empirical studies have shown that both our ODRF and ensemble boosting trees tend to have superior performance over other methods, including standard RF and other ODRF implementations, which was not clearly demonstrated before using the existing packages; see, for example, \cite{kim2001classification}.

      \item We establish an explicit relationship between neural networks and the gradient boosting tree in \cite{friedman2001greedy}. Then  this relationship is used to prove the consistency of the improved gradient boosting tree.      To our knowledge, the consistency rate is the fastest among all the tree-based regressions. Importantly, our results do not require any additional technical assumptions on the tree structure.
\end{itemize}

During our study of ODT, \cite{cattaneo2022convergence} also presented some consistency results about ODT.  We summarize two main differences between our results and theirs as follows. Firstly, the oracle inequalities in \cite{cattaneo2022convergence} were intentionally designed for high-dimensional regression and hold for functions whose Fourier transformations have finite first moments. Our study focuses on the consistency of  ODT for fixed dimension regression and shows that ODT is consistent for more general regression functions, as long as they are continuous. 
Secondly, we also introduce boosting trees and various random forests based on ODT and study their asymptotic properties in different cases, but these important parts are missing in \cite{cattaneo2022convergence}. Especially, the proof of the consistency of ODRF with fully grown trees is  different from that for a single tree in \cite{klusowski2021universal}. 





\subsection{Organization of this paper}
The rest of this paper is organized as follows. In Section 2, we introduce  notations used in the proofs and an algorithm to describe how to construct  ODT.  In Section 3, we first describe the idea of our proofs and then give our main results for the consistency of ODT. Section 4 presents consistency results for ODRFs based on trees that are either fully grown or partially grown.  In Section 5, we propose an ensemble of boosting trees for regression and analyze its statistical consistency. 
In Section 6, we explain the implementation of our two algorithms based on ODT and compare their numerical performance with RF and other decision forests.

\section{Preliminaries}\label{Notation and Algorithm}

\subsection{Notations}\label{sec:notations}

Suppose that the random variable $Y\in\R$ is the response and the random vector $X\in [0,1]^p$ is the predictor with the distribution $\mu(x),x\in \R^p$. Our interest is to estimate regression function $m(x):=\E(Y|X=x), x\in[0,1]^p$. Denote   by $\D_n=\{(X_i,Y_i)\}_{i=1}^n$ the i.i.d. sample of $(X,Y)$ and $\mathbb{Y}=(Y_1,\ldots,Y_n)^\top\in\R^n$  the response vector.  Let $ \Theta^p = \{ \theta: \theta \in \R^p \ \mbox{and } ||\theta||_2 = 1\}$ be the centered unit sphere in $\R^p$, where $\|\cdot\|_2$ is the $\ell^2$ norm in $\R^p$ space. Denote by $A$ a node of ODT, which  is a subset of $[0,1]^p$, and its two daughters are written as $A^+_{\theta,s}=\{x\in A: \theta^Tx\le s\}$ and $A^-_{\theta,s}=\{x\in A: \theta^Tx> s\}$ satisfying $A=A^+_{\theta,s}\cup A^-_{\theta,s}$. Note that either $A^+_{\theta,s}$ or $A^-_{\theta,s}$ can be empty.  For any node $A$, let $N(A):=Card(\{X_i\in A\})$ be the number of data points in $A$ and $\bar{Y}_A:=\frac{1}{N(A)}\sum_{X_i\in A}{Y_i}$ be the sample mean for data in $A$.
Define $\langle\mathbb{Y}-\bar{Y}_A, \mathbb{Y}+f\rangle_A:=\frac{1}{N(A)}\sum_{X_i\in A}{[(Y_i-\bar{Y}_A)\cdot (Y_i+f(X_i))]}$ and  $\|\mathbb{Y}-f\|_A^2:=\frac{1}{N(A)}\sum_{X_i\in A}{(Y_i-f(X_i))^2}$ for any $f: [0,1]^p\to\mathbb{R}$. If $A=[0,1]^p$, the notation of empirical norm $\|\cdot\|_n$ is used to present $\|\cdot\|_{[0,1]^p}$.
Under above notations, we can introduce  the impurity gain in the regression problem \citep{breiman1984classification}:
\begin{equation}\label{impuritygainreg}
	\Delta_A(\theta,s)= \|\mathbb{Y}-\bar{Y}_{A}\|_{A,n}^2-\left( P(A^+_{\theta,s})\|\mathbb{Y}-\bar{Y}_{A^+_{\theta,s}}\|_{A^+_{\theta,s},n}^2+P(A^-_{\theta,s})\|\mathbb{Y}-\bar{Y}_{A^-_{\theta,s}}\|_{A^-_{\theta,s},n}^2\right),
\end{equation}
where $P(A^+_{\theta,s})=N(A^+_{\theta,s})/N(A)$ and $P(A^-_{\theta,s})=N(A^-_{\theta,s})/N(A)$. Meanwhile, we use $t_n$ to denote the number of terminal nodes (leaves) of an ODT, where $1\le t_n\le n$. 

Next, we introduce the total variation of an univariate function $g(v), v\in [0,1]$: $$\|g\|_{TV}:= \sup_{m\ge 1}\sup_{0\le v_0<\cdots<v_{n}\leq 1}{\sum_{j=0}^{m-1}|f(v_{j+1})-f(v_j)|}.$$ 
One of its useful properties is that $\|g\|_{TV}=\int_0^1{|g'(x)|dx}$ if $g$ is continuously differentiable on $[0,1]$. For any linear combination of ridge functions $g(x)=\sum_{j=1}^J{g_j(\theta_j^T x)}, x\in [0,1]^p$,  we  define its total variation by $\|g\|_{TV}:=\sum_{j=1}^J{\|g_{j}\|_{TV}}$. 

In addition, we use $\mathbb{Z}_+$ to denote the class of positive integers. Finally, $c$ is used to denote a positive constant and $c(a_1,\ldots, a_\aleph)$ is also a positive function  depending  on parameters $a_1,\ldots,a_\aleph$ only. Note that both $c$ and  $c(a_1,\ldots, a_\aleph)$ are allowed to change from line to line.

\begin{figure}[ht]
	\includegraphics[width=0.95\textwidth]{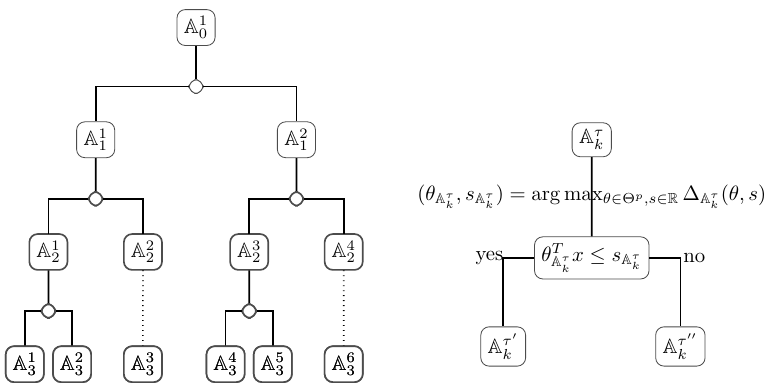}
	\caption{This figure shows an example of $T_{\mathcal{D}_n,6,3}$, which has three layers $\L=3$ and $6$ leaves. To be specific, we have root node $\mathbb{A}_0^1$ in the layer 0, nodes $\mathbb{A}_1^1$ and $\mathbb{A}_1^2$ in the layer 1,  nodes $\mathbb{A}_2^1, \mathbb{A}_2^2, \mathbb{A}_2^3, \mathbb{A}_2^4$ in the layer $\ell = 2$ and leaves $\mathbb{A}_3^1, \mathbb{A}_3^2, \mathbb{A}_3^3, \mathbb{A}_3^4, \mathbb{A}_3^5, \mathbb{A}_3^6, \mathbb{A}_3^7$ in the layer 3. Note that in this case $\mathbb{A}_2^2$ only contains one data point and can not be divided in further steps, which implies that  $\mathbb{A}_2^2=\mathbb{A}_3^3$. It is also noteworthy that no matter how many data points in  $\mathbb{A}_2^4$ we have $\mathbb{A}_2^4=\mathbb{A}_3^6$ because $t_n$ is preset to be 6. Finally, we have estimators $m_{n,2}(x)=\sum_{j=1}^4{\mathbb{I}(x\in \mathbb{A}_2^j)\cdot \bar{Y}_{\mathbb{A}_2^j}}$ and $m_{n,3}(x)=\sum_{j=1}^6{\mathbb{I}(x\in \mathbb{A}_3^j)\cdot \bar{Y}_{\mathbb{A}^j_3}}$ given data $\D_n$.}
	\label{fig:decisiontree}
\end{figure}

\subsection{ODT algorithm}
According to \cite{breiman1984classification}, the best splitting criteria of  node $A$ that contains at least two data points is to choose $(\hat{\theta}_A,\hat{s}_A)$ by  maximizing $\Delta_A(\theta,s)$ over $\mathbb{R}\times\Theta^p$.
Based on the best splitting, the construction of ODT in regression problem is shown in Algorithm \ref{Algorithm.ODTtreereg}.

For a regression tree, we call the root, i.e. original data, as layer 0, and the last layer that contains only leaves, $ \L $. Thus, the layer index $ \ell  $ satisfies $ 0 \le \ell \le \L $.
Let $k_\ell$ be the number of leaves at layer $\ell$, whose corresponding nodes are denoted by $\{\mathbb{A}_\ell^j\}_{j=1}^{k_\ell}$. The estimator of $m(x)$ by using nodes at layer $\ell$ is defined by
$$
m_{n,t_n,\ell}(x)=\sum_{j=1}^{k_\ell}{\mathbb{I}(x\in\mathbb{A}_\ell^j)\cdot \bar{Y}_{\mathbb{A}_\ell^j}}
$$
for each $ 0 \le \ell \le \L $. To simplify the notation, $m_{n,t_n,\ell}(x)$ is sometimes abbreviated to  $m_{n,\ell}(x)$ if there is no confusion.

With the above notations, a regression tree with $t_n$ leaves is written as a triple  $ T_{\D_n,t_n, \L} $ and the final estimator is $ m_{n,\L}(x)$ satisfying $k_\L=t_n$. An example of  regression tree is shown in Figure \ref{fig:decisiontree} with $ \L = 3 $ and $ t_n = 6$. 

In the  consistency proof, we also need  a truncated tree at each layer $ \ell: 0 \le \ell \le \L$, which is denoted by  $ T_{\D_n,t_n, \ell} $. Sometimes, we  use $ T_{\D_n,\ell}$ to denote  $ T_{\D_n,t_n, \ell} $. For example, tree $ T_{\D_n,t_n,2} $ in Figure \ref{fig:decisiontree} has leaves $\{\mathbb{A}_2^1, \mathbb{A}_2^2, \mathbb{A}_2^3, \mathbb{A}_2^4\}$. Note that leaves and nodes are relative, and a node on a fully grown tree can become a leaf on a truncated tree.


\begin{algorithm}[h]\label{Algorithm.ODTtreereg}
	\caption{Oblique Decision Tree in Regression}
	\KwIn{Data $\D_n$  and pre-specified number of leaves $t_n$}
	Set $\mathcal{P}_0=\{[0,1]^p\}$, the root node of the tree\;

	For all $1\le \ell\le n$, set $\mathcal{P}_\ell=\varnothing$\;
	
	Set  $\text{layer}=0$ and $n_{\text{nodes}}=1$\;
	
	\While{$n_{\text{nodes}} \in  \{1, 2, ..., t_n\}$}{
		
		\eIf{$\mathcal{P}_{\text{layer}}=\varnothing$}
		{
			$\text{layer}=\text{layer}+1$\;
		}{
			Let $A$ be the first node in $\mathcal{P}_{\text{layer}}$\;
			
			\eIf{{\color{black}$A$ contains only one data point}}
			{
				$\mathcal{P}_{\text{layer}}\leftarrow \mathcal{P}_{\text{layer}}- \{A\}$\;
				$\mathcal{P}_{\text{layer}+1}\leftarrow\mathcal{P}_{\text{layer}+1}\cup\{A\}$\;
			}{
				$(\hat{\theta}_A,\hat{s}_A)=\argmax_{\theta\in\Theta^p,s\in\mathbb{R}}{\Delta_A(\theta,s)}$ defined in \eqref{impuritygainreg}\;
				
				Partition node $A$ into two daughter nodes:  
    
                   \ \ \ \  \ \ \ \  \ \ \ \ $A^+_{\hat{\theta}_A,\hat{s}_A}=\{x\in A: \hat{\theta}_A^T\cdot x\le \hat{s}_A\}$ and  $A^-_{\hat{\theta}_A,\hat{s}_A}=\{x\in A: \hat{\theta}_A^T\cdot x> \hat{s}_A\}$\;
				
				$\mathcal{P}_{\text{layer}}\leftarrow \mathcal{P}_{\text{layer}}- \{A\}$\;
				$\mathcal{P}_{\text{layer}+1}\leftarrow\mathcal{P}_{\text{layer}+1}\cup\{A^+_{\hat{\theta}_A,\hat{s}_A}\}\cup\{A^-_{\hat{\theta}_A,\hat{s}_A}\}$\;
				$n_{\text{nodes}}=n_{\text{nodes}}+1$\;
			}
		}
		
	}
	\KwOut{Let $\{\mathbb{A}_\L^j\}_{j=1}^{t_n}$ be the set of leaves of above generated tree $T_{\mathcal{D}_n,t_n}$. Estimate $m(x)$ by $m_{n,\L}(x)=\sum_{j=1}^{t_n}{\mathbb{I}(x\in \mathbb{A}_\L^j)\cdot \bar{Y}_{\mathbb{A}_\L^j}}, x\in[0,1]^p$.}
\end{algorithm}

\section{Consistency of ODT}

Let us briefly describe our idea for the proof of the consistency of ODT, i.e. the consistency of $m_{n,\L}(x)$ as an estimator of $m(x)=\E(Y|X=x), x\in[0,1]^p$. In fact, three pieces of facts are used in the proof. The first fact is the well-known universal approximation theorem; see, for example, \cite{cybenko1989approximation}. Under some mild conditions, function $ m(x) $ can be approximated by a sequence of ridge functions with additive structure, i.e.
\begin{equation}\label{ridgefunctionm}
	m(x) =  \sum_{j=1}^\infty  g_j(\theta_j^T x),
\end{equation}
or for any $\varepsilon>0$, there exist $ j_\varepsilon\in\mathbb{Z}_+$ and
$
m_{\varepsilon}(x) = \sum_{j=1}^{j_\varepsilon}  g_j(\theta_j^T x)
$
so that
$$
\sup_{x\in[0,1]^p}{|m(x)-m_{\varepsilon}(x)|}\leq \varepsilon.
$$
The second fact is that the directions searched in the ODT algorithm can indeed play the role of those in the universal approximation \eqref{ridgefunctionm} or those directions in $    m_{\varepsilon}(x) $.  The third fact is the consistency results of CART or RF for the additive model proved by \cite{klusowski2021universal}. Note that    $
m_{\varepsilon}(x) $ indeed has an additive structure if $\theta_1, \theta_2, ...  $ are known.
Putting these facts together, the consistency result is given below.

\begin{theorem}[Consistency of ODT before pruning]\label{consistencyodtreg}
	Assume 
	$
	\E(e^{c\cdot Y^2})<\infty
	$
	for some  $c>0$,  and that $m(X)$ is $L^2$ integrable. If $t_n\to\infty$ and $t_n  =o\left(\frac{n}{\ln^4{n}}\right)$, we have
	$$
	\E\left(\int{|m_{n,\L}(x)-m(x)|^2d\mu(x)}\right)\to 0,\ \text{as}\ n\to\infty.
	$$
	
\end{theorem}

\begin{remark}\label{remark11}
This theorem shows that ODT is mean-square consistent under very mild conditions. By comparing Theorem \ref{consistencyodtreg} with the consistency of CART given in \cite{scornet2015consistency}, our result has several advantages. First, CART is shown to be consistent only when $m(x)$ has an additive structure of predictors while Theorem \ref{consistencyodtreg} guarantees the consistency of ODT for any smooth functions. Furthermore,  there is no restriction on the distribution of $X$ and we only require $Y$ to be a sub-Gaussian random variable in this theorem.
\end{remark}

We need additional notation in the proof of Theorem \ref{consistencyodtreg}. Let $\beta_n\asymp \ln{n}$ and 
\begin{equation}\label{tr}
    \hat{m}_{n,\L}(x)=\max\{\min\{m_{n,\L}(x),\beta_n\},-\beta_n\}.
\end{equation}
Then,
$\hat{m}_{n,\L}(x) $ is a truncated version of $ m_{n,\L}(x) $.  More generally, let $h:[0,1]^p\to\mathbb{R}$ be a function which is constant on each  $\mathbb{A}_{\L}^j, j=1,\ldots, {t_n}$ defined in Algorithm \ref{Algorithm.ODTtreereg}.
Let $\mathcal{H}_{t_n}$ be the collection of such function $ h $ and ${\mathcal{H}_{t_n}^{\beta_n}}:=\{\max\{\min\{h,\beta_n\},-\beta_n\}: h\in\mathcal{H}_{t_n}\}$. In  order to bound the generalization error,  we need both the growth function and the cover number,  which are explained below.

\begin{definition}[\cite{blumer1989learnability}]
	Let $\mathcal{F}$ be a Boolean function class in which each $f:\mathcal{Z}\to\{-1,1\}$ is binary valued. The growth function of $\mathcal{F}$ is defined by
	$$
	\Pi_\mathcal{F}(m)=\max_{z_1,\ldots,z_m\in \mathcal{Z}} Card(\{(f(z_1),\ldots,f(z_m)):f\in\mathcal{F}\})
	$$
	for each positive integer $m\in\mathbb{Z}_+$.
\end{definition}

\begin{definition}[\cite{gyorfi2006distribution}] \label{coveringnumber}
	Let $z_1,\ldots, z_n \in [0,1]^p$ and $z_1^n=\{z_1,\ldots, z_n\}$. Let $\mathcal{H}$ be a class of functions $h: [0,1]^p\to\mathbb{R}$. An $L_q$ $\varepsilon$-cover of $\mathcal{H}$ on $z_1^n$ is a finite set of functions $h_1,\ldots, h_N: \mathbb{R}^p
	\to \mathbb{R}$ satisfying
	$$
	\min_{1\leq j\leq N}{\left( \frac{1}{n}\sum_{i=1}^n{|h(z_i)-h_j(z_i)|^q}\right)^\frac{1}{q}}<\varepsilon,\ \ \forall h\in \mathcal{H}.
	$$
	Then, the $L_q$ $\varepsilon$-cover number of $\mathcal{H}$ on $z_1^n$, denoted by $\mathcal{N}_q(\varepsilon,\mathcal{H},z_1^n)$, is the minimal size of an $L_q$ $\varepsilon$-cover of $\mathcal{H}$ on $z_1^n$. If there exists no finite $L_q$ $\varepsilon$-cover of $\mathcal{H}$, then the above cover number is defined as $\mathcal{N}_q(\varepsilon,\mathcal{H},z_1^n)=\infty$.
\end{definition}

The proof of Theorem \ref{consistencyodtreg}  relies on Lemma \ref{agirov2009estimation} below, which will be frequently used during the proofs for random forests and boosting trees.

\begin{lemma}[\cite{bagirov2009estimation}]\label{agirov2009estimation}
	Assume
	$
	\E(e^{c\cdot Y^2})<\infty
	$
	for some  $c>0$. Then, the truncated estimator $\hat{m}_{n,\L}(x)$ satisfies
	\begin{equation}\label{Mainlemmaformulabeforeprun}
		\begin{aligned}
			\E_{\mathcal{D}_{n}} \int|\hat{m}_{n,\L}(x)-m(x)|^2d\mu(x)\leq & 2 \E_{\mathcal{D}_{n}}\left( \frac{1}{n}\sum_{i=1}^{n}|m_{n,\L}(X_i)-Y_i|^2-\frac{1}{n}\sum_{i=1}^{n}|m(X_i)-Y_i|^2\right)\\
			&+\frac{c\ln^2 n}{n}\cdot\sup_{z_1^n}\ln\left(\mathcal{N}_1(1/(80n\beta_n),\mathcal{H}_{t_n}^{\beta_n}, z_1^n)\right),
		\end{aligned}
	\end{equation}
	 where $\mathcal{N}_1$ is the cover number and $\mu$ is the distribution of $X$.
\end{lemma}

Next, we  use Lemma \ref{Carteqqw} whose proof idea was originally shown in  \cite{klusowski2021universal} to bound the first term on the RHS of \eqref{Mainlemmaformulabeforeprun}.  

\begin{lemma}\label{Carteqqw}
	Let $R(A):=\|\mathbb{Y}-\bar{Y}_A\|_A^2-\|\mathbb{Y}-g\|_A^2$ for any $g\in \mathcal{U}$, where $\mathcal{U}$ is a set of linear combinations of ridge functions:
	$$
	\mathcal{U}= \{\sum_{j=1}^J{g_j(\theta_j^Tx)}:\theta_j\in\Theta^p, g_j\in TV(\mathbb{R}), J\in\mathbb{Z}_+\}
	$$
	and $TV(\mathbb{R})$ consists of functions defined on $\mathbb{R}$ with bounded total variation. Let $A$ be an internal node  of tree $T_{\mathcal{D}_n,t_n}$
	which contains at least two data points. If $R(A)\ge 0$,  then $\Delta_A(\hat{\theta}_A,\hat{s}_A)$ satisfies
	$$
	\Delta_A(\hat{\theta}_A,\hat{s}_A)\ge \frac{R^2(A)}{\|g\|_{TV}^2}.
	$$
\end{lemma}

\begin{proof}
  The proof is similar to  Lemma \ref{regpart2dd} that is a random version of this lemma.
\end{proof}

In order to bound the cover number in \eqref{Mainlemmaformulabeforeprun}, we introduce the VC dimension below, which is usually a powerful tool to deal with generalization errors in statistical learning. After bounding the VC dimension of ODT class, Lemma \ref{CoveringnumbervsVCd} is employed to find an upper bound for $\mathcal{N}_1(\cdot,\cdot,\cdot)$.

\begin{definition}\label{defvc}
    Let $\mathcal{H}$ be a class of functions $h: [0,1]^p\to \R$. Define its Boolean class by $\mathcal{F}_{\mathcal{H}}:=\left\{\mathbb{I}(f(x)<y)-\mathbb{I}(f(x)\ge y),(x,y)\in [0,1]^p\times \R: f\in \mathcal{H}\right\},$ where $\mathbb{I}(\cdot)$ is the indicator function. The VC dimension of $\mathcal{H}$ is the largest $m\in\mathbb{Z}_+$ satisfying $2^m=\Pi_{\mathcal{F}_{\mathcal{H}}}(m)$ and denoted by $VC(\mathcal{H})$.
\end{definition} 

\begin{lemma}[\cite{gyorfi2006distribution}] \label{CoveringnumbervsVCd}
	Let $\mathcal{H}$ be a class of functions $h: [0,1]^p\to [0,B]$ with finite VC dimension $VC(\mathcal{H})\geq 2$. For any $B/4>\varepsilon>0$, the  cover number in Definition \ref{coveringnumber} satisfies
	$$
	\mathcal{N}_1(\varepsilon,\mathcal{H},z_1^n)\leq 3\left( \frac{2eB}{\varepsilon}\ln\left(\frac{3eB}{\varepsilon} \right) \right)^{VC(\mathcal{H})}
	$$
	for all $z_1^n=\{z_1,\ldots,z_n\}, z_i\in \mathbb{R}^p.$
\end{lemma}

\noindent \textbf{Proof of Theorem \ref{consistencyodtreg}.}
   The  proof outline is as follows. We first show that the truncated estimator $\hat{m}_{n,\L}(x)$ defined in \eqref{tr} is mean square consistent by bounding the two terms on the right hand side (RHS) of Lemma \ref{agirov2009estimation}. Those corresponding details are given in Part I and Part II in the following analysis respectively. Then, we establish the relationship between $\hat{m}_{n,\L}(x)$ and $m_{n,\L}(x)$ in Part III and show the untruncated estimator $m_{n,\L}(x)$ is also consistent.

   \
	
	\textbf{Part I:} We bound the first term on the RHS of \eqref{Mainlemmaformulabeforeprun} in this part. The following function class will be useful in our analysis:
	$$
	    Ridge_{J}=\left\{\sum_{j=1}^{J} {c_j\cdot \sigma(\theta_j^Tx+d_j)}: \theta_j\in\Theta^p, c_j,d_j\in\mathbb{R}, \forall j\ge 1\right\},
	$$
	where $\sigma(v)=e^v/(1+e^v)$ with  $ v\in\mathbb{R}$. 
	
	Recall the definition of $T_{\mathcal{D}_n,\ell}$, a truncated tree, and $m_{n,\ell}(x)$ is its corresponding estimator by taking average of data in each terminal leaf of $T_{\mathcal{D}_n,\ell}$, namely
	$$
	m_{n,\ell}(x)=\sum_{j=1}^{k_\L}{\mathbb{I}(x\in \mathbb{A}_{\ell}^j)\cdot \bar{Y}_{\mathbb{A}_{\ell}^{j}} }.
	$$
	Let $\L_0:= \lfloor\log_2{t_n}\rfloor\le \L$. Then, we know $T_{\D_n,\ell}$ is fully grown for each $0\leq \ell\le \L_0$, i.e. $T_{\D_n,\ell}$ is generated recursively by splitting all leaves of the previous tree $T_{\D_n,\ell-1}$ except for those leaves containing only one data point.

For any given $g\in\mathcal{G}_{J}$ and  $t_n>1$, we aim to prove 
	\begin{equation}\label{regpart1}
		\|\mathbb{Y}-m_{n,\L_0}(X)\|_n^2- \|\mathbb{Y}-g(X)\|_n^2\le \frac{\|g\|^2_{TV}}{\lfloor\log_2{t_n}\rfloor+4},
	\end{equation}
	 where $\|g\|_{TV}$ is the total variation of $g$ shown in Section \ref{sec:notations}. Define the  error of $\ell$-th layer by $R_{\D_n,\ell}: =\|\mathbb{Y}-m_{n,\ell}(X)\|_n^2-\|\mathbb{Y}-g(X)\|_n^2$, where $0\le \ell\le \L_0$.  Without loss of generality, we can  assume $R_{\D_n,\L_0-1}\ge 0$ since  \eqref{regpart1} holds obviously if $R_{\D_n,\L_0-1}< 0$ (note that $R_{\D_n,\L_0}\le R_{\D_n,\L_0-1}$). Similarly, define $R(A):=\|\mathbb{Y}-m_{t_n,\ell}(X)\|_A^2-\|\mathbb{Y}-g(X)\|_A^2$ for each $A\in \mathcal{O}_{\ell,1}:=\{\mathbb{A}_\ell^j\}_{j=1}^{k_\ell}$ that  contains at least one data point. Then, for any $0\le \ell\le \L_0$ we have $$R_{\D_n,\ell}=\sum_{A\in\mathcal{O}_{\ell,1}}{w(A)R(A)},$$
	where $w(A)=Card(A)/n$ is the proportion of data within $A$. Note that
	\begin{align}
		R_{\D_n,\L_0}= &  R_{\D_n,\L_0-1}-\sum_{A\in\mathcal{O}_{\L_0-1,2}}{w(A)\Delta_A(\hat{\theta}_A,\hat{s}_A)} \label{dhasdgasdj}\\
		\leq & R_{\D_n,\L_0-1}-\sum_{A\in\mathcal{O}_{\L_0-1,2}, R(A)> 0}{w(A)\Delta_A(\hat{\theta}_A,\hat{s}_A)} \nonumber\\
		\le & R_{\D_n,\L_0-1}-\frac{1}{\|g\|_{TV}^2}\sum_{A\in\mathcal{O}_{\L_0-1,2}, R(A)> 0}{w(A)R^2(A)},\label{ygdai}
	\end{align}
	where $\mathcal{O}_{\L_0-1,2}\subseteq \mathcal{O}_{\L_0-1,1}$ is a collection of nodes which must contain at least two data points, and \eqref{dhasdgasdj} follows from the definition of impurity gain, and \eqref{ygdai} follows from  Lemma \ref{Carteqqw}. Decompose  $R_{\D_n,\L_0-1}$ into two parts:
	\begin{align*}
		R_{\D_n,\L_0-1}^+ := &\sum_{A\in\mathcal{O}_{\L_0-1,1} R(A)> 0}{w(A)R(A)},\\
		\ R_{\D_n,\L_0-1}^-:=&\sum_{A\in\mathcal{O}_{\L_0-1,1}, R(A)\le 0}{w(A)R(A)}
	\end{align*}
	satisfying $R_{\D_n,\L_0-1}=R_{\D_n,\L_0-1}^++R_{\D_n,\L_0-1}^-$. By Jensen's inequality and  $R(A)\le 0$ for any leaf $A$ of $T_{\D_n,\L_0-1}$ which contains  one data point only, we have
	\begin{equation}\label{dgiasydf}
		\sum_{A\in\mathcal{O}_{\L_0-1,2}, R(A)> 0}{w(A)R^2(A)}\ge  \left(\sum_{A\in\mathcal{O}_{\L_0-1,2}, R(A)> 0}{w(A)R(A)}\right)^2
		=(R_{\D_n,\L_0-1}^+)^2.
	\end{equation}
		
  Similar to the proof of Theorem 4.2 in \cite{klusowski2021universal}, the combination of \eqref{ygdai} and \eqref{dgiasydf} implies that
	\begin{align}
		R_{\D_n,\L_0}&\le  R_{\D_n,\L_0-1}- \frac{1}{\|g\|_{TV}^2}(R_{\D_n,\L_0-1}^+)^2\nonumber\\
		&\le R_{\D_n,\L_0-1}- \frac{1}{\|g\|_{TV}^2} R_{\D_n,\L_0-1}^2, \label{dsfqwe}
	\end{align}
	where the last equation \eqref{dsfqwe} holds because  $R_{\D_n,\L_0-1}^+>R_{\D_n,\L_0-1}$ and $R_{\D_n,\L_0-1}\ge 0$  by  assumption. Again, using the same
	arguments above, \eqref{dsfqwe} also implies
	\begin{equation}\label{asdsa}
		R_{\D_n,\ell}\le R_{\D_n,\ell-1}- \frac{1}{\|g\|_{TV}^2} R_{\D_n,\ell-1}^2
	\end{equation}
	for any integer $1\le \ell \le \L_0$. By mathematical induction, the fact that  $R_{\D_n,1}\le R_{\D_n,0}- \frac{1}{\|g\|_{TV}^2} R_{\D_n,0}^2\le \|g\|_{TV}^2/{4}$ and \eqref{asdsa} imply that \eqref{regpart1} is true.

	According to the greedy scheme of ODT, we have
	$$
	\|\mathbb{Y}-m_{n,\L}(X)\|_n^2\le \|\mathbb{Y}-m_{n,\L_0}(X)\|_n^2.
	$$
	Therefore, the first part on the RHS of \eqref{Mainlemmaformulabeforeprun} satisfies, for any $g\in Ridge_J$,
	\begin{align}
		2 \E_{\mathcal{D}_{n}}\left( \|m_{n,\L}-\mathbb{Y}\|_n^2-\|m-\mathbb{Y}\|_n^2\right)
		\le & 2\E_{\mathcal{D}_{n}}(\|\mathbb{Y}-m_{n,\L}(X)\|_n^2- \|\mathbb{Y}-g(X)\|_n^2)\nonumber\\
		+&2\E_{\mathcal{D}_{n}}(\|\mathbb{Y}-g(X)\|_n^2-\|\mathbb{Y}-m(X)\|_n^2)\nonumber\\
		\le& \frac{2\|g\|^2_{TV}}{\lfloor\log_2{t_n}\rfloor+4}+2\E(g(X)-m(X))^2. \label{jlijl}
	\end{align}
	According to the standard density argument, we assume $m(x),x\in [0,1]^p$ is  continuous only. By \cite{pinkus1999approximation}, there is a series of $g_J\in Ridge_J, J\ge 1$ satisfying
	\begin{equation}\label{iopipo}
		\lim_{J\to\infty}{\max_{x\in [0,1]^p}{|g_J(x)-m(x)|}}=0.
	\end{equation}
	In conclusion, \eqref{jlijl} and \eqref{iopipo} together imply that
	\begin{equation}\label{djosijdisj}
		2 \E_{\mathcal{D}_{n}}\left( \frac{1}{n}\sum_{i=1}^{n}|m_{n,\L}(X_i)-Y_i|^2-\frac{1}{n}\sum_{i=1}^{n}|m(X_i)-Y_i|^2\right)\to 0
	\end{equation}
	as $t_n\to\infty$, which completes the proof of \textbf{Part I}.

    \
	
	\textbf{Part II:}   Now we analyze
    the second part on the RHS of \eqref{Mainlemmaformulabeforeprun} by applying Lemma \ref{CoveringnumbervsVCd}. The first step is to bound the VC dimension of $\mathcal{H}_{t_n}$ that is a class of ODTs having $t_n$ leaves. We first state a lemma, whose proof is provided in the supplementary document \cite{ran2024}.

    \begin{lemma}\label{vclemma} 
        The VC dimension of $\mathcal{H}_{t_n}$, denoted by $VC(\mathcal{H}_{t_n})$, is $O(t_n\ln t_n)$.
    \end{lemma}
    
     Then, by Lemma \ref{CoveringnumbervsVCd} we have
	\begin{align}
		\mathcal{N}_1(1/(80n \beta_n),\mathcal{H}^{\beta_n}_{t_n},z_1^n) \leq & 3\left( \frac{4e \beta_n}{1/(80n \beta_n)}\ln\left(\frac{6e \beta_n}{1/(80n \beta_n)} \right) \right)^{VC({\mathcal{H}^{\beta_n}_{t_n}})}\nonumber \\
		\leq & 3\left( \frac{4e \beta_n}{1/(80n \beta_n)}\ln\left(\frac{6e \beta_n}{1/(80n \beta_n)} \right) \right)^{VC({\mathcal{H}_{t_n}})}\nonumber \\
		\leq & 3\left( \frac{4e \beta_n}{1/(80n \beta_n)}\ln\left(\frac{6e \beta_n}{1/(80n \beta_n)} \right) \right)^{c(p)\cdot t_n \ln(t_n)}\nonumber \\
		\leq & 3\left( 480e n \beta_n^2 \right)^{c(p)\cdot t_n \ln(t_n)}\label{ConclusionPart2}
	\end{align}
	for any $z_1,\ldots,z_n\in\R^p$. Inequality \eqref{ConclusionPart2} implies that the second part on the RHS of \eqref{Mainlemmaformulabeforeprun} satisfies
	\begin{equation}\label{fsfsfz}
		\frac{c\ln^2 n}{n}\cdot\sup_{z_1^n}\ln\left(\mathcal{N}_1(1/(80n\beta_n),\mathcal{H}_{t_n}^{\beta_n}, z_1^n)\right)\to 0,\ \text{as}\ n\to\infty,
	\end{equation}
	when $t_n  =o\left(\frac{n}{\ln^4{n}}\right)$. This completes arguments in \textbf{Part II}.

    \
	
	\textbf{Part III}: By Lemma \ref{agirov2009estimation}, \eqref{djosijdisj} and \eqref{fsfsfz} imply that
	\begin{equation}\label{chnosnolzm}
		IV:=2\E_{\mathcal{D}_{n}} \int|\hat{m}_{n, \L}(x)-m(x)|^2d\mu(x)\to 0.
	\end{equation}
	Finally, we show that \eqref{chnosnolzm} also holds for the un-truncated estimator  $m_{n,\L}(x)$. Note that
	\begin{align}
		\E_{\mathcal{D}_{n}} \int|m_{n, \L}(x)-m(x)|^2d\mu(x)\le & 2\E \int|\hat{m}_{n, \L}(x)-m_{n, \L}(x)|^2d\mu(x)+IV\nonumber\\
		\le & 2\E \int|\hat{m}_{n, \L}(x)-m_{n, \L}(x)|^2\mathbb{I}(|m_{n, \L}(x)|>\beta_n)d\mu(x)+ IV\nonumber\\
		\le & 2\E\left( |\hat{m}_{n, \L}(X)-m_{n, \L}(X)|^2\mathbb{I}(|m_{n, \L}(X)|>\beta_n)\right)+ IV\nonumber\\
		:= &V+IV, \label{fhuifhsdoui}
	\end{align}
      where $X$ is independent to $\D_n$. Recall $\beta_n\asymp \ln n$. We bound $V$ as follows:
 	\begin{align}
		V
		\le &   2\E_{\D_n}{\left[\left(2\beta_n^2+2\max_{1\le i\le n}{|Y_i|^2}\right)\mathbb{I}\left(\max_{1\le i\le n}{|Y_i|}> \beta_n\right)\right]}\nonumber\\
		\le & 4\beta_n^2\cdot \P\left(\max_{1\le i\le n}{|Y_i|}\ge c\cdot\ln{n}\right)+4\left[\E\left(\max_{1\le i\le n}{|Y_i|^4}\right)\cdot \P\left(\max_{1\le i\le n}{|Y_i|}\ge c\cdot\ln{n}\right)\right]^{\frac{1}{2}}.\label{fhoxufp}
	\end{align}
	Since $Y_i,i=1,\ldots,n$ are i.i.d. and share with a common sub-Gaussian distribution, we have
	\begin{align} 
		\P\left(\max_{1\leq i\leq n}|Y_i|> c\cdot\ln{n}\right)=&1- \P\left(\max_{1\leq i\leq n}|Y_i|\leq c\cdot\ln{n}\right)\nonumber\\
		=&1-\left[\P(|Y_1|\leq c\cdot\ln{n})\right]^n
		\leq 1-(1-c\cdot e^{-c\cdot \ln^2{n}})^n\nonumber \\
		=& 1-e^{n\cdot \ln(1-c\cdot e^{-c\cdot \ln^2{n}})} \nonumber\\
		\leq & -n\cdot \ln(1-c\cdot e^{-c\cdot \ln^2{n}})\label{opipo}\\
		\leq & c\cdot n\cdot e^{-c\cdot \ln^2{n}},\label{duhsihrdksh}
	\end{align}
	where $\eqref{opipo}$ is obtained from the fact $1+v\le e^{v}, v\in\mathbb{R}$; and \eqref{duhsihrdksh} is due to the fact $\lim_{v\to 0}{\frac{\ln(1+v)}{v}}=1$. By \eqref{duhsihrdksh} and $\E\left(\max_{1\le i\le n}{Y_i^4}\right)\le n\cdot \E(Y_1^4)$, we have
	\begin{equation}\label{opsdpz}
		V\to 0.
	\end{equation}
	In conclusion, the combination of \eqref{fhuifhsdoui}, \eqref{chnosnolzm} and \eqref{opsdpz} finishes arguments for \textbf{Part III}.
	\hfill\(\Box\)
 \bigskip

\section{Consistency of ODRF}\label{sec.ODRF}

In this section, we study the ODT-based random forest, hereafter called ODRF. For ease of exposition,  several notations are introduced first.  Let $\Omega$ be a class of subsets of $\{1,2,\ldots,p\}$ and $q$ be a random number randomly chosen with equal probability from $\{1,\ldots,p\}$.  Given $q$, we suppose $\S$ is uniformly chosen from $\Omega_q=\{C\subseteq \Omega: Card(C)=q\}$. Thus, the law of $\S$ is defined as follows
\begin{equation}\label{probabilityS}
    \P(\S)=\frac{1}{p}\cdot\frac{1}{\binom{p}{Card(\S)}},\ \ \forall \S\subseteq \Omega.
\end{equation}
 Denote by $\mathcal{S}_\tau,\ \tau=1, 2, \ldots,$ a sequence of independent copies of $\mathcal{S}$, which can be regarded as a (random) sample of $ \S  $. Let $ X_\S $ be a sub-vector of $ X $ consisting of coordinates indexed by elements of $\mathcal{S}$ and $\Xi_{t_n-1}:=(\S_1,\ldots, \S_{t_n-1})$  with  $t_n\geq 2$.

Under these preparations, we first construct a single tree, called \textbf{random ODT}, based on data $\D_n(\mathfrak{G})\subseteq \D_n$ as follows. First,  resample data $\D_n(\mathfrak{G}):=\{(X_i,Y_i): i\in \mathfrak{G}\}$ without replacement from $\D_n$, where $\mathfrak{G}$ is a randomly chosen subset of $\{1,2,\ldots,n\}$ conditional on $Card(\mathfrak{G})=a_n$ and $a_n$ is a predefined size of subsample. 
  Then,  replace \textbf{lines 13-14} in Algorithm \ref{Algorithm.ODTtreereg} with the following steps:
 \begin{itemize}
	\item In the $\tau$-th division ($1\le \tau\le t_n-1$) of node $A$, randomly choose $\S_\tau\in\Omega $ with  $Card(\S_\tau)=q$;
	\item $(\hat{\theta}_{A},\hat{s}_A)=\argmax_{\theta_{\S_\tau}\in\Theta^q,s\in\mathbb{R}}{\Delta_{A,q}({\theta}_{\S_\tau},s)}$, where $\Delta_{A,q}$ is defined in the same way as $\Delta_{A}$ except that  data $\{(X_{i,\S_\tau},Y_i)\}_{i\in \mathfrak{G}}$ are used only in the calculation of $\Delta_{A,q}$;
	\item Partition the node $A$ into $A^+_{\hat{\theta}_A,\hat{s}_A}=\{x\in A: \hat{\theta}_A^T\cdot x_{\S_\tau}\le \hat{s}_A\}$ and  $A^-_{\hat{\theta}_A,\hat{s}_A}=\{x\in A: \hat{\theta}_A^T\cdot x_{\S_\tau}> \hat{s}_A\}$.
\end{itemize}
We call this modified algorithm as Algorithm $r1$ and denote its  output estimator as $m^r_{n,a_n,t_n,\L}(x)$. Meanwhile, the corresponding tree $T_{\D_n(\mathfrak{G}),a_n,t_n,\L}^r$ with $\L$ layers and $t_n$ leaves is named as \textbf{random ODT}. Note that $\L$ is a random variable depending on both data $\{(X_i,Y_i): i\in \mathfrak{G}\}$ and random seeds $\{\S_\tau\}_{\tau=1}^{t_n-1}$.

Next, we introduce the forest estimator based on the bagging technique.
First, we independently construct $B\in\mathbb{Z}_+$ random ODTs and their corresponding estimators are denoted by $m^{r,b}_{n,a_n,t_n,\L}(x), b=1,\ldots,B$. The ODT-based Random Forest (ODRF) in regression is defined by
\begin{equation}\label{ODTRandomforest}
m_{ODRF,B,n}(x):=\frac{1}{B}\sum_{b=1}^{B}{m^{r,b}_{n,a_n,t_n,\L}(x)}.
\end{equation}
It is worth noting that ODRF is very similar to Forest-RC in \cite{breiman2001random}, except that the latter only selects the best splitting plane with fixed $q$ variables from randomly generated hyperplanes.

In theory, we study the consistency of ODRF in two different schemes. In the first one, each tree of the forest is \textit{partially grown}, i.e. we follow a rule to stop growing a tree. Thus, each leaf might contain more than one observation. The estimation error of ODRF is controlled by restricting the divergent rate of $t_n$ and  the proof for the consistency of ODRF is similar to the proof of a single tree. For the second scheme, each tree of ODRF is \textit{fully grown} and each leaf of random ODT contains  one data point only. This makes the proof more difficult and different from the traditional error analysis for learning algorithms. In the second scheme, the key point is to assume $a_n$ is not large, which avoids the exploding of estimation error. We will use  techniques in \cite{scornet2015consistency} in the proof.

\subsection{ODRF with partially  grown trees}

The main theorem of this section is given below, which shows that ODRF is consistent in the first scheme for any general $m(x)$ by controlling its growth speed. For comparison, recall that the classical Random Forest is known to be consistent  for additive models only.

\begin{theorem}[Consistency of ODRF with partially grown trees]\label{Theoremhusidhuioh}
	Assume
	$
	\E(e^{c\cdot Y^2})<\infty,
	$
	where $c>0$ and $m(X)$ is $L^2$ integrable. If $a_n\to\infty, t_n\to\infty$ and $t_n  =o\left(\frac{a_n}{\ln^4{a_n}}\right)$, we have
	$$
	\E_{\mathcal{D}_{n},\Xi_{t_n-1}}\left(\int{|m_{ODRF,B,n}(x)-m(x)|^2d\mu(x)}\right)\to 0,\ \text{as}\ n\to\infty.
	$$
\end{theorem}

\noindent \textbf{Proof sketch.}
Without loss of generality, we can assume $Card(\mathfrak{G})=n$ because there is no difference between $a_n\to\infty$ and $n\to\infty$ in this first scheme. Firstly, we  prove that the truncated estimator $\hat{m}_{n,\L}^r(x)=\max{\{\min{\{m_{n,\L}^r(x),\beta_n\}},-\beta_n\}}$, where $\beta_n\asymp \ln{n}$, is mean square consistent. Secondly, we obtain the consistency of untruncated estimator $m_{n,\L}^r(x)$ based on the first step.

The first step can be done by using a variant of  Lemma \ref{agirov2009estimation}:
	\begin{equation}\label{Mainlemmaformula1}
		\begin{aligned}
			\E_{\mathcal{D}_{n},\Xi_{t_n-1}} \int|\hat{m}^r_{n,\L}(x)-m(x)|^2d\mu(x)\leq & 2\cdot \E_{\mathcal{D}_{n},\Xi_{t_n-1}}\left( \|m_{n,\L}^r(X)-\mathbb{Y}\|_n^2-\|m(X)-\mathbb{Y}\|_n^2\right)\\
			&+\frac{c\ln^2 n}{n}\cdot\sup_{z_1^n}\ln\left(\mathcal{N}_1(1/(80n\beta_n),\mathcal{H}_{t_n}^{\beta_n}, z_1^n)\right),
		\end{aligned}
	\end{equation}
where ${\mathcal{H}_{t_n}^{\beta_n}}$ is the truncated ODT class defined below Remark \ref{remark11}; $\mathcal{N}_1$ is the cover number shown in Definition \ref{coveringnumber};  and $\mu$ is the distribution of $X$. We use similar ways in Theorem \ref{consistencyodtreg} to bound the two terms on the RHS of \eqref{Mainlemmaformula1}.  The way to bound  the cover number is not difficult and can be done using Lemma \ref{CoveringnumbervsVCd} \& \ref{vclemma}. To bound the first part, we need  a key inequality that is similar to \eqref{jlijl}:
	\begin{equation}\label{regpart2dd}
		\E_{\Xi_{t_n-1}}\left(\|\mathbb{Y}-m_{n,\L}^r(X)\|_n^2- \|\mathbb{Y}-g_p(X)\|_n^2\right)\le c(p)\cdot \frac{\|g_p\|^2_{TV}}{\lfloor\log_2{t_n}\rfloor+4},
	\end{equation}
	where $t_n>1$ and  $g_p\in\mathcal{G}_{1,p,J}$ is presented in Lemma \ref{LemmaODTrandom}. Following the approach used in Theorem \ref{consistencyodtreg}, equation \eqref{regpart2dd} relies on Lemma \ref{LemmaODTrandom}, which is a randomized version of Lemma \ref{Carteqqw}. The proof of Lemma \ref{LemmaODTrandom} is provided in \cite{ran2024}.

{\red{
		\begin{lemma}\label{LemmaODTrandom}
			Let $A$ be any internal node of random ODT, $T_{\D_n,\L}^r$, in the $\tau$-th partition. Define $R(A):=\|\mathbb{Y}-\bar{Y}_A\|_A^2-\|\bar{Y}-g_q\|^2$ for any $g_q(v)=\sum_{k=1}^V{g_{k,q}(v)}$, where $ g_{k,q}(x)\in Ridge_{k,q,J}$ with
			$$
			Ridge_{k,q,J}= \Bigg\{\sum_{j=1}^{J} {c_{k,q,j}\cdot \sigma(\theta_{\C_{k,q,j}}^Tx+d_{k,q,j})}: \C_{k,q,j}\in \Omega _q, \theta_{\C_{k,q}}\in\Theta^{q}, c_{k,q,j},d_{k,q,j}\in\mathbb{R}\\
			 \Bigg\},
			$$
			and $\sigma(v)=e^v/(1+e^v), v\in\mathbb{R}$.  If $R(A)\ge 0$ given $q=q_0$, then we have
			$$
			\E_{\S_\tau}\left({\Delta_{A,q}(\hat{\theta}_{A},\hat{s}_A)}|\D_n,q=q_0\right)\geq c(p,q_0)\cdot \frac{R^2(A)}{\|g_{q_0}\|_{TV}^2\cdot V},
			$$
			where the expectation is taken  over the random seed $\S_\tau$ and $c(p,q_0)>0$.
		\end{lemma}
}}

    Lemma \ref{LemmaODTrandom}  tells us the error gap between two adjacent layers, which can be used to establish a inequality  related to training errors of these two layers. Based on this recursive inequality, we can find an upper bound as shown in \eqref{regpart2dd}.

    Finally, balance those upper bounds of the two terms on the RHS of \eqref{Mainlemmaformula1}. The proof will be completed by properly selecting the divergence rate of $t_n$. 	\hfill\(\Box\)

\bigskip

In Theorem \ref{Theoremhusidhuioh}, the number of features in each node splitting is randomly selected from $1$ to $p$. Next, we introduce a simplified ODRF with fixed $q$ (ODRF$_q$), namely   each $\S_\tau$ is only uniformly chosen from $\Omega_q$ and each corresponding tree is  called  random ODT$_q$. In this case, we show in Corollary \ref{Corollary1ds} that the estimator $m_{ODRF_q,n}$ of ODRF$_q$  is statistically consistent for an extended additive model:
\begin{equation}\label{Extendedmodel1}
	m(x)=\sum_{\tau=1}^V{m_\tau(x_{\C_\tau})}, x\in [0,1]^p,
\end{equation}
where $x_{\C_\tau}$ is defined similarly to $x_{\A_\tau}$  and each $m_\tau$ is a function w.r.t. $q$ variables $x_{\C_\tau}$ satisfying $ Card(\C_\tau) = q $. Note that $\C_\tau, 1\le \tau\le V$ are fixed indexes and the maximum of $ V $ in model \eqref{Extendedmodel1} is $ \binom{p}{q} $.

\begin{corollary}[Consistency of ODRF$_q$]\label{Corollary1ds}
	Assume
	$
	\E(e^{c\cdot Y^2})<\infty,
	$
	where $c>0$ and $m(x)$ follows model  \eqref{Extendedmodel1} with each continuous $m_\tau(x), 1\le\tau\le V$. If $a_n\to\infty$, $t_n\to\infty$ and $t_n  =o\left(\frac{a_n}{\ln^4{a_n}}\right)$, we have
	$$
	\E_{\mathcal{D}_{n},\Xi_{t_n-1}}\left(\int{|m_{ODRF_q,B,n}(x)-m(x)|^2d\mu(x)}\right)\to 0,\ \text{as}\ n\to\infty.
	$$
\end{corollary}


\vspace{0.1cm}

We summarize  Theorem \ref{Theoremhusidhuioh} and Corollary \ref{Corollary1ds} as follows. If the subset size of the features $ q $ is randomly selected from $1$ to $p$ in each splitting,  the estimator is consistent for any $L^2$ integrable  regression function; if $ q $ is fixed, the estimator is consistent when the underlying regression function follows a projection pursuit structure. As a special case, $q= 1$ corresponds to the traditional RF, and we thus prove that this estimator is consistent for additive models.

\subsection{ODRF with fully grown trees}

The main characteristic of  ODRF  with fully grown trees is that each leaf  contains one data point exactly. This makes it difficult to bound its generalization error or variance.  Employing the notation in \cite{scornet2015consistency}, $a_n$ is the number of bootstrapped data points used to construct each tree and the growth of the $b$-th tree  is determined with a random seed $\Theta_b$.  In fact, $\Theta_b\ (b=1,\ldots,B)$ is designed to resample $a_n$ data points in the construction of the $b$-th tree and select $q$ variables in each node splitting of that tree. Here, we suppose $\Theta_1,\ldots,\Theta_B$ are independent sharing with the same distribution. To save space, we only  consider the first way for feature bagging, where $q$ is random and \eqref{probabilityS} is applied.

To introduce main assumptions  in this section, several notations are required. Let $\Theta$ be the population version of $\Theta_1,\ldots,\Theta_B$. Let $Z_i=\mathbb{I}_{X \Theta X_i}$ be the indicator that both $X_i$ and $X$ are in the same partition in the random ODT generated by $\D_n$ and $\Theta$. Following the same way, we define $Z_i'=\mathbb{I}_{X \Theta' X_i}$, where $\Theta'$ is an independent copy of $\Theta$. We also need notations
$$
\psi_{i,j}(Y_i,Y_j):=\E(Z_i Z_j'|X,\Theta,\Theta',X_1,\ldots,X_n,Y_i,Y_j)
$$
and
$$
\psi_{i,j}:=\E(Z_i Z_j'|X,\Theta,\Theta',X_1,\ldots,X_n).
$$
Finally, for any random variables $W_1,W_2,Z$, notation $Corr(W_1, W_2|Z)$ denotes the conditional correlation coefficient. Similar to \cite{scornet2015consistency}, we use the following assumptions  in the consistency analysis  of ODRF with fully grown trees. Define $Z_{i,j}:=(Z_i,Z_j')$ and assume that one of the two assumptions in the following is satisfied.

\begin{enumerate}
    \item[A.1]  There exist $0<\delta<1$ and $c>0$ such that,
    $$
     \E^{\frac{1}{2}}(\max_{\substack{i,j\\i\neq j}}{|\psi_{i,j}(Y_i,Y_j)-\psi_{i,j}|^2})\le c\left(\frac{1}{a_n}\right)^{2+\delta}.
    $$
    \item[A.2] There are a constant $c>0$ and a sequence $\gamma_n\to 0$ such that,
    $$
       \max_{\iota_1,\iota_2=0,1}{\frac{|Corr(Y_i-m(X_i),\mathbb{I}_{Z_{i,j}=(\iota_1,\iota_2)}|X_i,X_j,Y_j)|}{\P^\frac{1}{2}(Z_{i,j}=(\iota_1,\iota_2)|X_i,X_j,Y_j)}}\le \gamma_n
    $$
    and
    $$
        \max_{\iota_1=0,1}{\frac{|Corr((Y_i-m(X_i))^2,\mathbb{I}_{Z_{i}=\iota_1}|X_i)|}{\P^\frac{1}{2}(Z_{i}=\iota_1|X_i)}}\le c.
    $$
\end{enumerate}
Let us make some remarks on two assumptions above. First, the Assumption A.2 is similar to (H2.2) in \cite{scornet2015consistency}. Second, the difference between assumption A.1 and (H2.1) in \cite{scornet2015consistency} is that we further assume the convergence rate for $\max_{\substack{i\neq j}}{|\psi_{i,j}(Y_i,Y_j)-\psi_{i,j}|}$ measuring the connection of $Y_i$ and $Y_j$. This change is due to the observation that ODRF has more complicated partitions than those of RF in \cite{breiman2001random}. Therefore, a slightly stronger assumption is required when dealing with the estimation error for ODRF with fully grown  trees. Especially in the case where partitions are independent of $Y_1,Y_2,\ldots,Y_n$, our assumption A.1 is satisfied. Although A.1 may seem overly strong, it is the most lenient assumption we could find to meet the technical requirements.

In order to analyze the consistency of ODRF, the condition $B\to\infty$ is necessary in this second scheme. If $B$ is constant as $n\to\infty$, the consistency is hard to be guaranteed. For example, we consider the case for  $B=1$, which is just a single tree with full growth. However, we show that $t_n=o(a_n)$ is an important condition for the consistency of random ODT in the last section. Meanwhile, standard consistency conditions for partitioning estimates also require that the number of cut regions goes to infinity as $n\to\infty$; see Theorem 4.2 in \cite{gyorfi2006distribution}.  Therefore, in this section, we allow $B$ to depend on $n$, written as $B_n$. 

Finally, this second scheme is more interesting because we will show that the ensemble of $B_n$ inconsistent trees is consistent for any continuous $m(x),x\in [0,1]^p$.

\begin{theorem}[Consistency of ODRF with fully  grown trees]\label{full tree consistency}
    Assume that $X$ follows the uniform distribution in $[0,1]^p$ and $\varepsilon=Y-m(X)$ follows standard normal distribution and $X,\varepsilon$ are independent. Suppose $m(x),x\in [0,1]^p$ is continuous and assumption A.1 or A.2 holds. If $t_n=a_n$, $t_n\to\infty$, $a_n\ln n/n\to 0$ and $B_n/\ln n\to\infty$, then
    $$
    \lim_{n\to\infty} \E(m_{ODRF,B_n,n}(X)-m(X))^2=0.
    $$
\end{theorem}

\noindent \textbf{Proof sketch.}
The proof for Theorem \ref{full tree consistency} is based on the  inequality below:
\begin{align*}
    \E(m_{ODRF,B_n, n}(X)-m(X))^2 
     & \le 2\E\left(m_{ODRF,B_n,n}(X)-\E_{\Theta}(m^{r}_{a_n}(X))\right)^2\\
     &+ 2\E\left(\E_{\Theta}(m^{r}_{a_n}(X))-m(X)\right)^2
    := 2(Part_{n,1}+Part_{n,2}),
\end{align*}
where $m^{r}_{a_n}(X):=m^{r}_{n,a_n,a_n,\L}(X)$ is the abbreviation.

In fact, it is not difficult to bound the first part $Part_{n,1}$. The main point of this part is to  approximate the mean by using the empirical average.

Therefore, it is sufficient to prove the $L^2$ consistency for $\E_{\Theta}(m^{r}_{a_n}(X))$, namely $Part_{n,2}\to 0$.  This part involves much more technical analysis and we will follow the route of proof for Theorem 2 in  \cite{scornet2015consistency}. During the proof for $Part_{n,2}\to 0$, the following Lemma  \ref{key general function} plays a very important role. 

Before presenting Lemma  \ref{key general function}, we need to introduce some notations. For any node $A$, define the theoretical CART rule by
\begin{align}
\Delta_A^*(\theta,s)&= Var(Y|X\in A )- \P(\theta^TX<s|X\in A)Var(Y|\theta^TX<s,X\in A)\nonumber\\
&-\P(\theta^TX\ge s|X\in A)Var(Y|\theta^TX\ge s,X\in A). \label{jksf!2!2}
\end{align}
According to the strong law of large numbers, we know $\Delta_A(\theta,s)$, which actually depends on $n$, almost surely converges to $\Delta_A^*(\theta,s)$ as $n\to\infty$. Now we construct a theoretical random oblique decision tree $T^{*,r}(\Theta)$ by using $\Delta_A^*(\theta,s)$. Instead of applying the empirical rule $\Delta_A(\theta,s)$, each node $A$ of $T^{*,r}(\Theta)$ in this case is split based on the  cut
$$
  (\theta_A^*, s_A^*)\in \argmax_{\theta_{A(\Theta)},s} {\Delta_A^*(\theta_{A(\Theta)},s)},
$$
where $A(\Theta)\subseteq\{1,\ldots,p\}$ are chosen indices/features that are related to $\Theta$ and  used in the partition of $A$. 

Now we prove that the population mean in a theoretical node $A^*_{final}\subseteq [0,1]^p$  equals to $m(x), \forall x\in A^*_{final}$
when  $A^*_{final}$ no longer needs to be divided  according to this theoretical CART rule.

\begin{lemma}\label{key general function}
    Let $ \bm{A}_*\subseteq [0,1]^p$ be a  node satisfying $\Delta_{\bm{A}_*}(\theta,s)\equiv 0$ for any $\theta$ and $s$. Then, the regression
function $m(x)=\E(Y|X=x),x\in [0,1]^p$ is constant on $\bm{A}_*$.
\end{lemma}

The proof of Lemma \ref{key general function} is provided in \cite{ran2024}.  Lemma \ref{key general function} sheds light on the reason why ODRF is consistent for general regression functions. This result guarantees the training error goes to zero as long as there are sufficient cuts.

On the other hand, we use the condition $a_n=o(n)$ to bound the generalization error of ODRF. The proof idea of this part comes from the variance bound in \cite{scornet2015consistency}. \hfill\(\Box\)

\section{Consistency of gradient boosting tree and its bagging}\label{sec.Consistency of the  boosting tree}
{
The algorithm of gradient boosting tree was first developed by \cite{friedman2001greedy}, and has been popularly used in machine learning. For simplicity, we call it the boosting tree later. In a boosting process, trees are constructed in a sequential manner, where the $k-$th tree is trained by using predictors and residuals from previous trees. Then, the estimator of $m(x)$ in step $k$ is a linear combination of the first $k$ trees. We refer to \cite{friedman2001greedy} for more details about the boosting tree. In this section, we will use ODT as the basic tree model (also known as the base learner) in the boosting process. To improve the performance of a boosting tree, we apply the feature bagging later in the same way as the random forest; see also in Section \ref{sec.ODRF}. Our final estimator of $m(x)$, which is called the ensemble of ODT-based boosting trees and denoted by ODBT, is the average of many boosting trees.


Motivated by the spirit in \cite{friedman2001greedy}, the ODT-based boosting tree is constructed as follows.  
 




    

	\begin{algorithm}[H]\label{Algorithm.BODT}
		\caption{ODT-based boosting tree }
		\KwIn{Data $\D_n=\{(X_i,Y_i)\}_{i=1}^n$, pre-specified number of leaves $t$, and maximum number of iterations $k$.}
		
		Initialize the number of iteration $j=1$, data $ \D_n^1=\D_n$.
		
		\For{$j \in  \{1, 2,\cdots, k\}$}{

            Obtain the tree estimator $ m_{n, t,\L_j} $ by using Algorithm \ref{Algorithm.ODTtreereg} with data $\D_n^j$;
            
			Estimate  $m(x)$ in the step $j$  by
			\begin{equation}\label{jhajkhdjkDSTEP323}
				m_{j,boost}(x):= \sum_{\ell=1}^j{a_{j,\ell}^*m_{n, t,\L_\ell}(x)},
			\end{equation}
			where
			\begin{equation}\label{jhajkhdjkDSTEP3}
				(a_{j,1}^*,a_{j,2}^*,\ldots,a_{j,j}^*):=arg\min_{\substack{a_{j,\ell}\in\R\\\ell=1,\ldots,j}}{\sum_{i=1}^n{\left(Y_i- \sum_{\ell=1}^j{a_{j,\ell}m_{n, t,\L_\ell}(X_i)}\right)^2}};
			\end{equation}


           Calculate residual $r_{j,i}=Y_i-m_{j,boost}(X_i)$ and update data $\D_n^{j+1}=\{(X_i,r_{j,i})\}_{i=1}^n$;

           Update $j\leftarrow j+1$;
		}
		\KwOut{The boosting estimator  $m_{k,boost}(x), x\in [0,1]^p$.}
	\end{algorithm}

Let us give some explanations of Algorithm \ref{Algorithm.BODT}. Usually, the generic boosting algorithm consists of two steps, including the gradient step and the multiplier choice step; see \cite{friedman2001greedy}. Line 3 in Algorithm \ref{Algorithm.BODT} is to choose the best (gradient) tree  during each loop, which totally coincides with the gradient step in \cite{friedman2001greedy}. However, we make an improvement in the second step. Instead of determining the single step length $a_{j,j}^*$ in original boosting algorithm by 
$$
 a_{j,j}^*:=\argmin_{a\in\R}{ \sum_{i=1}^n(r_{j-1,i}-a\cdot m_{n,t,\L_j}(X_i))^2},
$$
we optimize previous coefficients $(a_{j,1}^*,a_{j,2}^*,\ldots,a_{j,j}^*)$ in \eqref{jhajkhdjkDSTEP3} together. This technique can improve our training process, which was called orthogonal matching pursuit and originally proposed by \cite{pati1993orthogonal}.

 Next, we use bagging to improve the efficiency of above boosting tree,  called  \textbf{ODBT}, by  replacing ODT with \textbf{random ODT} (see its definition in Section \ref{sec.ODRF}).  To simplify notations, we set $a_n=n$ here and thus the full data $\D_n$ is used in this boosting process. 
It is worth noting that during each construction of  random ODT, the choice of random seed $\Xi_{t-1}$ (see the beginning of Section \ref{sec.ODRF}) is different and independent. With a slight abuse of notation, $ m_{k,boost}(x)$ in \eqref{jhajkhdjkDSTEP323} is also used to denote the estimator obtained by running Algorithm \ref{Algorithm.BODT} with random ODT. After running $B$ times of this modified  Algorithm \ref{Algorithm.BODT} independently,  the ensemble estimator (\textbf{ODBT}) is given by
    \begin{equation}\label{boostingtruncatedone}
    m^{ens}_{k,boost}(x):= \frac{1}{B}\sum_{b=1}^B m^{b}_{k,boost}(x),
    \end{equation}
where $m^{b}_{k,boost}(x)$ is the estimator of the $b$-th round. In  following theoretical analysis, the truncated version of \eqref{boostingtruncatedone}:
$$\hat{m}^{ens}_{k_n,boost}(x):=\frac{1}{B} \sum_{b=1}^B\max{\{\min{\{m^{b}_{k_n,boost}(x),\beta_n\}},-\beta_n\}}$$ is considered, where $\beta_n>0$ denotes the threshold. In  asymptotical analysis, we require $k$ depends on $n$ and $k$ is replaced by $k_n$ in above expression.

    \begin{theorem}[Consistency of the ensemble of boosting trees]\label{Boosting tree consi}
        Assume 
	$
	\E(e^{c\cdot Y^2})<\infty
	$
	for some  $c>0$  and that $m(X)$ is  $L^2$ integrable. If 
 $\beta_n \asymp\ln n$, $k_n\to\infty$ and $k_n  =o\left(\frac{n}{\ln^4{n}}\right)$, for any fixed $t\ge 2$ we have
	$$
	\E\left(\int{|\hat{m}^{ens}_{k_n,boost}(x)-m(x)|^2d\mu(x)}\right)\to 0,\ \text{as}\ n\to\infty.
	$$ 
    \end{theorem}

    The proof of Theorem \ref{Boosting tree consi} is based on the key proposition below. 
    
    \begin{proposition}\label{proposition1}
        The boosting tree  
        with $k$ ODTs or $k$ random ODTs (namely the output of Algorithm \ref{Algorithm.BODT} or its modified version) equals to a neural network with the Heaviside activation $\sigma_0(v)=\mathbb{I}(v\ge 0), v\in\R$. Meanwhile, this neural network has three layers with $(p+1)t^2k$ neurons in the first hidden layer and  $t(t+1)k$ neurons in the second hidden layer and $tk$ neurons in the final layer.
    \end{proposition}


The proof of Proposition \ref{proposition1} is provided in \cite{ran2024}. This proposition is significant in its own right, as it demonstrates that a boosting tree can be interpreted as a neural network. Consequently, techniques from neural network analysis can be leveraged to study the boosting tree.  In recent works,  the \textit{Conclusion and Future Work} in \cite{cattaneo2022convergence} suggests that it is likely to study multi-layer networks by ODT. Here, we do not follow their proof but consider the converse way. Actually, we will show that theories about trees can be indeed studied by using the properties of neural networks.  

 \
 
 \noindent \textbf{Proof sketch of Theorem \ref{Boosting tree consi}.} Following Jensen's inequality, it is sufficient to prove the consistency for a single boosting tree. Here, we employ Lemma   \ref{agirov2009estimation} once again to show the consistency for $\hat{m}^1_{k_n,boost}$. We first bound the second term on the RHS of \eqref{Mainlemmaformulabeforeprun},  which is related to the generalization error of $\hat{m}^1_{k_n,boost}$, by using Proposition \ref{proposition1}. Next techniques in neural networks are applied again to bound the first term that has relationship with the training error of $\hat{m}^1_{k_n,boost}$. The proof will be completed after balancing these two parts and choosing a proper sequence of $k_n$.
 
 \

    \textit{Generalization error of $\hat{m}^1_{k_n,boost}$.} Let $\mathcal{NN}_{3}$ be the neural network class indicated in Proposition \ref{proposition1} and its truncated version be 
    $$
            \mathcal{NN}_{3}^{\beta_n}:=\left\{\max{\{\min{\{f,\beta_n\}},-\beta_n\}}: f\in  \mathcal{NN}_{3} \right\}.
    $$
    Note that $\mathcal{NN}_{3}$ has $W=k_n((p+2)t^2+2t)$ parameters. Following the  result in the section of \textbf{Piecewise constant} that is given on page 5 of \cite{bartlett2019nearly}, the VC dimension of $\mathcal{NN}_3$ satisfies
    $$
          VC(\mathcal{NN}_3)=O(W\ln W),
    $$
    where $VC(\cdot)$ denotes the VC dimension; see Definition \ref{defvc}. Since both $p$ and $t$ are fixed, we also have \begin{equation}\label{ajdhnasjdjBerr}
       VC(\mathcal{NN}_3^{\beta_n})\le VC(\mathcal{NN}_3)=O(k_n\ln k_n).
    \end{equation}
    Following similar analysis in \eqref{ConclusionPart2}, \eqref{ajdhnasjdjBerr} leads that
    \begin{equation}\label{aihkdjbmaBGKOLqqb}
        \mathcal{N}_1(1/(80n \beta_n),\mathcal{NN}_3^{\beta_n},z_1^n)\le 3\left( 480e n \beta_n^2 \right)^{c(p,t)\cdot k_n \ln(k_n)}
    \end{equation}
   for any  $z_1^n=\{z_1,\ldots, z_n\}$ with $z_1,\ldots,z_n\in\R^p$. This completes the arguments in this part. 
   
   \

   \textit{Training error of $\hat{m}^1_{k_n,boost}$.} In the second part, we bound the training error of $\hat{m}^1_{k_n,boost}$ using techniques in neural networks again. Here, we aim to prove an oracle inequality of $m_{k_n,boost}(x),x\in [0,1]^p$ constructed by random ODTs.   Before presenting this result in Lemma \ref{Oracle inequality for boosting tree}, we introduce a class of shallow neural networks:
   \begin{equation}\label{sdhbfcjhjkdbGGGG}
       \mathcal{G}:=\left\{  \sum_{j=1}^\infty a_{j}\sigma_0(\theta_j^\top x+s_j): \sum_{j=1}^\infty{|a_j|}<\infty, \theta_j\in\R^p, s_j\in\R \right\},
   \end{equation}
   where $\sigma_0(v)=\mathbb{I}(v\ge 0), v\in\R$ is the Heaviside activation. 
   The details of the above arguments and the following lemma are provided in \cite{ran2024}.

\begin{lemma}[Oracle inequality for a  boosting tree]\label{Oracle inequality for boosting tree} Let $m_{k_n,boost}$ be a single boosting tree with $k_n$ random ODTs. For any $h(x)=\sum_{j=1}^\infty a_{j}\sigma_0(\theta_j^\top x+s_j)\in \mathcal{G}$, we have 
       $$
        \E\left(\|\mathbb{Y}-m_{k_n,boost}\|_n^2|\D_n\right)\leq \|\mathbb{Y}-h\|_n^2 + c(p) \left(\sum_{j=1}^\infty{|a_j|}\right)^2\cdot\frac{1}{k_n+1},
     $$
     where $\mathbb{Y}:=(Y_1,Y_2,\ldots,Y_n)^\top\in\R^n$  and the constant $c(p)>0$  depends on $p$ only.
   \end{lemma}
   
\

If the regression function $m(x)$ belongs to  $\mathcal{G}$ defined in \eqref{sdhbfcjhjkdbGGGG}, we can further show that the ensemble boosting tree has a fast convergence rate $\ln^4 n/\sqrt{n}$ as follows.
     
    \begin{theorem}[Fast consistency rate of ODBT]\label{Fast Boosting tree consi}
        Assume 
	$
	\E(e^{c\cdot Y^2})<\infty
	$
	for some  $c>0$  and that $m(x)\in\mathcal{G}$ defined in \eqref{sdhbfcjhjkdbGGGG}.  If 
   $\beta_n\asymp\ln n$ and  $k_n\asymp \sqrt{n}$, for any  fixed $t\ge 2$ we have
	$$
	\E\left(\int{|\hat{m}^{ens}_{k_n,boost}(x)-m(x)|^2d\mu(x)}\right)=O\left(\frac{\ln^4 n}{\sqrt{n}}\right).
	$$ 
  \end{theorem}

The proof of Theorem \ref{Fast Boosting tree consi}  is similar to that for Theorem \ref{Boosting tree consi} but with no need to approximate $m(x)$ by using functions in $\mathcal{G}$. 

  \begin{remark} 
In fact, the function class $\mathcal{G}$ in Theorem \ref{Fast Boosting tree consi} is quite large because the closure of $ \mathcal{G}$ is equal to the $L^2$ space, $\{m(x):\E(m^2(X))<\infty\}$, by the universal approximation theorem of neural networks. 
  \end{remark}

In Theorem \ref{Fast Boosting tree consi}, we obtain the consistency  rate $\ln^4 n/\sqrt{n}$ for the boosting tree provided that $m(x)\in\mathcal{G}$. Next, we compare Theorem \ref{Fast Boosting tree consi} with Theorem 3.2 in \cite{cattaneo2022convergence}, which gives a fast consistency rate of ODT. These comparisons are in three-fold. Firstly, their rate  $n^{-2/(2+q)}$  is  slower than our rate $\ln^4 n/\sqrt{n}$, where $q>2$ is not clearly specified in their Assumption 3. Secondly, our rate $\ln^4 n/\sqrt{n}$ is free from the curse of dimensionality, while  \cite{cattaneo2022convergence} cannot show that $q$ does not depend on the dimension. Thirdly, in order to get fast consistency rate for ODT, two additional technical conditions, namely Assumption 3 $\&$ 4,  are required in \cite{cattaneo2022convergence} and  hard to verify. On the other hand,   the class of functions considered in their paper is slightly larger than ours when the summation in the space $\mathcal{G}$ above is finite,
but they are otherwise the same since the closure of each one equals the $L^2$ space.  According to these comparisons, it is highly possible that the boosting tree based on ODT is more efficient than the original ODT. Meanwhile, numerous real data performances in Section \ref{secReal} also show the superiority of our ensemble boosting tree over many existing tree methods. 

 Finally, we end this section by giving a lower bound of boosting tree. Consider the normalized version of $\mathcal{G}$:
 $$
\mathcal{G}_1:=\left\{  \sum_{j=1}^\infty a_{j}\sigma_0(\theta_j^\top x+s_j), x\in [0,1]^p: \sum_{j=1}^\infty{|a_j|}\le 1, \theta_j\in\R^p, s_j\in\R \right\}.
$$
From Theorem \ref{Fast Boosting tree consi}, we know the upper bound of our boosting tree satisfies 
\begin{equation}\label{GJYHGVBJGBbdasjbsdQ1}
\sup_{m(x)\in\mathcal{G}_1}\E\left(\int{|\hat{m}^{ens}_{k_n,boost}(x)-m(x)|^2d\mu(x)}\right)\le c\cdot \frac{\ln^4 n}{\sqrt{n}} 
\end{equation}
In fact, the following result shows that the consistency rate in \eqref{GJYHGVBJGBbdasjbsdQ1} is nearly minimax optimal  when $p$ is large,  whose proof is provided in \cite{ran2024}.

\begin{theorem}[Lower bound of ODBT]\label{theorem_lowerbound}
    Suppose $X\sim U[0,1]^p$ and $\varepsilon=Y-m(X)\sim N(0,1)$. There exists a constant $c > 0$ such that
    $$     \inf_{\hat{m}}\sup_{m(x)\in\mathcal{G}_1}\E\left(\int{|\hat{m}(x)-m(x)|^2dx }\right)\ge c\cdot\left(\frac{1}{n}\right)^{\frac{2p+2}{4p+2}},
    $$
    where the infimum is taken over all estimators.
\end{theorem}

\section{Numerical performance in real data} \label{secReal}


Note that the main difficulty in implementing ODT or ODRF is the estimation of the coefficient, $ \theta $, for the linear combinations, which is also one of the main differences amongst all the existing packages. The estimation methods of $ \theta $ include random projection, logistic regression, dimension reduction and many others. However, our experiments suggest that these estimations actually make little difference in the results. In our calculation, instead of using one single projection or linear combination, we provide a number of $ \theta $'s, each of which is for the projection of a set of randomly selected $ q $ predictors, and then use Gini impurity or residuals sum of squares to choose one combination as splitting variable and splitting point. This is similar to \cite{menze2011oblique}, where a number of random projections are provided from which one is selected. Because of the low estimation efficiency of $ \theta $ as dimension $ q $ increases, we select  $q $ randomly from 1 to $min([n^{0.5}], p) $ when splitting each node. This selection of $ q $ satisfies the requirements of Theorem \ref{Theoremhusidhuioh} or Theorem \ref{full tree consistency} for the consistency. Our ODT and ODRF are implemented using our "ODRF" package in R  via link \url{https://cran.r-project.org/web/packages/ODRF} and  codes for ODBT are publicly accessible at \url{https://github.com/liuyu-star/ODBT}.  In our calculation, the logistic regression function is used to find $ \theta $ for each combination of $ q $ predictors, but other alternatives are also provided in the package. We also scale the predictors individually before the computation, and while this makes no difference in theory, it sometimes makes the calculation more stable.

Our ODRF and ODBT are compared with the following methods or packages: the Random Rotation Random Forest (RotRF) of \cite{blaser2016random} which randomly rotates the data prior to inducing each tree,  the Sparse Projection Oblique Random Forests (SPORF) of \cite{tomita2020sparse} which simply uses the random projection method and other methods that also use linear combinations as splitting variables including the method of \cite{silva2021projection}, denoted by \texttt{PPF},  and the method of  \cite{menze2011oblique}, denoted by \texttt{ORF}. The comparison is also made with three axis-aligned popular methods, including Random Forest (RF) of \cite{breiman2001random}, Generalized Random Forest (GRF) of \cite{athey2019generalized} and Reinforcement Learning Trees (RLT) of \cite{zhu2015reinforcement}, and three popular boosting methods, including the extreme gradient boosting (XGB) of \cite{chen2016xgboost}, the Boosted Regression Forest (BRF) of \cite{athey2019generalized}, and Generalized Boosted Regression Models (GBM) of \cite{friedman2002stochastic}.

The following functions and packages in R are used for the calculations: Axis-aligned methods including \texttt{randomForest} for RF,  \texttt{regression\_forest} and \texttt{Classification\_forest} in package \texttt{grf} for GRF,  and \texttt{RLT} for RLT. Oblique methods including \texttt{rotationForest} for RotRF, \texttt{RerF} in package \texttt{rerf} for SPORF, \texttt{obliqueRF} for ORF, and \texttt{PPforest} for PPF. boosting  methods including \texttt{xgboost} for XGB, \texttt{boosted\_regression\_forest} in package \texttt{grf} for BRF, and \texttt{gbm} for GBM. We used the default tuning parameter values for all packages, but we used 100 trees for the ensemble methods. Note that because \texttt{PPF} and \texttt{ORF} cannot be used for regression, we only report their classification results.

We use 20 real data sets with continuous responses and 20 data sets with binary categorical response (0 and 1) to demonstrate the performance of the above methods. The data are available at one of the following websites (A) \url{https://archive.ics.uci.edu/ml/datasets}, (B) \url{https://github.com/twgr/ccfs/} and (C) \url{https://www.kaggle.com}. If there are any missing values in the data, the corresponding samples are removed from the data. In the calculation, each predictor is scaled to $[0, 1]$. Specific information about all data is summarized in a supplementary; see Table S.1.1 in the supplementary document.

\begin{table}
		\centering
		\caption{Regression: average RPE based on 100 random partitions of each data set into training and test sets}\label{Table1}%
		\setlength{\tabcolsep}{1.6mm}
		{
			\begin{tabular}{lrrrrrrrrrr}
						\toprule%
						& \multicolumn{3}{c}{Axis-aligned} & \multicolumn{3}{c}{Oblique}& \multicolumn{4}{c}{Boosting}  \\
						\cmidrule(lr){2-4}\cmidrule(lr){5-7}\cmidrule(lr){8-11}%
						data & RF    & GRF   & RLT  & RotRF & SPORF & ODRF&XGB&BRF&GBM& ODBT \\
						\midrule
						data.1 & 0.287  & 0.268  & 0.310  & 0.377  & 0.246  & 0.179  & $\bm{0.108}$ & 0.385  & 0.308  & 0.127  \\
						data.2 & 0.844  & 0.793  & 0.822  & 0.797  & $\bm{0.774}$ & 0.778  & 1.197  & 0.817  & 0.862  & 0.820  \\
						data.3 & 0.132  & 0.149  & 0.135  & 0.141  & 0.143  & 0.130  & 0.152  & 0.139  & 0.144  & $\bm{0.117 }$ \\
						data.4 & 0.397  & $\bm{0.365}$ & 0.391  & 0.403  & 0.415  & 0.369  & 0.504  & 0.367  & 0.381  & 0.395  \\
						data.5 & 0.398  & 0.388  & 0.307  & 0.541  & 0.503  & $\bm{0.321}$ & 0.353  & 0.374  & 0.444  & 0.356  \\
						data.6 & 0.078  & 0.039  & $\bm{0.028}$ & 0.163  & 0.139  & 0.034  & 0.039  & 0.039  & 0.041  & 0.037  \\
						data.7 & 0.011  & $\bm{0.000}$ & $\bm{0.000}$ & 0.265  & 0.114  & $\bm{0.000}$ & $\bm{0.000}$ & $\bm{0.000}$ & 0.001  & 0.001  \\
						data.8 & 0.240  & 0.202  & $\bm{0.049}$ & 0.456  & 0.362  & 0.292  & 0.068  & 0.117  & 0.675  & 0.189  \\
						data.9 & 0.402  & 0.460  & 0.435  & 0.438  & 0.421  & 0.353  & 0.491  & 0.445  & 0.423  & $\bm{0.349 }$ \\
						data.10 & 0.114  & 0.198  & 0.150  & 0.216  & 0.166  & 0.100  & 0.147  & 0.141  & 0.138  & $\bm{0.085 }$ \\
						data.11 & 0.244  & 0.359  & 0.289  & 0.284  & 0.274  & 0.246  & 0.263  & 0.237  & 0.647  & $\bm{0.179 }$ \\
						data.12 & 0.012  & 0.016  & $\bm{0.001}$ & 0.293  & 0.165  & $\bm{0.001}$ & $\bm{0.001}$ & 0.005  & 0.003  & 0.006  \\
						data.13 & 0.820  & 0.873  & 0.788  & 0.817  & 0.814  & $\bm{0.763}$ & 0.927  & 0.848  & 0.829  & $\bm{0.763 }$ \\
						data.14 & 0.046  & 0.081  & 0.021  & 0.196  & 0.157  & 0.019  & $\bm{0.018}$ & 0.045  & 0.026  & 0.019  \\
						data.15 & 0.064  & 0.101  & 0.050  & 0.138  & 0.118  & 0.045  & 0.050  & 0.050  & 0.051  & $\bm{0.040 }$ \\
						data.16 & $\bm{0.737}$ & 0.831  & 0.776  & 0.792  & 0.773  & 0.753  & 0.822  & 0.805  & 0.798  & 0.756  \\
						data.17 & 0.529  & 0.582  & $\bm{0.527}$ & 0.766  & 0.570  & 0.541  & 0.629  & 0.560  & 0.566  & 0.652  \\
						data.18 & 0.111  & 0.151  & 0.119  & 0.251  & 0.124  & $\bm{0.102}$ & 0.113  & 0.105  & 0.169  & 0.137  \\
						data.19 & 0.049  & 0.123  & 0.048  & 0.129  & 0.080  & 0.050  & 0.061  & $\bm{0.039}$ & 0.224  & 0.059  \\
						data.20 & 0.012  & 0.027  & 0.022  & 0.044  & 0.015  & $\bm{0.010 }$ & 0.024  & 0.015  & 0.130  & 0.018  \\
					\hline
					Average&0.276  & 0.300  & 0.263  & 0.375  & 0.319  & $\bm{0.254}$  & 0.298  & 0.277  & 0.343  & 0.255  \\
				no. of bests &1     & 2     & 5     & 0     & 1     & $\bm{6}$     & 4     & 2     & 0     & $\bm{6}$  \\
						\bottomrule
					\end{tabular}%
				}
			\end{table}%

			\begin{table}[!h]
					\centering
					\caption{Classification: average MR (\%) based on 100 random partitions of each data set into training and test sets}\label{Table2}%
					\setlength{\tabcolsep}{1.2mm}
					{
							\begin{tabular}{lrrrrrrrrrrrr}
										\toprule%
										& \multicolumn{3}{c}{Axis-aligned} & \multicolumn{5}{c}{Oblique}& \multicolumn{4}{c}{Boosting}  \\
										\cmidrule(lr){2-4}\cmidrule(lr){5-9}\cmidrule(lr){10-13}%
										data & RF    & GRF   & RLT  & RotRF & SPORF &PPF&ORF& ODRF&XGB&BRF&GBM&ODBT \\
										\midrule
									data.21 & 6.88  & 5.50  & 6.34  & 6.64  & 6.27  & 9.51  & 7.56  & 4.99  & 11.23  & 5.70  & $\bm{1.87  }$ & 6.16  \\
									data.22 & 12.85  & 18.46  & 17.57  & 14.14  & 12.14  & 47.49  & 8.03  & 9.03  & 16.62  & 12.80  & 34.92  & $\bm{5.25 }$ \\
									data.23 & 32.56  & 33.45  & 30.20  & 31.93  & 28.34  & 33.36  & 24.43  & 24.15  & 34.85  & 33.19  & 34.71  & $\bm{23.95 }$ \\
									data.24 & 27.91  & 30.29  & 30.09  & 27.96  & 27.22  & 34.96  & 26.23  & 27.21  & 30.83  & 29.77  & 34.56  & $\bm{25.00 }$ \\
									data.25 & 10.49  & 11.90  & 11.04  & 9.19  & 9.26  & 10.85  & 7.39  & 6.88  & 11.02  & 10.81  & 10.83  & $\bm{6.08 }$ \\
									data.26 & 4.28  & 5.87  & 5.38  & 2.85  & 3.30  & 4.20  & $\bm{2.81  }$ & 2.88  & 6.11  & 5.11  & 4.20  & 3.26  \\
									data.27 & 15.20  & 17.03  & 93.62  & 14.02  & 13.77  & 19.75  & 40.62  & 14.44  & 14.61  & 12.85  & 14.54  & $\bm{11.63 }$ \\
									data.28 & 13.37  & 15.36  & 13.68  & 13.22  & 13.03  & 17.48  & 12.94  & 12.77  & 15.79  & 14.65  & 17.11  & $\bm{12.19 }$ \\
									data.29 & 5.41  & 7.45  & 5.43  & 7.54  & 5.19  & 10.60  & 7.37  & 5.37  & 6.75  & 6.28  & 8.16  & $\bm{4.97 }$ \\
									data.30 & 1.54  & 6.40  & 1.50  & 1.25  & 1.00  & 4.21  & $\bm{0.60  }$ & 1.56  & 6.89  & 1.63  & 11.83  & 1.28  \\
									data.31 & 6.19  & 6.83  & 6.45  & 6.21  & 6.16  & 19.69  & 6.06  & $\bm{6.05  }$ & 7.23  & 6.75  & 6.71  & 6.40  \\
									data.32 & $\bm{3.08  }$ & 3.21  & 3.09  & 3.19  & 3.13  & 3.11  & 3.18  & 3.14  & 3.62  & 3.18  & 3.13  & 3.19  \\
									data.33 & 41.32  & 46.01  & 42.38  & 11.21  & 1.30  & 29.23  & $\bm{0.00  }$ & $\bm{0.00  }$ & 38.11  & 45.98  & 50.51  & $\bm{0.00 }$ \\
									data.34 & 43.93  & 49.36  & 76.03  & 21.07  & 14.47  & 29.38  & 4.70  & 4.34  & 76.97  & 76.95  & 99.81  & $\bm{3.64 }$ \\
									data.35 & $\bm{3.97  }$ & 6.91  & 4.62  & 5.53  & 4.71  & 8.06  & 4.23  & 4.53  & 4.62  & 4.60  & 9.98  & 4.47  \\
									data.36 & 11.95  & 19.85  & 11.98  & 18.41  & $\bm{9.39  }$ & 21.68  & 17.53  & 10.52  & 14.15  & 10.66  & 19.69  & 18.90  \\
									data.37 & $\bm{19.66  }$ & 24.21  & 22.54  & 24.10  & 20.96  & 23.00  & 23.14  & 22.06  & 24.86  & 20.33  & 20.07  & 27.81  \\
									data.38 & 1.82  & 1.30  & 0.19  & 14.31  & 2.40  & 1.98  & 12.11  & 0.10  & $\bm{0.06  }$ & 0.30  & $\bm{0.06  }$ & $\bm{0.06 }$ \\
									data.39 & 32.10  & 36.23  & $\bm{29.22  }$ & 41.25  & 38.29  & 41.83  & 43.99  & 39.93  & 29.29  & 36.46  & 41.11  & 46.85  \\
									data.40 & 0.06  & 0.14  & 0.12  & 0.05  & 0.06  & 0.08  & $\bm{0.00  }$ & 0.06  & $\bm{0.00  }$ & 0.13  & $\bm{0.00  }$ & 0.07  \\
										\hline
										Average & 14.73  & 17.29  & 20.57  & 13.70  & 11.02  & 18.52  & 12.65  & $\bm{10.00}$  & 17.68  & 16.91  & 21.19  & 10.56  \\
										no. of bests & 3     & 0     & 1     & 0     & 1     & 0     & 4     & 2     & 2     & 0     & 3     & $\bm{10}$  \\
										\bottomrule
									\end{tabular}%
								}
							\end{table}%
			
			For each data set, we randomly partition it into the training set and the testing set. The training set consists of $n=\min(\lfloor 2N/3\rfloor,2000)$ randomly selected observations, where $N$ is the number of observations in the original data sets, and the remaining observations form the test set. For regression, the relative prediction error, defined as
			$$RPE=\sum_{i\in \text{test set}}(\hat{y}_i-y_i)^2/\sum_{i\in \text{test set}}(\bar{y}_{\text{train}}-y_i)^2,$$
			where $\bar{y}_{\text{train}}$ is naive predictions based on the average of $y$ in the training sets, is used to evaluate the performance of a method. For classification, the misclassification rate, defined as
			$$MR=\sum_{i\in \text{test set}} 1(\hat{y}_i \neq y_i) /(N-n),$$
			is used to assess the performance. For each data set, the random partition is repeated 100 times,  and averages of the RPEs or MRs are calculated to compare different methods. The calculation results are listed in {Table \ref{Table1}} and {Table \ref{Table2}}. The smallest RPE or MR for each data set is highlighted in \textbf{bold} font.

By comparing the prediction errors, either in terms of RPE for regression or MR for classification, our ODRF is generally smaller than the other methods. Our ODRF is quite stable and achieves the smallest RPE and MR in most datasets as listed in {Table \ref{Table1}} and {Table \ref{Table2}}. The advantages of ODRF are also confirmed by the fact that it has the smallest average RPE (or MS) across all datasets of all methods. The number of data sets for which a method is the best among all competitors, denoted by \textit{no. of bests}, also suggests the superiority of ODRF over others, including both marginal-based forests and those with linear combinations as partitioning variables.  On the other hand, our ODBT has similar performance to our ODRF, as shown in {Table \ref{Table1}} and {Table \ref{Table2}}, and the number of bests for ODBT in Table \ref{Table2} is even far superior to that of ODRF and other methods.
			
						
		Finally, we make a brief conclusion for the above experiments. Although many computer programs are developed for ODRF, but they don't show consistently better performance over RF and are not commonly received; see for example \cite{majumder2020ensembles}. We attribute this lack of improvement to the programming details and the choice of linear combinations for the splitting. After refining these issues and redesigning the bagging, all of our experiments, including many not reported here, can indeed produce a more significant overall improvement than RF. With this numerical improvement and theoretical guarantee of consistency, ODRF or e.ODBT is expected to become more popular in the future.

\section*{Supplements}

\cite{ran2024}, a publicly accessible supplementary document, contains additional details on the proofs, explanations of the R package \textsf{ODRF} developed from this paper, its tuning parameter choices, and further calculation results comparing different methods along with their computation times.


\section*{Acknowledgements}
The authors are grateful to the Editor, Associate Editor and two referees for their meticulous review and valuable comments. The research is partially supported by the National Natural Science Foundation of China (72033002 and 12271081) and an MOE AcRF tier 1 grant of Singapore (A-8001949-00-00) and the SUSTech-NUS Joint Research Program. Yu Liu is supported by the Central Government Fund for Guiding Local Scientific and Technological Development (2024ZYD00192) and the Fundamental Research Funds for the Central Universities (2682024ZTPY024).


\bibliographystyle{imsart-nameyear} 





\bibliography{ODTref2}


\newpage
\setcounter{page}{1}
\setcounter{figure}{0}
\setcounter{table}{0}
\setcounter{algocf}{0}

\renewcommand{\thetable}{S.\arabic{table}}
\renewcommand\thesection{S.\arabic{section}}
\renewcommand{\theequation}{s.\arabic{equation}}
\renewcommand{\thealgocf}{S.\arabic{algocf}}

\begin{frontmatter}
	\title{Supplement to "Consistency of  oblique decision tree and its boosting and random forest"}
	\runtitle{Supplement to "Consistency of Oblique Decision Tree"}

\begin{aug}
	\author[AC]{\fnms{Haoran}~\snm{Zhan}\ead[label=e1]{haoran.zhan@u.nus.edu}}
	\author[B]{\fnms{Yu}~\snm{Liu}\ead[label=e2]{liuy8stat@sicnu.edu.cn}}
	\author[AC]{\fnms{Yingcun}~\snm{Xia}\ead[label=e3]{staxyc@nus.edu.sg}}


      \address[AC]{Department of Statistics and Data Science,
	National University of Singapore}
    
	\address[B]{School of Mathematics Science, Sichuan Normal University \\ \printead[presep={\ }]{e1,e2,e3}}


\end{aug}

\end{frontmatter}

\renewcommand\thetheorem{S.\arabic{theorem}}
\renewcommand\thelemma{S.\arabic{lemma}}
\renewcommand\thefigure{S.\arabic{figure}}
\renewcommand\thedefinition{S.\arabic{definition}}
\renewcommand{\theequation}{s\arabic{section}.\arabic{equation}}
\setcounter{theorem}{0}
\setcounter{corollary}{0}

\allowdisplaybreaks

In this supplement, we provide auxiliary and supporting results essential to the proofs in the main paper.  Any notations introduced in the main text will be used consistently in this document without further elaboration. To avoid confusion, we prefix section, equation, inequality, algorithm, lemma, and theorem labels in this supplement with "S" or "s" to distinguish them from those in the main paper.  This supplement is organized as follows:
\begin{itemize}
\item Section \ref{sec::Pruning} introduces the pruning process for the oblique decision tree (ODT) and analyzes its statistical consistency.
\item Sections \ref{sec::2}–\ref{sec::last} contain the full detailed proofs of Lemma \ref{vclemma}, Lemma  \ref{LemmaODTrandom}, Theorem \ref{Theoremhusidhuioh}, Theorem \ref{full tree consistency}, Lemma \ref{key general function}, Lemma  \ref{equicontinuous lemma}, Proposition \ref{proposition1}, Lemma \ref{Oracle inequality for boosting tree}, Theorem \ref{Boosting tree consi} and Theorem \ref{theorem_lowerbound}, respectively.

\item Section \ref{computation} presents additional simulation results and addresses computational considerations, including the selection of tuning parameters.
\end{itemize}


\section{Pruning of ODT and its consistency}\label{sec::Pruning}

\subsection{Regression}
Pruning is an important technique in reducing the complexity of the tree model and the variance of the tree estimator. See, for example, \cite{breiman1984classification}. For  a regression tree, pruning is to select the best number of leaves by balancing the squared loss function and a penalty:
$$
t_{n,r}^*=\argmin_{1\le \tau\le n}{\frac{1}{n}\sum_{i=1}^n{(Y_i-m_{n, \tau, \L}(X_i))^2}+\alpha_n\cdot \tau},
$$
where $\alpha_n>0$ is a penalty parameter. Then, the pruned estimator in the  regression is given by
$$
m_{pru,n}(x):= m_{n, t_{n,r}^*, \L}(x).
$$

\begin{theorem}[Consistency of ODT after pruning]\label{consistencyodtreg2}
	Assume
	$
	\E(e^{c\cdot Y^2})<\infty
	$
	for some $c>0$ and $m(X)$ is $L^2$ integrable. Choose $\alpha_n=\nu \frac{\ln^5 n}{n},\nu>0$.  There is $\nu_0>0$ such that 
	$$
	\E_{\D_n}\left(\int{|m_{pru,n}(x)-m(x)|^2d\mu(x)}\right)\to 0,\ \text{as}\ n\to\infty
	$$
    as long as $\nu>\nu_0$.
\end{theorem}

\begin{remark}
    The term $\ln^5 n$ in $\alpha_n$ is used to control the sub-Gaussian tail of $Y$ and will be reduced to $\ln n$ if $Y$ is bounded.
\end{remark}

\begin{proof}

First, we introduce a result.
\begin{lemma}\label{afterprunlemmaa}
	Let $\beta_n\asymp \ln{n}$ and
	$
	\E(e^{c\cdot Y^2})<\infty
	$
	for some $c>0$ and $\alpha_n=\nu \frac{\ln^5 n}{n},\nu>0$. There is $\nu_0>0$ such that for any $1\le \tau\le n$ and $g_J\in \mathcal{G}_{J}$, we have
	\begin{equation}\label{bSta4}
		\begin{aligned}
			&\E_{\mathcal{D}_{n}}\left( \int|\hat{m}_{pru,n}(x)-m(x)|^2d\mu(x)\cdot\mathbb{I}(A_n)\right) \\
			&\le 2\E_{\mathcal{D}_{n}}\left(\left(\|m_{\tau}(X)-\mathbb{Y}\|_n^2-\|g_J(X)-\mathbb{Y}\|_n^2\right)\cdot \mathbb{I}(A_n)\right) \\
			&+ c\cdot\E\left(m(X)-g_J(X)\right)^2
			+ 2\tau\cdot\alpha_n + c\cdot \frac{\ln^2 n}{n},
		\end{aligned}
	\end{equation}
	where $\hat{m}_{pru,n}(x)=\max{\{\min{\{m_{pru,n}(x),\beta_n\}},-\beta_n\}}$, and $A_n=\{\max_{1\le i\le n}{|Y_i|\le \beta_n}\}$, $c>0$, 
     $$
	Ridge_{J}=\Big\{\sum_{j=1}^{J} {c_j\cdot \sigma(\theta_j^Tx+d_j)}: \theta_j\in\Theta^p, c_j,d_j\in\mathbb{R}, \forall j\ge 1\Big\}
	$$
	and $\sigma(v)=e^v/(1+e^v), v\in\mathbb{R}$.
\end{lemma}

\begin{proof}
	The proof is similar to the proof of  (6.24) in Lemma 1 of \cite{zhan2022ensemble}.
\end{proof}

Let $A_n=\{\max_{1\le i\le n}{|Y_i|\le \beta_n}\}$  and $\hat{m}_{pru}(x)=\max{\{\min{\{m_{pru}(x),\beta_n\}},-\beta_n\}}$, where $\beta_n\asymp \ln{n}$. Note that
	\begin{align}
		\E_{\D_n}\left(\int{|m_{pru,n}(x)-m(x)|^2d\mu(x)}\right)\le & 2\E_{\D_n}\left(\int{|m_{pru,n}(x)-\hat{m}_{pru,n}(x)|^2d\mu(x)}\right)\nonumber\\
		&+2\E_{\D_n}\left(\int{|\hat{m}_{pru,n}(x)-m(x)|^2d\mu(x)}\cdot\mathbb{I}(A_n)\right)\nonumber\\
		&+2\E_{\D_n}\left(\int{|\hat{m}_{pru,n}(x)-m(x)|^2d\mu(x)}\cdot\mathbb{I}(A_n^c)\right) \nonumber  \\
		:= & I+II+III.\label{guifshuik}
	\end{align}
By  similar arguments  of $V\to 0$ in \textbf{Part III} of the proof of Theorem \ref{consistencyodtreg}, it is not difficult to show that
	\begin{equation}\label{adhgauidh}
		I\to 0.
	\end{equation}
	Next, we consider $II$. For any $\varepsilon>0$, \eqref{iopipo} shows that for some large $J$ there exists $g_J\in Ridge_J$ such that
	\begin{equation}\label{jijipoq}
		c\cdot\E\left(m(X)-g_J(X)\right)^2\le \frac{\varepsilon}{3}.
	\end{equation}
	Fix  $g_J$ above and let $\tau_n\asymp\sqrt{n}$. Then, by \eqref{regpart1} we know that there is  $N_1\in\mathbb{Z}_+$ such that \begin{equation}\label{reserrt1}
		2\E\left((\|\mathbb{Y}-m_{\tau_n}(X)\|_n^2- \|\mathbb{Y}-g(X)\|_n^2)\cdot\mathbb{I}(A_n)\right)\le \frac{2\|g_J\|^2_{TV}}{\log_2{\tau_n}+4}\le\frac{\varepsilon}{3}
	\end{equation}
	for all $n\ge N_1$. We can also find $N_2\in\mathbb{Z}_+$ such that
	\begin{equation}\label{jdisjdfsp}
		2\tau_n\cdot\alpha_n + c\cdot \frac{\ln^2 n}{n}\le \frac{\varepsilon}{3}
	\end{equation}
	for all $n\ge N_2$. By Lemma \ref{afterprunlemmaa}, the combination of \eqref{jijipoq}, \eqref{reserrt1} and \eqref{jdisjdfsp} shows that
	\begin{equation}\label{huihdaskuhdku}
		II\to 0.
	\end{equation}
	Note that $\sup_{x\in [0,1]^p}|\hat{m}_{pru,n}(x)|\le \beta_n$ and $m(X)$ is $L^2$ integrable. From \eqref{duhsihrdksh} we have
	\begin{equation}\label{huohud}
		III\to 0.
	\end{equation}
	Finally,  the combination of \eqref{guifshuik}, \eqref{adhgauidh}, \eqref{huihdaskuhdku} and \eqref{huohud} completes the proof.
\end{proof}

\subsection{Classification}
For binary classification problem, where $Y$ only takes $0$ or $1$, we can still use ODT but employ  Gini impurity  to divide each internal node. This criterion is defined by
$$
\Delta^c_A(\theta,s)=-\sum_{k=1}^2{P^2(k|A)}+P(A^+_{\theta,s})\sum_{k=1}^2{P^2(k|A^+_{\theta,s})}+P(A^-_{\theta,s})\sum_{k=1}^2{P^2(k|A^-_{\theta,s})},
$$
where $P(k|A)$ denotes the proportion of class $k$ in $A$ and $k\in\{1,2\}$. For this oblique classification tree,  it can be checked that Algorithm \ref{Algorithm.ODTtreereg} still works after changing $\Delta_A(\theta,s)$ in  \textbf{line 13} by $\Delta_A^c(\theta,s)$. After voting, the estimated class of input $x\in[0,1]^p$ is equal to
\begin{equation}\label{hdfuioshjdfois}
	\hat{C}_{n,t_n}(x)=
	\left\{
	\begin{array}{lc}
		1,      &   m_{n,\L}(x) \geq 0.5, \\
		0,      &   m_{n,\L}(x)< 0.5 .\\
	\end{array}
	\right.
\end{equation}

For classification tree, pruning is also used select the best number of leaves by balancing empirical risk and a penalty:
$$
t_{n,c}^*=\argmin_{1\le \tau\le n}{\frac{1}{n}\sum_{i=1}^n{\mathbb{I}(Y_i\neq  \hat{C}_{n,\tau}(X_i))}+\alpha_n\cdot \tau},
$$
where $\alpha_n>0$ is a given penalty parameter and $\mathbb{I}(\cdot)$ denotes the indicator function. Then, the classifier after pruning is given by
$$
\hat{C}_{pru,n}(x):= \hat{C}_{n,t_{n,c}^*}(x).
$$

With some calculations, we have
	$$	\Delta^c_A(\theta,s)=2\Delta_A(\theta,s)
	$$
	if $Y$ only takes $0$ or $1$. Therefore, Theorem \ref{consistencyodtreg} in main paper can be also used in the case of classification tree. We know  $\hat{C}_{n,t_n}$ is  Bayes-optimal in asymptotical sense by using Theorem \ref{consistencyodtreg} and Theorem 1.1 in \cite{gyorfi2006distribution}.

\begin{corollary}
	The  misclassification  probability of classifier $\hat{C}_{n,t_n}(x)$ in \eqref{hdfuioshjdfois} satisfies
	$$
	\P(\hat{C}_{n,t_n}(X)\neq Y)-\inf_{f:[0,1]^p\to\{0,1\}}{\P\left(f(X)\neq Y\right)} \to 0\ \ \text{as}\ \ n\to \infty
	$$
	under conditions in Theorem \ref{consistencyodtreg}.
\end{corollary}

\section{Proof of Lemma \ref{vclemma}}\label{sec::2}
For any tree $h_{t_n}\in\mathcal{H}_{t_n}$, it can be written as 
    \begin{equation}\label{ODT tree2}
         h_{t_n}(x)=\sum_{j=1}^{t_n}\I(x\in\mathbb
    {A}_{\L}^j)\cdot c_j,
    \end{equation}
    where $c_j\in\R$ and $\mathbb{A}_{\L}^j, j=1,\ldots,t_n$ are leaves obtained in Algorithm \ref{Algorithm.ODTtreereg}. To reduce notations,  each leaf $\mathbb
    {A}_{\L}^j$ are abbreviated as $\mathbb
    {A}_j$ in this proof. Meanwhile, we also use $\{\mathbb
    {A}_1,\ldots,\mathbb
    {A}_{t_n};c_1,\ldots,c_{t_n}\}$ to denote tree \eqref{ODT tree2} later.

    In order to bound $VC(\mathcal{H}_{t_n})$, we need to consider the related Boolean class:
    	$$
	\mathcal{F}_{t_n}= \{sgn(f(x,y)): f(x,y)=h(x)-y, h\in\mathcal{H}_{t_n}\},
	$$
	where $sgn(v)=1$ if $v\ge 0$ and $sgn(v)=-1$ otherwise and $\mathcal{H}_{t_n}$ is defined below Lemma \ref{CoveringnumbervsVCd}. Recall the VC dimension of $\mathcal{F}_{t_n}$, denoted by $VC(\mathcal{F}_{t_n})$, is the largest integer $m\in\mathbb{Z}_+$ satisfying $2^m\leq\Pi_{\mathcal{F}_{t_n}}(m)$; see Definition \ref{defvc}. Therefore, we next focus  on bounding $\Pi_{\mathcal{F}_{t_n}}(m)$ for each positive integer $m\in\mathbb{Z}_+$. 

    Define the partition set generated by $\mathcal{H}_{t_n}$:
    $$
\Pa_{t_n}:=\left\{ \{\mathbb{A}_1,\ldots,\mathbb
    {A}_{t_n}\}: \{\mathbb
    {A}_1,\ldots,\mathbb
    {A}_{t_n};c_1,\ldots,c_{t_n}\}\in\mathcal{H}_{t_n}\right\}
    $$
    and the maximal partition number of $m$ points cut by $\Pa_{t_n}$:
    \begin{equation}\label{1}
        K(t_n):=\max_{x_1,\ldots,x_m\in [0,1]^p}Card\left(\left\{ \bigcup_{\Pa\in\Pa_{t_n}}\{\{x_1,\ldots,x_m\}\cap A: A\in\Pa\}\right\}\right).
    \end{equation}

    Since each ODT takes constant value on its leaf $\mathbb
    {A}_j$ $(j=1,\ldots,t_n)$, thus there are at most $\ell+1$ ways to pick out any fixed points $\{(x_1,y_1),\ldots, (x_\ell.y_\ell)\}\subseteq[0,1]^p\times \R$ that lie in a same leaf. In other words, 
    \begin{align}
        \max_{\substack{(x_j,y_j)\in[0,1]^p\times\R\\ j=1,\ldots,\ell}}Card\Big(\Big\{(sgn(f(x_1,y_1))&,\ldots,sgn(f(x_\ell,y_\ell)))\nonumber\\
        &: f(x,y)=c-y,x\in [0,1]^p, y\in\R, c\in\R\Big\}\Big)\le \ell +1.\label{2}
    \end{align}

    According to the structure of ODT in \eqref{ODT tree2}, the growth function of $\mathcal{H}_{t_n}$ satisfies
    \begin{equation}\label{3}
        \Pi_{\mathcal{F}_{t_n}}(m)\le K(t_n)\cdot (m+1)^{t_n}.
    \end{equation}

    Therefore, the left task is to bound $K(t_n)$ only. In fact, we will use induction method to prove 
    \begin{equation}\label{4}
        K(t_n)\le \left[ c(p)(3m)^{p+1}\right]^{t_n},
    \end{equation}
    where the constant $c(p)>0$ depends on $p$ only.  The arguments are given below.

    When $t_n=1$, the partition generated by $\mathcal{H}_{1}$ is $\{[0,1]^p\}$. Thus, $K(1)=1$ and \eqref{4} is satisfied obviously in this case.

    Suppose \eqref{4} is satisfied with $t_n-1$. Now we consider the case for $t_n$. Here, we need two facts. First, any $h_{t_n}$ must be grown from a $h_{t_n-1}$ by splitting one of its leaves; Second, any $h_{t_n-1}$ can grow to be another $h_{t_n}$ by partitioning one of its leaves. In conclusion, 
    \begin{equation}\label{-1}
          \mathcal{H}_{t_n}=\{h_{t_n}:h_{t_n} \text{\ is\ obtained by splitting one of the leaves of } h_{t_n-1}\in \mathcal{H}_{t_n-1}\}.
    \end{equation}
    Suppose $x_1^*,\ldots,x_m^*\in [0,1]^p$ maximizes $K(t_n)$. We divide $\mathcal{H}_{t_n-1}$ into $K(t_n-1)$ groups:
    $$
 \mathcal{H}_{t_n-1}^j:=\left\{ h_{t_n-1}\in  \mathcal{H}_{t_n-1}: h_{t_n-1}s\text{ share with a same partition for } \{x_1^*,\ldots,x_m^*\} \right\},
    $$
    where $j=1,2,\ldots, K(t_n-1)$. Meanwhile, we write the partition for $\{x_1^*,\ldots,x_m^*\}$ that is cut by $ \mathcal{H}_{t_n-1}^j$ as
    \begin{equation}\label{5}
      \underbrace{  \{x^*_{u_1^j},x^*_{u_2^j},\ldots,x^*_{u^j_{m_{1,j}}}\}} _{m_{1,j}\ \text{times}},\ldots, \underbrace{\{x^*_{u^j_{m_{1,j}+\cdots+m_{t_n-2,j}+1}},\ldots,x^*_{u_m}\}}_{m_{t_n,j}\ \text{times}},
    \end{equation}
    where $(u_1^j,\ldots,u_m^j)$ is a permutation of $(1,\ldots,m)$ and the cardinality of above sets are written as $m_{1,j},m_{2,j},\ldots, m_{t_n,j}$. Note that both $(u_1^j,\ldots,u_m^j)$ and $(m_{1,j},m_{2,j},\ldots,m_{t_n,j})$ are determined once the index $j$ is fixed.

    If we have generated a tree with $t_n$ leaves from its parent in $\mathcal{H}_{t_{n-1}}^j$, the corresponding partition for $\{x_1^*,\ldots,x_m^*\}$ can be obtained by following two steps below
    \begin{enumerate}
        \item Partition one of sets in \eqref{5} by using a hyperplane $\theta^Tx=s$, where $\theta\in\R^p, s\in\R$.
        \item Keep other $t_n-2$ sets in \eqref{5} unchanged.
    \end{enumerate}
    Denote by $\tilde{K}_j$ the number of partition for $\{x_1^*,\ldots,x_m^*\}$ after following above process. Then, the Sauer–Shelah lemma (see for example Lemma 6.10 in \cite{shalev2014understanding} ) tells us 
    \begin{equation*}
        \tilde{K}_j\le  \sum_{\ell=1}^{t_n-1}c(p)\left( \frac{m_\ell e}{p+1}\right)^{p+1}\le \sum_{\ell=1}^{t_n-1}c(p)\left( 3m_\ell\right)^{p+1}
        \le c(p)\cdot (3m)^{p+1}.
    \end{equation*}
    Since  each $\mathcal{H}_{t_{n-1}}^j$ can at most produce $c(p)(3m)^{p+1}$ new partitions of $\{x_1^*,\ldots,x_m^*\}$ based on \eqref{5}, \eqref{-1} and $\mathcal{H}_{t_{n}-1}=\cup_{j=1}^{K(t_n-1)}\mathcal{H}_{t_{n-1}}^j$ imply
    $$
   K(t_n)\le  K(t_n-1)\cdot c(p) (3m)^{p+1}.
    $$
    Therefore, \eqref{4} also holds for the case of $K(t_n)$.

According to \eqref{3} and \eqref{4}, we finally have
$$
  \Pi_{\mathcal{F}_{t_n}}(m)\le [c(p)\cdot (3m)^{p+1}]^{t_n}\cdot (m+1)^{t_n}\le (3c(p)m)^{t_n(p+2)}.
$$
Solving the inequality
	$$
	2^m\le (3c(p)m)^{t_n(p+2)}
	$$
	by using the basic inequality $\ln x\leq\gamma\cdot x-\ln \gamma-1 $ with  $ x,\gamma >0$ yields
	\begin{equation}
		VC(\mathcal{H}_{t_n})\le c(p)\cdot t_n\ln(t_n), \forall t_n\ge 2,
	\end{equation}
	where the constant $c(p)>0$ depends on $p$ only. This completes the proof. \hfill\(\Box\)

\section{Proof of Lemma  \ref{LemmaODTrandom}}

	For simplicity,  notations $A^+$ and $A^-$ are used to denote $A^+_{\hat{\theta}_A,\hat{s}_A}$ and $A^-_{\hat{\theta}_A,\hat{s}_A}$ respectively when $q=q_0$. For each $1\le k\le L$, let $\Delta_{A}(\hat{\theta}_{A},\hat{s}_A, x_{\C_{k,q_0}})$ be the value of impurity gain in \eqref{impuritygainreg} when variables $x_{\C_{k,q_0}}$ are   used only in the calculation of ${\Delta_{A}(\hat{\theta}_{A},\hat{s}_A)}$, namely only $x_{\C_{k,q_0}}$ are employed to divide the node $A$. Then, we have
	\begin{equation}\label{huhadikl}
		\E_{\S_\tau}\left({\Delta_{A,q}(\hat{\theta}_{A},\hat{s}_A)}|\D_n,q=q_0\right)\geq p_{q_0}\cdot\sum_{k=1}^V{\Delta_{A}(\hat{\theta}_{A},\hat{s}_A, x_{\C_{k,q_0}})}\geq p_{q_0}\cdot\max_{1\le k\le V}{\Delta_{A}(\hat{\theta}_{A},\hat{s}_A,x_{\C_{k,q_0}}}),
	\end{equation}
	where $p_{q_0}=1/p\cdot1/\binom{p}{q_0}$ is the probability of each ${\C_{k,q_0}}$. Let $g_{q_0}(x)=\sum_{k=1}^V{g_{k,q_0}(\theta_{\C_{k,q_0}}^Tx)}\in Ridge_{k,q_0, J}$, where $g_{k,q_0}(\theta_{\C_{k,q_0}}^Tx)= \sum_{j=1}^{J} {c_{k,q_0,j}\cdot \sigma(\theta_{\C_k}^Tx+d_{k,j})},  c_{k,q_0,j},d_{k,q_0,j}\in\mathbb{R}$. Define a series of weight functions of parameter $s\in [-1,1]$ by
	$$
	w_{k,q_0}(s)=\frac{|g'_{k,q_0}(s)|\sqrt{P(A^+|\C_{k,q_0})P(A^-|\C_{k,q_0})}}{\sum_{k=1}^V{\int_{-1}^{1}{|g'_{k,q_0}(s)|\sqrt{P(A^+|\C_{k,q_0})P(A^-|\C_{k,q_0})}}ds}},
	$$
	where $P(A^+|\C_{k,q_0})$ and $P(A^-|\C_{k,q_0})$ are proportions of data in $A^+$ and $A^-$ within $A$ respectively if the hyperplane $\theta^Tx_{\C_{k,q_0}}=s$ is used to divide $A$. Note that $w_{k,q_0}(s)\ge 0$ and $\int_{-1}^1{w_{k,q_0}(s)ds}\leq 1$. Following arguments on page 19 of \cite{klusowski2021universal}, it is not difficult to show that
	\begin{equation}\label{pikojiq}
		\max_{1\le k\le V}{\Delta_{A}(\hat{\theta}_{A},\hat{s}_A,x_{\C_{k,q_0}}})\geq \frac{|\langle \mathbb{Y}-\bar{Y}_A,g_{k,q_0}(\theta^T_{\C_{k,q_0}}x)\rangle_A|^2}{\left(\sum_{k=1}^V{\int_{-1}^{1}{|g'_{k,q_0}(s)|\sqrt{P(A^+|\C_{k,q_0})P(A^-|\C_{k,q_0})}}ds}\right)^2}
	\end{equation}
	for each $1\le k\le V$. By taking average on $|\langle \mathbb{Y}-\bar{Y}_A,g_{k,q_0}(\theta^T_{\C_{k,q_0}}x)\rangle_A|^2, 1\le k\le V$, \eqref{pikojiq} implies
	\begin{equation}\label{bkdsahdlkh}
		\max_{1\le k\le V}{\Delta_{A}(\hat{\theta}_{A},\hat{s}_A,x_{\C_{k,q_0}}})\geq \frac{|\langle \mathbb{Y}-\bar{Y}_A,g_{q_0}\rangle_A|^2}{V\cdot\left(\sum_{k=1}^V{\int_{-1}^{1}{|g'_{k,q_0}(s)|\sqrt{P(A^+|\C_{k,q_0})P(A^-|\C_{k,q_0})}}ds}\right)^2}.
	\end{equation}
	Note that
	\begin{align}
		\langle \mathbb{Y}-\bar{Y}_A,g_{q_0}\rangle_A &= \langle \mathbb{Y}-\bar{Y}_A,\mathbb{Y}\rangle_A-\langle \mathbb{Y}-\bar{Y}_A,\mathbb{Y}-g_{q_0}\rangle_A\nonumber\\
		&\geq \|\mathbb{Y}-\bar{Y}_A\|_A^2- \|\mathbb{Y}-\bar{Y}_A\|_A\|\mathbb{Y}-g_{q_0}\|_A\nonumber\\
		&\geq \|\mathbb{Y}-\bar{Y}_A\|_A^2-\frac{1}{2}\cdot \left(\|\mathbb{Y}-\bar{Y}_A\|_A\|^2+\|\mathbb{Y}-g_{q_0}\|_A^2\right),\label{lkpoakd}
	\end{align}
	where the Cauchy-Schwarz inequality is used in the second line.  Therefore,  \eqref{bkdsahdlkh} and \eqref{lkpoakd} and the fact that $R(A)\ge 0$ by the assumption of this lemma imply
	\begin{equation}\label{eqeqa}
		\max_{1\le k\le V}{\Delta_{A}(\hat{\theta}_{A},\hat{s}_A,x_{\C_{k,q_0}}})\geq \frac{R^2(A)}{4V\cdot\left(\sum_{k=1}^V{\int_{-1}^{1}{|g'_{k,q_0}(s)|\sqrt{P(A^+|\C_{k,q_0})P(A^-|\C_{k,q_0})}}ds}\right)^2}.
	\end{equation}
	Next, we only need to consider how to bound the denominator of the RHS of \eqref{eqeqa}. In fact, following similar arguments on page 20 of \cite{klusowski2021universal}, it is easy to get a bound by using the total variation of each $g_{k,q_0}, 1\le k\le V$:
	\begin{equation}\label{mdipajdmop}
		\sum_{k=1}^V{\int_{-1}^{1}{|g'_{k,q_0}(s)|\sqrt{P(A^+|\C_{k,q_0})P(A^-|\C_{k,q_0})}}ds} \le \frac{1}{2}\cdot \|g_{q_0}\|_{TV}.
	\end{equation}
	Therefore, the combination of \eqref{huhadikl} and \eqref{eqeqa} and \eqref{mdipajdmop} completes the proof.
	\hfill\(\Box\)

\section{Proof of Theorem \ref{Theoremhusidhuioh}}

To prove the consistency, we also need  a truncated random tree at any particular layer $ \ell: 0 \le \ell \le \L $, denoted by  $ T_{\D_n,a_n,t_n,\ell}^r $. For simplicity, we abbreviate  $ T_{\D_n(\mathfrak{G}),a_n,t_n,\ell}^r $ as $ T_{\D_n(\mathfrak{G}),\ell}^r$ for each layer $\ell$. Similarly, $m^r_{n,\L}(x)$ is an abbreviation of $m^r_{n,a_n,t_n,\L}(x)$ if there is no confusion. 

Here, we highlight two points. First, the expectation in \eqref{regpart2dd} is not taken over $\D_n$. Thus, we just fix $\D_n$ and omit it when writing conditional expectations or probabilities during the proof for \eqref{regpart2dd}.   Second, for any internal node $A$ of $T^r_{\D_n,\L}$, $\Delta_{A,q}(\hat{\theta}_A,\hat{s}_A)$ only depends on $\S_\lambda$ for some $1\le\lambda\le t_n-1$ once data $\D_n$ is given.

It is easy to observe the following two facts:
	\begin{itemize}
		\item Regardless of the choice of each $\S_\tau,1\le \tau\le t_n-1$, the level $\L$ of any random ODT, $T_{\D_n,\L}^r$, must not be less than $\L_0=\lfloor\log_2{t_n}\rfloor$;
		\item Tree $T_{\D_n,\ell}^r$ is fully grown for each $0\leq \ell\le \L_0$, namely $T_{\D_n,\ell}^r$ is obtained recursively by splitting all leaves of $T_{\D_n,\ell-1}^r$ until all leaves contain only one data point.
	\end{itemize}	
	Recall that $\L$ is a random variable depending on both $\D_n$ and $\S_\tau, 1\le\tau\le t_n-1,$ and that $m_{n,\ell}^r(x)$ is the estimator obtained using the leaves of $T_{\D_n,\ell}^r$. Given $\D_n$, define the expectation of $\|\mathbb{Y}-m_{n,\ell}^r(X)\|_n^2$ over random $\theta$'s by
	\begin{equation}\label{pokpok}
		\E_{\ell}\left(\|\mathbb{Y}-m_{n,\ell}^r(X)\|_n^2\right):=\sum_{T^r_{\D_n,\ell}} {\P_\theta{\left(T_{\D_n,\ell}^r\right)}}\cdot \|\mathbb{Y}-m_{n,\ell}^r(X)\|_n^2,
	\end{equation}
	where $\P_\theta{\left(T_{\D_n,\ell}^r\right)}$ is the probability of a realization of $\S_\tau$s, each corresponding to a partition of an internal node of $T_{\D_n,\ell}$. Using the fact that $\|\mathbb{Y}-m_{n,\ell}^r(X)\|_n^2$ is almost certainly a decreasing sequence as $\ell$ increases, it is easy to check that
	\begin{equation}\label{fuosjhnfliszjnkf}
		\E_{\Xi_{t_n}}\left(\|\mathbb{Y}-m_{n,\L}^r(X)\|_n^2\right)\leq \E_{\L_0}\left(\|\mathbb{Y}-m_{n,\L_0}^r(X)\|_n^2\right).
	\end{equation}
	So the next inequality we need to prove is
	\begin{equation}\label{dhzuadhnilq}
		\E_{\L_0}\left(\|\mathbb{Y}-m_{n,\L_0}^r(X)\|_n^2\right)-\|\mathbb{Y}-g_p(X)\|_n^2\leq c(p)\cdot\frac{\|g_p\|^2_{TV}}{\lfloor\log_2{t_n}\rfloor+4}
	\end{equation}
	for some $c(p)>0$.	
 
Using the same notation in the proof of Theorem \ref{consistencyodtreg},  we define $R_{\D_n,\ell}: =\|\mathbb{Y}-m^r_{n,\ell}(X)\|_n^2-\|\mathbb{Y}-g(X)\|_n^2$ for each $1\le\ell\le \L_0$. Similarly, let    $\mathcal{O}_{\ell-1,1}:=\{\mathbb{A}_{\ell-1}^j\}_{j=1}^{k_{\ell-1}}$ be leaves of $T^r_{\D_n,\ell-1}$ and $\mathcal{O}_{\ell-1,2}:=\{\mathbb{A}_{\ell-1}^{m_j}\}_{j=1}^{k'_{\ell-1}}\subseteq \mathcal{O}_{\ell-1,1}$ of which each node  $\mathbb{A}_{\ell-1}^{m_j}$ must   contain at least two data points. Let $\mathcal{Q}_{\ell-1,2}\subseteq\{\S_\tau\}_{\tau=1}^{t_n-1}$ which corresponds to each partition of each node in $\mathcal{O}_{\ell-1,2}$. Then, we have
			\begin{equation}\label{fsjinflz}
				R_{\D_n,\ell}=  R_{\D_n,
					\ell-1}-\sum_{A\in \mathcal{O}_{\ell-1,2}}{w(A)\cdot \Delta_{A,q}(\hat{\theta}_{A},\hat{s}_{A}}),
			\end{equation}
			where  $w(A)= \frac{1}{n}\cdot Card(\{X_i\in A: i=1,\ldots, n\})$ as defined above. Given the tree $T^r_{\D_n,\ell-1}$, define the conditional expectation on $\mathcal{Q}_{\ell-1,2}$ by
	\begin{equation}
		\E_{\mathcal{Q}_{\ell-1,2}}\left(R_{\D_n,\ell}|T^r_{\D_n,\ell-1} \right) :=  \sum_{m_{n,\ell}^r\ \text{is from } T^r_{\D_n,\ell}\ \text{generated by } T^r_{\D_n,\ell-1}}{\P_\theta{\left(T_{\D_n,\ell}^r\right)}\cdot \|\mathbb{Y}-m_{n,\ell}^r(X)\|_n^2}.\label{rtyrfht}
	\end{equation}
	Then, based on \eqref{fsjinflz} we also have
	\begin{equation}\label{fjxisjflisdjflisj}
		\E_{\mathcal{Q}_{\ell-1,2}}\left(R_{\D_n,\ell}|T^r_{\D_n,\ell-1} \right) = R_{\D_n,
			\ell-1}-\sum_{A\in  \mathcal{O}_{\ell-1,2}}{w(A)\cdot \E_{\S^A}\left(
			\Delta_{A,q}(\hat{\theta}_{A},\hat{s}_{A})
			\right)},
	\end{equation}
	where random index $\S^A\in \{\S_\tau\}_{\tau=1}^{t_n-1}$ corresponds to the partition of $A$. Note that leaves $\mathcal{O}_{\ell-1,2}$ are not randomized once the tree $T^r_{\D_n,\ell-1}$ is given.
	
	Now, we are ready to prove \eqref{dhzuadhnilq}. By \eqref{pokpok} and \eqref{rtyrfht}, it is easy to check
	\begin{align}
		\E_{\ell}\left(R_{\D_n,\ell}\right)&=\sum_{T^r_{\D_n,\ell}} {\P_\theta{\left(T_{\D_n,\ell}^r\right)}}\cdot \E_{\mathcal{Q}_{\ell-1,2}}\left(R_{\D_n,\ell}|T^r_{\D_n,\ell-1} \right)\nonumber\\
		&= \E_{\Xi_{t_n}}\left(\E_{\mathcal{Q}_{\ell-1,2}}\left(R_{\D_n,\ell}|T^r_{\D_n,\ell-1} \right)\right),\label{pkoopkq}
	\end{align}
	which is similar to the law of iterated expectations. We highlight the fact that both  $T^r_{\D_n,\ell-1}$ and $\mathcal{Q}_{\ell-1,2}$ are random, so \eqref{pkoopkq} is not a trivial result from the classical law of iterated expectations. Then, from \eqref{fjxisjflisdjflisj} and Lemma \ref{LemmaODTrandom}, we have
	\begin{align}
		\E_{\mathcal{Q}_{\ell-1,2}}\left(R_{\D_n,\ell}|T^r_{\D_n,\ell-1} \right) &= R_{\D_n,
			\ell-1}-\sum_{A\in\O_{\ell-1,2}}{w(A)\cdot \E_{\S^A}\left(\Delta_{A,q}(\hat{\theta}_{A},\hat{s}_{A})\right)}
		\nonumber \\
		&\leq R_{\D_n,
			\ell-1}-\sum_{A\in\O_{\ell-1,2}}{\frac{w(A)}{p}\cdot \E_{\S^A}\left(\Delta_{A,q}(\hat{\theta}_{A},\hat{s}_{A})|q=p\right)}
		\nonumber \\
		&\leq R_{\D_n,
			\ell-1}-c(p)\cdot\sum_{A\in\O_{\ell-1,2}}{w(A)\cdot  \frac{R^2(A)}{\|g_p\|_{TV}^2}}\nonumber\\
		&\leq R_{\D_n,
			\ell-1}-\frac{c(p)}{\|g_p\|_{TV}^2}\cdot\sum_{A\in\O_{\ell-1,2}:R(A)> 0}{w(A)\cdot R^2(A)},\label{koadk}
	\end{align}
	where $R(A)=\|\mathbb{Y}-\bar{Y}_{A}\|_{A}^2-\|\mathbb{Y}-g_p\|_{A}^2$ and $c(p)>0$.  Decompose $R_{\D_n,\ell-1}$ into two parts
$$R_{\D_n,\ell-1}^+:=\sum_{A\in\O_{\ell-1,1}:R(A)> 0}{w(A)R(A)} \  \text{and} \   R_{\D_n,\ell-1}^-:=\sum_{A\in\O_{\ell-1,1}:R(A)\le 0}{w(A)R(A)}$$ 
 satisfying $R_{\D_n,\ell-1}=R_{\D_n,\ell-1}^++R_{\D_n,\ell-1}^-$. By Jensen's inequality and the fact that $R(A)\le 0$ for any leaf $A$ of $T_{\D_n,\ell-1}$ which contains only one data point, we have
	\begin{equation}\label{dgiasydfuip}
		\sum_{A\in\O_{\ell-1,2}:R(A)> 0}{w(A)\cdot R^2(A)} \ge \left(\sum_{A\in\O_{\ell-1,2}:R(A)> 0}{w(A)\cdot R(A)}\right)^2=(R_{\D_n,\ell-1}^+)^2.
	\end{equation}
	Therefore, if $R_{\D_n,\ell-1}\ge 0$ then the combination of \eqref{koadk} and \eqref{dgiasydfuip} implies that
	\begin{align}
		\E_{\mathcal{Q}_{\ell-1,2}}\left(R_{\D_n,\ell}|T^r_{\D_n,\ell-1} \right) &\le  R_{\D_n,\ell-1}- \frac{c(p)}{\|g_p\|_{TV}^2}\cdot(R_{\D_n,\ell-1}^+)^2\nonumber\\
		&\le R_{\D_n,\ell-1}- \frac{c(p)}{\|g_p\|_{TV}^2}\cdot R_{\D_n,\ell-1}^2, \label{dsfqwe23}
	\end{align}
	where \eqref{dsfqwe23} is from $R_{\D_n,\ell-1}^+>R_{\D_n,\ell-1}\ge 0$. In conclusion, \eqref{dsfqwe23} implies that
	\begin{equation}\label{ojdfiasjdop}
		\E_{\mathcal{Q}_{\ell-1,2}}\left(R_{\D_n,\ell}|T^r_{\D_n,\ell-1} \right) \le R_{\D_n,\ell-1}- \frac{c(p)}{\|g\|_{TV}^2}\cdot \max{\{R_{\D_n,\ell-1},0\}}^2
	\end{equation}
	because $\E_{\mathcal{Q}_{\ell-1,2}}\left(R_{\D_n,\ell}|T^r_{\D_n,\ell-1} \right) \le R_{\D_n,\ell-1}$  holds no matter what the sign of $R_{\D_n,\ell-1}$ takes. By \eqref{pkoopkq} and Jensen's inequality again, taking expectation over $\Xi_{t_n}$ on both sides of \eqref{ojdfiasjdop} yields
	\begin{align}
		\E_{\ell}(R_{\D_n,\ell})&\le \E_{\ell-1}(R_{\D_n,\ell})-\frac{c(p)}{\|g_p\|_{TV}^2}\cdot \E_{t_n}\left(\max{\{R_{\D_n,\ell-1},0\}}^2\right)\nonumber\\
		&\leq \E_{\ell-1}(R_{\D_n,\ell-1})-\frac{c(p)}{\|g_p\|_{TV}^2}\cdot [\E_{t_n}\left(\max{\{R_{\D_n,\ell-1},0\}}\right)]^2\label{jjoijkq}
	\end{align}
	for each $1\le\ell\le \L_0$. Next, we complete the proof of \eqref{dhzuadhnilq} by considering the sign of $ \E_{\ell-1}(R_{\D_n,\ell-1})$ in the following two cases.
	
	\textbf{Case 1:} There is an $ \ell_0 $, with $1\le\ell_0\le \L_0$,  such that $\E_{\ell_0-1}{(R_{\D_n,\ell_0-1})}\le 0$. By checking \eqref{rtyrfht}, it is easy to know that
	\begin{equation}\label{poqqasd}
		\E_{\mathcal{Q}_{\ell-1,2}}\left(R_{\D_n,\ell}|T^r_{\D_n,\ell-1} \right) \le R_{\D_n,\ell-1}.
	\end{equation}
	By using \eqref{pkoopkq}, taking expectation over $\Xi_{t_n}$ on both sides of \eqref{poqqasd} implies that
	\begin{equation*}
		\E_{\ell}(R_{\D_n,\ell})\leq   \E_{\ell-1}(R_{\D_n,\ell-1})
	\end{equation*}
	for each $1\le\ell\le \L_0$. Therefore, we have
	\begin{equation}\label{poqaqasd}
		\E_{\L_0}(R_{\D_n,\L_0})\le\E_{\ell_0-1}(R_{\D_n,\ell_0-1})\le c(p)\cdot\frac{\|g_p\|^2_{TV}}{\lfloor\log_2{t_n}\rfloor+4}
	\end{equation}
	by the assumption of this case.
	
	\textbf{Case 2:} For each $1\le\ell\le \L_0$, we have $\E_{\ell-1}(R_{\D_n,\ell-1})> 0$. In this case, \eqref{jjoijkq} implies that
	\begin{equation}\label{dgakdhkaa}
		\E_{\ell}(R_{\D_n,\ell})\le \E_{\ell-1}(R_{\D_n,\ell-1})-\frac{c(p)}{\|g_p\|_{TV}^2}\cdot [\E_{\ell-1}\left({R_{\D_n,\ell-1}}\right)]^2
	\end{equation}
	for all $1\le\ell\le \L_0$ because $\E_{t_n}\left(\max{\{R_{\D_n,\ell-1},0\}}\right)\ge \E_{\ell-1}(R_{\D_n,\ell-1})>0$. When $\ell=1$, by \eqref{dgakdhkaa} it is easy to know that $\E_{1}(R_{\D_n,1})\le c(p)\frac{\|g_p\|_{TV}^2}{4}$. Using the above initial condition and \eqref{dgakdhkaa} again, we also have, 	by mathematical induction,
	\begin{equation}\label{iaudfghbukajk122}
		\E_{\L_0}{(R_{\D_n,\L_0})} \le c(p)\cdot \frac{\|g_p\|^2_{TV}}{\lfloor\log_2{t_n}\rfloor+4},
	\end{equation}
	
	In conclusion, \eqref{dhzuadhnilq} follows \eqref{poqaqasd} and \eqref{iaudfghbukajk122}, which implies that inequality \eqref{regpart2dd} is also true. By  \eqref{Mainlemmaformula1} and following Parts I and II in the proof of Theorem \ref{Theoremhusidhuioh}, it is easy to know that $\hat{m}^r_{n,\L}(x)$ is mean squared consistent. Then following the same arguments in Part III of that proof, we can also prove
	\begin{equation}\label{ODTrandomfinal1beforeprun}
		\E_{\mathcal{D}_{n},\Xi_{t_n-1}} \int|m^r_{n,\L}(x)-m(x)|^2d\mu(x)\to 0 \ \text{as}\ n\to 0.
	\end{equation}
	Finally, by   Jensen's inequality, we have
	$$
	\E_{\mathcal{D}_{n},\Xi_{t_n-1}} \int|m_{ODRF,B, n}(x)-m(x)|^2d\mu(x)\le \E_{\mathcal{D}_{n},\Xi_{t_n}} \int|m^r_{n,\L}(x)-m(x)|^2d\mu(x).
	$$
	This shows that \eqref{ODTrandomfinal1beforeprun} also holds for  the estimator  $m_{ODRF,B,n}(x),x\in[0,1]^p$ obtained by ODRF.
	\hfill\(\Box\)

\section{Proof of Theorem \ref{full tree consistency}}
 As specified in paper, we only need to  bound $Part_{n,1},Part_{n,2}$ that are defined in the proof sketch. We will finish the corresponding arguments in four steps. During the first step, we show that $Part_{n,1}\to 0$ and in  next three steps we aim to prove $Part_{n,2}\to 0$.

\subsection{Step 1}
The term $Part_{n,1}$ is bounded as follows. Let $Z_{j,n}:= m^{r,b}_{n,a_n,a_n,\L}(X)$. Then, we have two observations below. Firstly, $Z_{1,n},\ldots, Z_{B_n, n}$ are i.i.d. given $\D_n$ and $X$. Secondly, $|Z_{1,n}|\le \max_{1\le i\le n}{|\varepsilon_i|}+\|m\|_\infty$, where $\|\cdot\|_\infty$ is the supremum norm. The above two pieces of observations lead that
\begin{align}
    Part_{n,1}&= \E \E[\left(m_{ODRF,B_n,n}(X)-\E_{\Theta}(m^{r}_{a_n}(X))\right)^2|\D_n, X]\nonumber\\
    &= \E\E\left[ \left(\frac{1}{B_n}\sum_{b=1}^{B_n} (Z_{b,n}-\E_{\Theta_b}(Z_{b,n}))\right)^2| \D_n, X \right] \nonumber\\
    &= \frac{1}{B_n}\cdot \E\E[(Z_{1,n}-\E_{\Theta_1}(Z_{1,n}))^2|\D_n, X]\nonumber\\
    &\le \frac{4}{B_n}\cdot \left[\E(\max_{1\le i\le n}{|\varepsilon_i^2|}) +\|m\|_\infty^2\right]. \label{shhbsJNber}
\end{align}

By $\E(\max_{1\le i\le n}{|\varepsilon_i^2|})\le c\cdot \ln n$ and \eqref{shhbsJNber}, we know
$$
Part_{n,1}\le \frac{4}{B_n}\cdot (\ln n+\|m\|_\infty^2),
$$
which goes to $0$ as $n\to\infty$ under our assumptions.

\subsection{Step 2}
The second step is to show Lemma \ref{key general function} is true. For completeness, we restate this result below.

\

\noindent \textbf{Lemma 4.4} Let $ \bm{A}_*\subseteq [0,1]^p$ be a  node satisfying $\Delta_{\bm{A}_*}(\theta,s)\equiv 0$ for any $\theta$ and $s$. Then, the regression
function $m(x)=\E(Y|X=x),x\in [0,1]^p$ is constant on $\bm{A}_*$.

\begin{proof}
Without loss of generality, we assume that $m(x)$ is always positive on $\bm{A}_*$. Otherwise, one can replace $m(x)$ by the function $m(x)+\sup_{x\in [0,1]^p}{|m(x)|}+1$. For any Lebesgue measurable set $A\subseteq \bm{A}_*$, we define a function
    $$
        G(A):=\frac{\int_A{m(x)dx}}{\int_{\bm{A}_*}{m(x)dx}}.
    $$
    Since $\int_{\bm{A}_*}{m(x)dx}>0$, $G(A)$ is well defined. It is easy to check that $G(A)$ is a probability measure on the measurable space $(\bm{A}_*, \mathcal{B}({\bm{A}_*}))$, where  $\mathcal{B}({\bm{A}_*})$ is the class of Lebesgue measurable sets contained in
    $\bm{A}_*$. Next we calculate $\int_A{m(x)dx}$ for node $A$ which takes the form
    $\{x\in \bm{A}_*: \theta^T x\le s\}$. Such calculation is based on the condition that $\Delta_{\bm{A}_*}(\theta,s)\equiv 0$ for any $\theta$ and $s$.
    
    Let $X_{\bm{A}_*}$ be the restriction of $X$ on $\bm{A}_*$. Namely, we have
    $$
    \P(X_{\bm{A}_*}\in A)=\frac{\P(X\in A)}{\P(X\in \bm{A}_*)}, \ \ \forall A\in \mathcal{B}({\bm{A}_*}).
    $$
    By some simple analysis, we can verify
    \begin{align*}
        Var(Y|X\in \bm{A}_*)&=Var(m(X)|X\in \bm{A}_*)= Var(X_{\bm{A}_*}),\\
        \P(\theta^TX\le s|X\in \bm{A}_*)&=\P(\theta^T X_{\bm{A}_*}\le s),\\
        Var(Y|X\in \bm{A}_*, \theta^T X_{\bm{A}_*}\le s)&=Var(m(X)|X\in \bm{A}_*, \theta^T X_{\bm{A}_*}\le s)\\
        &=Var(m(X_{\bm{A}_*})|\theta^T X_{\bm{A}_*}\le s).
    \end{align*}
    Using the above equations, the rule of population CART can be calculated in the following way.
    \begin{align*}
        \Delta_{\bm{A}_*}(\theta,s):&= Var(Y|X\in A)- \P(\theta^TX\le s|X\in \bm{A}_*)Var(Y|X\in A, \theta^TX\le s) \\
        &- \P(\theta^TX > s|X\in \bm{A}_*)Var(Y|X\in A, \theta^TX> s)\\
        &= Var(m(X_{\bm{A}_*}))-\P(\theta^T X_{\bm{A}_*}\le s)Var(m(X_{\bm{A}_*})|\theta^TX_{\bm{A}_*}\le s)\\
        &- \P(\theta^T X_{\bm{A}_*}> s)Var(m(X_{\bm{A}_*})|\theta^TX_{\bm{A}_*}> s)\\
        &= Var(m(X_{\bm{A}_*}))- \E(Var(m(X_{\bm{A}_*})|\theta^TX_{\bm{A}_*}\le s))\\
        &= Var(\E(m(X_{\bm{A}_*})|\theta^TX_{\bm{A}_*}\le s)),
    \end{align*}
where in the last line we use the law of total variance. Note that $\E(m(X_{\bm{A}_*})|\theta^TX_{\bm{A}_*}\le s)$ is a Bernoulli random variable, whose p.d.f is
\begin{equation*}
  \E(m(X_{\bm{A}_*})|\theta^TX_{\bm{A}_*}\le s) =
    \begin{cases}
      \frac{1}{S((\bm{A}_*)_{\theta,s}^+)}\int_{(\bm{A}_*)_{\theta,s}^+} {m(x)dx}, & \P(\theta^T X_{\bm{A}_*}\leq s )\\
    \frac{1}{S((\bm{A}_*)_{\theta,s}^-)}\int_{(\bm{A}_*)_{\theta,s}^-} {m(x)dx}, & \P(\theta^T X_{\bm{A}_*}> s ),
    \end{cases}
\end{equation*}
    where $S(A)$ denotes the Lebesgue measure of any $A\in \mathcal{B}({\bm{A}_*})$ and $(\bm{A}_*)_{\theta,s}^+, (\bm{A}_*)_{\theta,s}^-$ are two daughters of $\bm{A}_*$ (see definitions in Section \ref{sec:notations}). First, we consider the case where  $0<\P(\theta^T X_{\bm{A}_*}\leq s )<1$. Recall our condition that $\Delta_{\bm{A}_*}(\theta,s)\equiv 0$ for any $\theta$ and $s$. Such a condition implies that in this case, we have
$$
    \frac{1}{S((\bm{A}_*)_{\theta,s}^+)}\int_{(\bm{A}_*)_{\theta,s}^+} {m(x)dx}= \frac{1}{S((\bm{A}_*)_{\theta,s}^-)}\int_{(\bm{A}_*)_{\theta,s}^-} {m(x)dx},
$$
which implies
\begin{equation}\label{HKJHKJHberno}
    \int_{(\bm{A}_*)_{\theta,s}^+} {m(x)dx}= S((\bm{A}_*)_{\theta,s}^+)\bar{X}_{\bm{A}_*}, \ \
    \bar{X}_{\bm{A}_*}:= \frac{1}{S(\bm{A}_*)} \int_{\bm{A}_*}{m(x)dx},
\end{equation}
where $\bar{X}_{\bm{A}_*}$ denotes the average of $m(x)$ in $\bm{A}_*$. If $\P(\theta^T X_{\bm{A}_*}\leq s )=1$, then $(\bm{A}_*)_{\theta,s}^+=\bm{A}_*$ almost surely in Lebesgue measure. It is obvious that  \eqref{HKJHKJHberno} holds in this second case. With the similar argument, we know \eqref{HKJHKJHberno} is also true if $\P(\theta^T X_{\bm{A}_*}\leq s )=0$. In conclusion, \eqref{HKJHKJHberno} holds for whatever  value of $\P(\theta^T X_{\bm{A}_*}\leq s )$.

Suppose $Z$ is a random vector defined on $\bm{A}_*$ whose law is $G(A), A\in \mathcal{B}({\bm{A}_*})$.  For any $\theta\in\R^p$, $\theta^TZ$ is a random variable with the distribution function
$$
F_{\theta^TZ}(s)=\P(\{ Z\in \bm{A}_*: \theta^TZ\leq s\})=\frac{S((\bm{A}_*)_{\theta,s}^+)}{S(\bm{A}_*)}
$$
by using inequality  \eqref{HKJHKJHberno}. Suppose $U$ is another random vector defined in $\bm{A}_*$, which has a uniform distribution.  Then, $\theta^TU$ is a random variable with the distribution function
$$
F_{\theta^TU}(s)=\P(\{ U\in \bm{A}_*: \theta^TU\leq s\})=\frac{S((\bm{A}_*)_{\theta,s}^+)}{S(\bm{A}_*)}
$$
for any $\theta\in\R^p$ and $s\in\R$. Therefore, we have
$$\E(e^{i\theta^TU})=\E(e^{i\theta^TZ}), \forall \theta\in\R^p.$$
In other words, the characteristic functions of $U$ and $Z$ are the same, indicating that $Z$ has the same distribution with $U$, see for example Theorem 1.6 in \cite{shao2003mathematical}. Thus, for any  $A\in \mathcal{B}({\bm{A}_*})$,
$$
  G(A)=\P(U\in A)= \frac{S(A)}{S(\bm{A}_*)}=\frac{\int_A{m(x)dx}}{\int_{\bm{A}_*}{m(x)dx}}.
$$
The above inequality indicates that $\frac{1}{\int_{\bm{A}_*}{m(x)dx}}\cdot m(x), x\in \bm{A}_*$ is a density function of $U$. In conclusion,
\begin{equation}\label{dhbjkdkjasnjansbern}
    m(x)= \int_{\bm{A}_*}{m(x)dx}\cdot \frac{1}{S(\bm{A}_*)}, \ \ \forall x\in \bm{A}_*\ \ a.s..
\end{equation}
Since $m(x)$ is assumed to be continuous in $[0,1]^p$, \eqref{dhbjkdkjasnjansbern} also holds for any $x\in \bm{A}_*$. 	
\end{proof}

\subsection{Step 3}

The third step is to show that the cuts based on $\Delta_A^*(\theta,s)$ and $\Delta_A(\theta,s)$  are close to each other.  We refer to pages 12 \& 14 in \cite{scornet2015consistency} for the notations below. In the following analysis, $d=(\theta,s)$ is used to denote a cut corresponding with the hyperplane $\theta^Tx=s$, where $\theta\in \Theta^p, $ and $ s\in [-\sqrt{p},\sqrt{p}]$. For any $x\in [0,1]^p$, we call $\A_k(x)$ the class of all possible $k\ge 1$
consecutive cuts used to construct the convex polytope containing $x$. In other words, the above
mentioned polytope is obtained after  a sequence of cuts $\bm{d}_k=(d_1,\ldots,d_k)$. For any $\bm{d}_k\in \A_k(x)$, let $A(x,\bm{d}_k)$ be the polytope containing $x$ that is built by using $\bm{d}_k$. Then, the distance between two cuts $\bm{d}_k$ and $\bm{d}_k'$ is defined by
$$
\|\bm{d}_k-\bm{d}_k'\|_\infty:= \sup_{1\le j\le k}\max\{ \|\theta_j-\theta_j'\|_2,|s-s'|\}.
$$
For any $x\in [0,1]^p$ and $\bm{d}_k\in \A_k(x)$, we define rule
\begin{align*}
  &\Delta_{n,k}(x,\bm{d}_k): = \frac{1}{N(A(x,\bm{d}_{k-1}))}\sum_{i=1}^n{(Y_i-\bar{Y}_{A(x,\bm{d}_{k-1})})^2\mathbb{I}(X_i\in A(x,\bm{d}_{k-1}))}\\
  &- \frac{1}{N(A(x,\bm{d}_{k-1}))}\sum_{i=1}^n{(Y_i-\bar{Y}_{A_L(x,\bm{d}_{k-1})}\mathbb{I}(X_i\in A_L(x,\bm{d}_{k-1})))^2\mathbb{I}(X_i\in A_L(x,\bm{d}_{k-1}))}\\
  &- \frac{1}{N(A(x,\bm{d}_{k-1}))}\sum_{i=1}^n{(Y_i-\bar{Y}_{A_R(x,\bm{d}_{k-1})}\mathbb{I}(X_i\in A_L(x,\bm{d}_{k-1})))^2\mathbb{I}(X_i\in A_R(x,\bm{d}_{k-1}))},
\end{align*}
where  $A_L(x,\bm{d}_{k-1}):= A(x,\bm{d}_{k-1})\cap \{z:\theta_k^Tz\le s_k\}$ and  $A_R(x,\bm{d}_{k-1}):= A(x,\bm{d}_{k-1})\cap \{z:\theta_k^Tz > s_k\}$ are two daughters of $A(x,\bm{d}_{k-1})$. Actually, $ \Delta_{n,k}(x,\bm{d}_k)$ is the CART rule which will be used to find the best cut $\bm{d}_k$ in the polytope $A(x,\bm{d}_{k-1})$. Similar to \cite{scornet2015consistency}, $A^\xi(x,\bm{d}_{k-1})\subseteq \A_{k-1}(x)$ denotes the set of all $\bm{d}_{k-1}$ such that $A(x,\bm{d}_{k-1})$ contains a hypercube of edge length $\xi>0$.
Finally, define $\bar{A}^\xi(x):=\{\bm{d}_k:\bm{d}_{k-1}\in A^\xi(x,\bm{d}_{k-1})$ that is equipped with the norm $\|\bm{d}_k\|_\infty$.

\begin{lemma}\label{equicontinuous lemma}
    Fix $x\in [0,1]^p$, $k\in \mathbb{Z}_+$ and suppose $\xi>0$. Then $\Delta_{n,k}(x,\cdot)$ is stochastically equicontinuous on $\bar{\A}_k^\xi(x)$; that is,
    for all $\alpha,\rho>0$, there exists $\delta>0$ such that
    \begin{equation}\label{aGBJHKkjajdnadjknBernn}
     \lim_{n\to\infty} \P\Big( \sup_{\substack{\|\bm{d}_k-\bm{d}_k'\|_\infty\le\delta \\ \bm{d}_k,\bm{d}_k'\in\bar{\A}_k^\xi(x)}} |\Delta_{n,k}(x,\bm{d}_k)-\Delta_{n,k}(x,\bm{d}_k')|>\alpha  \Big)\le \rho.
    \end{equation}
\end{lemma}

\begin{proof}
   The proof is long and technical. We defer it to Section \ref{red}.
\end{proof}

\subsection{Step 4}

Now we give the proof of $L^2$ consistency for $\E_{\Theta}(m^{r}_{a_n}(X))$. Suppose $A_n(X,\Theta)$ is one of terminal nodes of fully grown random ODT $T_{\D_n(\mathfrak{G}),a_n,a_n,\L}^r$ that contains the data point $X$. By Lemma \ref{key general function}, Lemma \ref{equicontinuous lemma},  Lemma 3 in \cite{scornet2015consistency} which also works for ODRF and the proof of Proposition 2 in \cite{scornet2015consistency}, we know for all $\rho,\xi>0$, there exists $N\in\mathbb{Z}_+$ such that $\forall n>N$,
\begin{equation}\label{final equantion}
  \P\left( \sup_{x,x'\in A_n(X,\Theta)}{|m(x)-m(x')|\le \xi}\right)\ge 1-\rho.
\end{equation}
We will use inequality \eqref{final equantion}  to bound the approximation error of $\E_{\Theta}(m^{r}_{a_n}(X))$. Following arguments about the consistency of fully grown RF, which is given  in Section 5.4 in  \cite{scornet2015consistency}, all their results can be adopted here except bounding 
$$
I_n':= \E\Big[  \sum_{\substack{i,j\\ i\neq j}}{\mathbb{I}_{X \Theta X_i}\mathbb{I}_{X \Theta' X_j}(Y_i-m(X_i))(Y_j-m(X_j)}\Big],
$$
which occurs in the estimation error  of $\E_{\Theta}(m^{r}_{a_n}(X))$. If assumption $A.2$ holds, there is nothing new. Here, we only complete the proof if assumption $A.1$ is true. In the case of ODRF, the technique of \textit{layered nearest neighbors} in \cite{scornet2015consistency} can not be applied. For example, consider a simple case where both $X_i=(0,0,\ldots,0)$ and $X=(1,1,\ldots,1)$ locate in the same partition. When $n\ge 3$, $X_i$ can never be the layered nearest neighbor of $X$ with probability $1$. Therefore, it is possible that for any $i=1,\ldots,n$, $X$ is located in the partition that contains point $X_i$. This indicates summands in $I_n'$ for ODRF may not be sparse anymore. Meanwhile, more careful calculations are required in order to bound $I_n'$. Let $\D_n(\Theta)\subseteq\{X_1,X_2,\ldots,X_n\}$ be the sampled data given $\Theta$ and $\varepsilon_i=Y_i-m(X_i),$  $i=1,\ldots,n$. Then,
\begin{align*}
    &I_n'=  \E\Big[  \sum_{\substack{i,j\\ i\neq j}}{\mathbb{I}_{X \Theta X_i}\mathbb{I}_{X \Theta' X_j}\mathbb{I}_{X_i\in \D_n(\Theta)}\mathbb{I}_{X_j\in \D_n(\Theta')}\varepsilon_i\varepsilon_j}\Big]\\
    &= \E\Big[  \sum_{\substack{i,j\\ i\neq j}}{\mathbb{I}_{X_i\in \D_n(\Theta)}\mathbb{I}_{X_j\in \D_n(\Theta')}\varepsilon_i\varepsilon_j\E(\mathbb{I}_{X \Theta X_i}\mathbb{I}_{X \Theta' X_j}|X,\Theta,\Theta',X_1,\ldots,X_n,Y_i,Y_j)}\Big]\\
    &=\E\Big[  \sum_{\substack{i,j\\ i\neq j}}{\mathbb{I}_{X_i\in \D_n(\Theta)}\mathbb{I}_{X_j\in \D_n(\Theta')}\varepsilon_i\varepsilon_j\psi_{i,j}(Y_i,Y_j)}\Big]\\
    &= \E\Big[  \sum_{\substack{i,j\\ i\neq j}}{\mathbb{I}_{X_i\in \D_n(\Theta)}\mathbb{I}_{X_j\in \D_n(\Theta')}\varepsilon_i\varepsilon_j\psi_{i,j}}\Big]+ \E\Big[  \sum_{\substack{i,j\\ i\neq j}}{\mathbb{I}_{X_i\in \D_n(\Theta)}\mathbb{I}_{X_j\in \D_n(\Theta')}\varepsilon_i\varepsilon_j(\psi_{i,j}(i,j)-\psi_{i,j})}\Big]\\
    &=I_{n,1}'+I_{n,2}'.
\end{align*}

Conditional on $X,\Theta, \Theta',X_1,\ldots X_n$, we know $I_{n,1}'=0$ because $\varepsilon_i$ and $\varepsilon_j$ are independent when these  variables are given. Let $w_{i,j}:=\psi_{i,j}(i,j)-\psi_{i,j}$. By the law of total expectation and Jensen's inequality, 
\begin{align}
    |I_{n,2}'|^2 &\le \E\Big[| \E( \sum_{\substack{i,j\\ i\neq j}}{\varepsilon_i\varepsilon_j\mathbb{I}_{X_i\in \D_n(\Theta)}\mathbb{I}_{X_j\in \D_n(\Theta')}\cdot   w_{i,j} }\mathbb{I}_{X_i\in \D_n(\Theta)}\mathbb{I}_{X_j\in \D_n(\Theta')}|\Theta,\Theta')|^2\Big]\nonumber\\
    &\le  \E\Big[ \E( \sum_{\substack{X_i\in \D(\Theta),X_j\in\D(\Theta')\\ i\neq j}}{\varepsilon_i^2\varepsilon_j^2| \Theta,\Theta')\cdot \E( \sum_{\substack{i\in \D(\Theta),j\in\D(\Theta')\\ i\neq j}} w_{i,j}^2}|\Theta,\Theta')\Big]. \label{dhjhjBer}
\end{align}
Then, we focus on bounding the two conditional expectations above. Given realizations $\Theta=\vartheta$ and $ \Theta'=\vartheta'$, it is known that both $\D(\vartheta) $ and $ \D(\vartheta')$ are determined. Without loss of generality, we can assume  $\D(\vartheta)=\D(\vartheta')=\{X_1,\ldots, X_{a_n}\}$ below. Note that
\begin{equation}\label{bdhbfbjJNKNbnfnmBer}
    \E( \sum_{\substack{X_i\in \D(\vartheta),X_j\in\D(\vartheta')\\ i\neq j}}{\varepsilon_i^2\varepsilon_j^2| \Theta=\vartheta,\Theta'=\vartheta') }= \E( \sum_{\substack{i,j=1\\
    i\neq j}}^{a_n}{\varepsilon_i^2\varepsilon_j^2}) 
    \le a_n^2.
\end{equation}
By \eqref{dhjhjBer} and \eqref{bdhbfbjJNKNbnfnmBer}, 
\begin{equation}\label{hbsbfkjJKJber}
    |I_{n,2}'|^2\le a_n^2\cdot \E_{\Theta,\Theta'} \E( \sum_{\substack{i\in \D(\Theta),j\in\D(\Theta')\\ i\neq j}} w_{i,j}^2|\Theta,\Theta')= a_n^2\cdot \E( \sum_{\substack{i\in \D(\Theta),j\in\D(\Theta')\\ i\neq j}} w_{i,j}^2).
\end{equation}
For any realizations $\vartheta $ and $ \vartheta'$, the number of pairs in $\{(i,j): i\neq j, i\in\D(\vartheta), j\in\D(\vartheta')\}$ does not exceed $a_n^2$. According to this fact and assumption A.1, \eqref{hbsbfkjJKJber} implies 
$$
 |I_{n,2}'|^2\le a_n^4\cdot \E(\max_{\substack{i,j\\i\neq j}}{w_{i,j}^2})\le \frac{c^2}{a_n^{2\delta}},
$$
which converges to $0$ as $a_n\to\infty$. This completes the proof. \hfill\(\Box\)


\section{Proof of Lemma  \ref{equicontinuous lemma}}\label{red}

 We will first consider a simple case with $k=1 $ and $ p=2$. Then, we will find that the arguments for the other cases are similar to this simple case. Our goal is to choose a $\delta>0$ such that \eqref{aGBJHKkjajdnadjknBernn} holds. Now we have two cuts denoted by $\bm{d}_1=(\theta_1,s_1)$ and $\bm{d}_1'=(\theta_2,s_2)$ satisfying $\max\{ \|\theta_1-\theta_2\|_2,|s_1-s_2|\}\le \delta$. Denote by $\mathfrak{R}(\bm{d}_1,\bm{d}_1',\delta)$ the rectangle with the smallest area which contains points lying on $\theta_1^Tz=s_1, z\in [0,1]$ or $\theta_2^Tz=s_2, z\in [0,1]$. Then, it is not difficult to see that
    \begin{equation}\label{Areadelta}
        Area(\delta):=\sup\limits_{{ \bm{d}_1, \bm{d}_1',\|\bm{d}_1-\bm{d}_1'\|_\infty \le \delta}}{S(\mathfrak{R}(\bm{d}_1,\bm{d}_1',\delta))}\to 0
    \end{equation}
     as $\delta\to 0$.  In this simple case, $\bar{\A}_1^\xi(x)=[0,1]^p$ for any $x\in \bar{\A}_1^\xi(x)$. To prepare our arguments, we need the following three preliminary results.
    
    First, there exists $N_1\in\mathbb{Z}_+$ and $ c(\rho)>0$ such that for all $n>N_1$,
    \begin{equation}\label{eq1}
        \max_{1\le i\le n}{|\varepsilon_i|}\leq c(\rho)\sqrt{\ln n}
    \end{equation}
    holds with probability at least $1-\rho$.
    Second, denote  $\mathscr{F}$ by the class of all subsets of $[0,1]^p$. Note that there are at most $n^2$ sets taking the form $\{i:X_i\in F\}$ for $F\in \mathscr{F}$. Therefore, for any $\delta>0$ there exists $N_2(\delta)\in\mathbb{Z}_+$ such that for all $n>N_2(\delta)$ and $F\in\mathscr{F}$ satisfying $N(F)>\sqrt{n}$ such that
    \begin{equation}\label{eq2}
        \left|\frac{1}{N(F)}\sum_{i:X_i\in F}{\varepsilon_i}\right|\leq \frac{\alpha}{4}\sqrt{Area(\delta)}
    \end{equation}
    holds with probability at least $1-\rho$.
    Third, we prove a uniform error bound for approximating $\P(X\in A(x,\bm{d}_k))$ through the  empirical process method. Define a class of sets
    $$
       \mathscr{F}_{k_1,k_2}:=\Bigg\{[0,1]^p\bigcap_{j=1}^{k_1}\{z: \theta_{j,1}^Tz\le s_{j,1} \}\bigcap_{j=1}^{k_2}\{z:\theta_{j,2}^Tz>s_{j,2}\}: \theta_{j,1},\theta_{j,2}\in \R^p, s_{j,1},s_{j,2}\in \R \Bigg\},
    $$
    where $k_1,k_2\in \mathbb{Z}_+\cup \{0\}$. Then, we know
    $$
       A(x,\bm{d}_k)\in \mathscr{F}_{k}:= \bigcup_{\substack{k_1,k_2\in\mathbb{Z}_+\cup\{0\}\\k_1+k_2=k}} \mathscr{F}_{k_1,k_2}
    $$
    for any $x\in [0,1]^p$ and any cut $\bm{d}_k$. Next we bound the uniform error of approximating expectation of the indicator functions in $\mathscr{G}_k:= \{\mathbb{I}(z\in F): F\in \mathscr{F}_{k}\}.$ In the first step, we show $\mathscr{G}_k$ is a VC class. By Lemma 9.12 $(i)$ in \cite{kosorok2008introduction}, we  thus know that either $\mathscr{F}_{1,0}$ or $\mathscr{F}_{0,1}$ has VC dimension $p+1$. According to Lemma 9.7 in \cite{kosorok2008introduction}, we know $\mathscr{F}_{k}$ is a VC class with dimension no larger than $k(k+1)(p+1)$. Then,
    by following the standard arguments in Example 4.8 in \cite{sen2018gentle}, we know
    \begin{equation}\label{eq3.1}
        \E\left( \sup_{g\in \mathscr{G}_k}{|(\mathbb{P}_n-P)g|}\right)\leq c\sqrt{\frac{k(k+1)(p+1)}{n}},
    \end{equation}
    where $\mathbb{P}_n(g)=\frac{1}{n}\sum_{j=1}^n{g(X_j)}$ and $P(g)=\E(g(X))$ stand for the operators of empirical and population expectation respectively. On the other hand, it is easy to check that $\sup_{g\in \mathscr{G}_k}{|(\mathbb{P}_n-P)g|}$
    is a  function of $(X_1,\ldots, X_n)$ and has the bounded differences property for constants $1/n$'s (see page 56 in \cite{boucheron2013concentration}). Therefore, the application of McDiarmid’s inequality leads that for any $t>0$,
    \begin{equation}\label{eq3.2}
        \P\left(  \sup_{g\in \mathscr{G}_k}{|(\mathbb{P}_n-P)g|}- \E\left( \sup_{g\in \mathscr{G}_k}{|(\mathbb{P}_n-P)g|}\right)>t\right)\le e^{-2nt^2}.
    \end{equation}
   Finally, the combination of \eqref{eq3.1} and \eqref{eq3.2} implies that with probability larger than $1-\rho$,
   $$
      \sup_{g\in \mathscr{G}_k}{|(\mathbb{P}_n-P)g|}\leq c_1\sqrt{\frac{k(k+1)(p+1)}{n}}+
      c_2\sqrt{\frac{1}{2n}\log\frac{1}{\rho}}
   $$
   for some $c_1,c_2>0$. In conclusion, for any $\delta>0$ there exists $N_3(\delta)\in\mathbb{Z}_+$ such that for all
   $n>N_3(\delta)$ and all $A(x,\bm{d}_k)$,
   \begin{equation}\label{eq3}
       (S(A(x,\bm{d}_k))-Area^2(\delta))n\le N(A(x,\bm{d}_k))\le (S(A(x,\bm{d}_k))+Area^2(\delta))n
   \end{equation}
   holds with probability at least $1-\rho$.

   The following analysis of the simple case is carried out on the event where equations \eqref{eq1}, \eqref{eq2} and \eqref{eq3}  all hold. We divide all cases of $\bm{d}_1$ (or $\bm{d}_1'$) into two groups:
   \begin{align*}
       Area_1([0,1]^p): &= \{(\theta,s): S(A_{\theta,s}^+)\le Area(\delta)\ \text{or}\ S(A_{\theta,s}^-)\le Area(\delta)\} \\
       Area_2([0,1]^p): &= \{(\theta,s): S(A_{\theta,s}^+)> Area(\delta)\ \text{and}\ S(A_{\theta,s}^-)> Area(\delta)\},
   \end{align*}
 where $ Area(\delta)$ is defined in \eqref{Areadelta}.
Let $\bm{d}_1=(\theta_1,s_1)$ and $\bm{d}'=(\theta_2,s_2)$. We have three cases of locations of $\bm{d}_1, \bm{d}_1'$:  (i) $\bm{d}_1, \bm{d}_1'\in Area_2([0,1]^p)$; (ii) $\bm{d}_1, \bm{d}_1'\in Area_1([0,1]^p)$; (iii) One of them is in $Area_1([0,1]^p)$ and the other is in $Area_2([0,1]^p)$. We only study the first case in the proof and the arguments of the other two cases are similar to the first one. In fact, there are two different situations in the case 1: (a). $\bm{d}_1$ does not intersect with $\bm{d}_1'$ in $[0,1]^p$; (b). $\bm{d}_1$ intersects with $\bm{d}_1'$ in $[0,1]^p$. The situation (a) is similar to the {first case} in the {proof of Lemma 2} in \cite{scornet2015consistency}. So we only give the proof for situation (b).
   \begin{figure}[ht]
       \centering
       \includegraphics[width=0.4\linewidth]{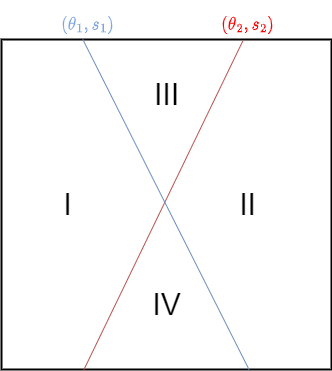}
       \caption{This is an example of the situation (b). Here, $\bm{d}_1=(\theta_1,s_1)$ divides $[0,1]^p$ into two parts where the left part $A_{L,1}= I\cup IV$ and the right part $A_{R,1}= II\cup III$, while cut $\bm{d}_1'=(\theta_2,s_2)$ divides $[0,1]^p$ into another two parts where the left one $A_{L,2}= I\cup III$ and the right one $A_{R,2}= II\cup IV$.}
       \label{fig:2}
   \end{figure}

    As illustrated in  Figure \ref{fig:2}, we can rewrite $\Delta_{n,1}(\theta_1,s_1)-\Delta_{n,1}(\theta_2,s_2)$ in the following way:
   \begin{align*}
      \Delta_{n,1}(\theta_1,s_1)-\Delta_{n,1}(\theta_2,s_2)&= \frac{1}{n}\sum_{i:X_i\in A_{L,1}}{(Y_i-\bar{Y}_{A_{L,1}})^2}+\frac{1}{n}\sum_{i:X_i\in A_{R,1}}{(Y_i-\bar{Y}_{A_{R,1}})^2}\\
        &-\frac{1}{n}\sum_{i:X_i\in A_{L,2}}{(Y_i-\bar{Y}_{A_{L,2}})^2}-\frac{1}{n}\sum_{i:X_i\in A_{R,2}}{(Y_i-\bar{Y}_{A_{R,2}})^2}\\
        &=\left[\frac{1}{n}\sum_{i:X_i\in I}{(Y_i-\bar{Y}_{A_{L,1}})^2}-\frac{1}{n}\sum_{i:X_i\in I}{(Y_i-\bar{Y}_{A_{L,2}})^2}\right]\\
        &+ \left[\frac{1}{n}\sum_{i:X_i\in II}{(Y_i-\bar{Y}_{A_{R,1}})^2}-\frac{1}{n}\sum_{i:X_i\in II}{(Y_i-\bar{Y}_{A_{R,2}})^2}\right]\\
        &+ \Bigg[\frac{1}{n}\sum_{i:X_i\in IV}{(Y_i-\bar{Y}_{A_{L,1}})^2}+\frac{1}{n}\sum_{i:X_i\in III}{(Y_i-\bar{Y}_{A_{R,1}})^2}\\
        &-\frac{1}{n}\sum_{i:X_i\in III}{(Y_i-\bar{Y}_{A_{L,2}})^2}-\frac{1}{n}\sum_{i:X_i\in IV}{(Y_i-\bar{Y}_{A_{R,2}})^2}\Bigg]\\
        &:=J_1+J_2+J_3.
   \end{align*}
   First, we bound $|J_1|$ in the following way.
   \begin{align*}
       |J_1|:&=\left| \frac{1}{n}\sum_{i:X_i\in I}{(Y_i-\bar{Y}_{A_{L,1}})^2}-\frac{1}{n}\sum_{i:X_i\in I}{(Y_i-\bar{Y}_{A_{L,2}})^2}\right|\\
       &=\left| (\bar{Y}_{A_{L,1}}-\bar{Y}_{A_{L,2}})\cdot \frac{2}{n}\sum_{i:X_i\in I}{\left(Y_i-\frac{\bar{Y}_{A_{L,1}}+\bar{Y}_{A_{L,2}}}{2}\right)}\right|\\
       &= |\bar{Y}_{A_{L,1}}-\bar{Y}_{A_{L,2}}| \cdot \left| \frac{2}{n}\sum_{i:X_i\in I}{\left(Y_i-\frac{\bar{Y}_{A_{L,1}}+\bar{Y}_{A_{L,2}}}{2}\right)} \right|\\
       &:= J_4\times J_5.
   \end{align*}
   Let us find an upper bound for $J_4$. Note that
   \begin{align*}
       |J_4|&= \left| \frac{1}{N(A_{L,1})}\sum_{i:X_i\in A_{L,1}}{Y_i}- \frac{1}{N(A_{L,2})}\sum_{i:X_i\in A_{L,2}}{Y_i}\right|\\
       &= \left| \frac{1}{N(A_{L,1})}\sum_{i:X_i\in I}{Y_i}+\frac{1}{N(A_{L,1})}\sum_{i:X_i\in IV}{Y_i} - \frac{1}{N(A_{L,2})}\sum_{i:X_i\in I}{Y_i}-\frac{1}{N(A_{L,2})}\sum_{i:X_i\in III}{Y_i}\right|\\
       &\leq \left| \frac{1}{N(A_{L,1})}\sum_{i:X_i\in I}{Y_i}-\frac{1}{N(A_{L,2})}\sum_{i:X_i\in I}{Y_i}\right|+\left| \frac{1}{N(A_{L,1})}\sum_{i:X_i\in IV}{Y_i}- \frac{1}{N(A_{L,2})}\sum_{i:X_i\in III}{Y_i}\right|\\
       &\le \left|1-\frac{N(A_{L,1})}{N(A_{L,2})} \right|\frac{1}{N(A_{L,1})}\left|\sum_{i:X_i\in I}{Y_i}\right|+\left(\frac{1}{N(A_{L,1})}+\frac{1}{N(A_{L,2})}\right)\left|\sum_{i:X_i\in III\cup IV}{Y_i}\right|\\
       &:=J_6+J_7.
   \end{align*}
   We bound $J_6$ by observing that $III\cup IV$ is contained in the rectangle $\mathfrak{R}(\bm{d}_1,\bm{d}_1',\delta)$, whose area is no larger than $Area(\delta):=\sup\limits_{{ \bm{d}_1, \bm{d}_1',\|\bm{d}_1-\bm{d}_1'\|_\infty \le \delta}}{S(\mathfrak{R}(\bm{d}_1,\bm{d}_1',\delta))}$. Meanwhile, we can always have $0<Area(\delta)<1$ by choosing $\delta$ small enough. Applying
   \eqref{eq3} for $\mathfrak{R}(\bm{d}_1,\bm{d}_1',\delta)$, the above argument implies that
   \begin{align}
       \left|1-\frac{N(A_{L,1})}{N(A_{L,2})} \right|&\le \left|\frac{N(III\cup IV)}{N(A_{L,2})}\right|\le \frac{Area(\delta)+Area^2(\delta)}{\sqrt{Area(\delta)}-Area^2(\delta)}\nonumber\\
       &\le 4\sqrt{Area(\delta)}. \label{dhjxahkdjlbern}
   \end{align}
   On the other hand, we also have
   \begin{align}
       N(I)&=N(A_{L,1})-N(IV) \nonumber\\
       &\ge n\sqrt{Area(\delta)}-n(Area(\delta)+Area^2(\delta))\nonumber\\
       &=n(\sqrt{Area(\delta)}-Area^2(\delta))\nonumber\\
       &\ge \sqrt{n},\label{fbcziofjwAPRDIPIBERN}
   \end{align}
   where the last line holds if we choose $\delta>0$ small enough. Therefore, the combination of \eqref{dhbjkdkjasnjansbern} and \eqref{fbcziofjwAPRDIPIBERN} implies that
   \begin{align}
       J_6&=\left|1-\frac{N(A_{L,1})}{N(A_{L,2})} \right|\frac{1}{N(A_{L,1})}\left|\sum_{i:X_i\in I}{Y_i}\right| = \left|1-\frac{N(A_{L,1})}{N(A_{L,2})} \right|\frac{N(I)}{N(A_{L,1})}\frac{1}{N(I)}\left|\sum_{i:X_i\in I}{Y_i}\right|\nonumber\\
       &\leq 4\sqrt{Area(\delta)}\cdot\left( \left|\frac{1}{N(I)}\sum_{i:X_i\in I}{\varepsilon_i}\right|+\|m\|_\infty\right)\nonumber\\
       &\leq 4\sqrt{Area(\delta)} (\alpha+\|m\|_\infty),\label{egfjhsidjiolberno}
   \end{align}
   where the last line follows from \eqref{eq2}.

   Next, we consider $J_7$. Since $N(A_{L,1}),N(A_{L,2})\ge \sqrt{n}$,  we only consider its first term without loss of generality:
   $$
   J_{7,1}:=\frac{1}{N(A_{L,1})}\left|\sum_{i:X_i\in III\cup IV}{Y_i}\right|.
   $$
   By $N(A_{L,1})\ge \sqrt{Area(\delta)}\cdot n$ and $N(IV)\le (Area(\delta)+Area^2(\delta))\cdot n$, we have 
   \begin{align}
       J_{7,1}&\le \frac{1}{N(A_{L,1})}\left| \sum_{i:X_i\in IV}{m(X_i)}\right|+\frac{1}{N(A_{L,1})} \left|\sum_{i:X_i\in IV}{\varepsilon_i}\right| \nonumber\\
       &\le \frac{N(IV)}{N(A_{L,1})}\|m\|_\infty+\frac{1}{N(A_{L,1})} \left|\sum_{i:X_i\in IV}{\varepsilon_i}\right| \nonumber\\
       &\le 2\sqrt{Area(\delta)}\cdot \|m\|_\infty+ \frac{1}{\sqrt{Area(\delta)}}\cdot \frac{1}{n} \left|\sum_{i:X_i\in IV}{\varepsilon_i}\right|.\label{dbjhsgfdiaObbBernou}
   \end{align}
   If $N(IV)\ge \sqrt{n}$, by \eqref{eq2} we have
   \begin{equation}\label{VjdiojalidjBernno}
       \frac{1}{n} \left|\sum_{i:X_i\in IV}{\varepsilon_i}\right| \le \frac{1}{N(IV)} \left|\sum_{i:X_i\in IV}{\varepsilon_i}\right|\le \frac{\alpha}{4}\sqrt{Area(\delta)}.
   \end{equation}
   If $N(IV)< \sqrt{n}$, by \eqref{eq1} we have
   \begin{equation}\label{GHbdjhsbJKHSNABBer}
       \frac{1}{n} \left|\sum_{i:X_i\in IV}{\varepsilon_i}\right| \le \frac{c(\rho)\sqrt{\log n}}{\sqrt{n}}.
   \end{equation}
   Therefore, by \eqref{dbjhsgfdiaObbBernou}, \eqref{VjdiojalidjBernno} and \eqref{GHbdjhsbJKHSNABBer}, there exits $\delta(\alpha)>0 $ and $ N_4(\delta(\alpha))\in\mathbb{Z}_+$ such that for all $n>N_4(\delta(\alpha))$, we have
   \begin{equation}\label{Bygbhjkgdber}
       J_7\le 4\sqrt{Area(\delta(\alpha))}\cdot \|m\|_\infty+ \frac{2}{\sqrt{Area(\delta(\alpha))}}\cdot \frac{1}{n} \left|\sum_{i:X_i\in IV}{\varepsilon_i}\right|\le \alpha.
   \end{equation}
Next,  we bound $J_5$. Note that
   \begin{align*}
       J_5 &:= \left| \frac{2}{n}\sum_{i:X_i\in I}{\left(Y_i-\frac{\bar{Y}_{A_{L,1}}+\bar{Y}_{A_{L,2}}}{2}\right)} \right|\\
       &\le  \frac{2}{n}\left| \sum_{i:X_i\in I}{(Y_i-\bar{Y}_{A_{L,1}})}\right|+ \frac{2}{n}\left| \sum_{i:X_i\in I}{(Y_i-\bar{Y}_{A_{L,2}})}\right|\\
           &:= J_{5,1}+J_{5,2}.
   \end{align*}
   Since $J_{5,1}$ is similar to $J_{5,2}$, we only need to analyze $J_{5,1}$. By some calculations, we have
   \begin{align*}
       J_{5,1}&= \frac{2}{n}\left| \sum_{i:X_i\in I}{(Y_i-\bar{Y}_{A_{L,1}})}+\sum_{i:X_i\in IV}{(Y_i-\bar{Y}_{A_{L,1}})}-\sum_{i:X_i\in IV}{(Y_i-\bar{Y}_{A_{L,1}})}\right|\\
       &= \frac{2}{n}\left| \sum_{i:X_i\in {A_{L,1}}}{(Y_i-\bar{Y}_{A_{L,1}})}-\sum_{i:X_i\in IV}{(Y_i-\bar{Y}_{A_{L,1}})}\right|\\
       &\le \frac{2}{n}\left| \sum_{i:X_i\in {A_{L,1}}}{(Y_i-\bar{Y}_{A_{L,1}})}\right| +\frac{2}{n}\left|\sum_{i:X_i\in IV}{(Y_i-\bar{Y}_{A_{L,1}})}\right|\\
       &=\frac{2}{n}\left|\sum_{i:X_i\in IV}{(Y_i-\bar{Y}_{A_{L,1}})}\right|\\
       &\leq \frac{2}{n}\left|\sum_{i:X_i\in IV}{Y_i}\right|+ \frac{2}{n}\left|\sum_{i:X_i\in IV}{\bar{Y}_{A_{L,1}}}\right|:=J_{5,1,1}+J_{5,1,2}.
   \end{align*}
   According to \eqref{VjdiojalidjBernno} and \eqref{GHbdjhsbJKHSNABBer}, there exits $\delta(\alpha)>0 $ and $ N_5(\delta(\alpha))\in\mathbb{Z}_+$ such that for all $n>N_5(\delta(\alpha))$,
   \begin{equation}\label{VHJBBjdsankjHBJbern}
       J_{5,1,1}\le \frac{\alpha}{4}.
   \end{equation}
   Next, we bound $J_{5,1,2}$. Since $S(A_{L,1})>Area(\delta)$, $N(A_{L,1})\ge (Area(\delta)-Area^2(\delta))n\ge \sqrt{n}$ whenever $n$ is larger than some $N_6(\delta)\in\mathbb{Z}_+$. Therefore, when $n>N_6(\delta)$, we have $|\bar{Y}_{A_{L,1}}|\le \|m\|_\infty+\alpha$. By using this result, if $n>N_6(\delta)$,
   \begin{equation}\label{GVYHBkxfvsBern}
       J_{5,1,2}=\frac{N(IV)}{n}|\bar{Y}_{A_{L,1}}|\le \frac{N(IV)}{n}(\|m\|_\infty+\alpha)\le Area(\delta)\cdot (\|m\|_\infty+\alpha).
   \end{equation}
   By \eqref{VHJBBjdsankjHBJbern} and \eqref{GVYHBkxfvsBern}, there exits $ \delta(\alpha)>0 $ with $ N_7(\delta(\alpha))\in\mathbb{Z}_+$ such that for all $n>N_7(\delta(\alpha))$, we have
   \begin{equation}\label{GHBJBvghdsvdGVHHVBer}
       J_5\le \alpha.
   \end{equation}

    Finally, the combination of  \eqref{egfjhsidjiolberno}, \eqref{Bygbhjkgdber} and \eqref{GHBJBvghdsvdGVHHVBer} leads that there exits $\delta(\alpha)>0 $ with $ N_8(\delta)\in\mathbb{Z}_+$ such that for all $n>N_8(\delta(\alpha))$,
    \begin{equation}\label{GHJBbdsjdbHJHBernou}
        |J_1|\le (4\sqrt{Area(\delta(\alpha))} (\alpha+\|m\|_\infty)+\alpha)\alpha\le \frac{\alpha}{3}.
    \end{equation}
    With the similar argument, we also know there exits $\delta>0 $ with $ N_9(\delta(\alpha))\in\mathbb{Z}_+$ such that for all $n>N_9(\delta(\alpha))$, we have
    \begin{equation}\label{GHJBbdsjdbHJHBernou2}
        |J_2|\le \frac{\alpha}{3}.
    \end{equation}

    Now we bound the last term $J_3$ by decomposing it into two parts:
    \begin{align*}
        J_3&= \Bigg[\frac{1}{n}\sum_{i:X_i\in IV}{(Y_i-\bar{Y}_{A_{L,1}})^2}-\frac{1}{n}\sum_{i:X_i\in IV}{(Y_i-\bar{Y}_{A_{R,2}})^2}\Bigg]\\
          &+\Bigg[\frac{1}{n}\sum_{i:X_i\in III}{(Y_i-\bar{Y}_{A_{R,1}})^2}
        -\frac{1}{n}\sum_{i:X_i\in III}{(Y_i-\bar{Y}_{A_{L,2}})^2}\Bigg]\\
        &:= J_{3,1}+J_{3,2}.
    \end{align*}
    By symmetry of $J_{3,1}$ and $J_{3,2}$, we only need to bound $J_{3,1}$. After some calculations, we know
    \begin{align*}
        |J_{3,1}|&:= \left|\frac{1}{n}\sum_{i:X_i\in IV}{(Y_i-\bar{Y}_{A_{L,1}})^2}-\frac{1}{n}\sum_{i:X_i\in IV}{(Y_i-\bar{Y}_{A_{R,2}})^2}\right|\\
        &= \frac{1}{n}\sum_{i:X_i\in IV}{ \left(|\bar{Y}_{A_{R,2}}-\bar{Y}_{A_{L,1}}| \cdot \left| Y_i-\frac{\bar{Y}_{A_{L,1}}+\bar{Y}_{A_{R,2}}}{2}\right|\right)}\\
        &\le |\bar{Y}_{A_{R,2}}-\bar{Y}_{A_{L,1}}|\cdot\left( \frac{1}{n}\left|\sum_{i:X_i\in IV}{Y_i}\right|+ \frac{N(IV)}{n}\left|\frac{\bar{Y}_{A_{L,1}}+\bar{Y}_{A_{R,2}}}{2}\right|\right).
    \end{align*}
   Recall that $|\bar{Y}_{A_{L,1}}|\le \|m\|_\infty+\alpha$ if $n>N_6(\delta)$ and $\frac{1}{n}\left|\sum_{i:X_i\in IV}{Y_i}\right|\le\frac{\alpha}{4}$ and $\frac{N(IV)}{n}\le Area(\delta)$. Therefore, there exits $\delta(\alpha)>0 $ with $ N_{10}(\delta(\alpha))\in\mathbb{Z}_+$ such that for all $n>N_{10}(\delta(\alpha))$, $|J_{3,1}|\le \frac{\alpha}{6}$ and
   \begin{equation}\label{GHJgbGBHBBer}
       |J_3|\le \frac{\alpha}{3}.
   \end{equation}

   Finally, the combination of \eqref{GHJBbdsjdbHJHBernou}, \eqref{GHJBbdsjdbHJHBernou2} and \eqref{GHJgbGBHBBer} implies that there exits $\delta(\alpha)>0, N_{11}\in\mathbb{Z}_+$ such that if $n>N_{11}$,
$$
|\Delta_{n,1}(\theta_1,s_1)-\Delta_{n,1}(\theta_2,s_2)|\le |J_1|+|J_2|+|J_3|\le \alpha
$$
with probability larger than $1-3\rho$.  This finishes the proof for the simple case $k=1 $and $ p=2$.


Next, we consider the case where $k>1$. Without loss of generality, we prove the equicontinuous of $\Delta_{n,k}(\cdot,\cdot)$ only for $k=4, p=2$.
Consider the decomposition of $\Delta_{n,4}(\bm{d}_4)-\Delta_{n,4}(\bm{d}_4')$:
\begin{align*}
    \Delta_{n,4}(\bm{d}_4)-\Delta_{n,4}(\bm{d}_4')&=\Delta_{n,4}(d_1,d_2,d_3,d_4)-\Delta_{n,4}(d_1,d_2,d_3,d_4')\\
    &+\Delta_{n,4}(d_1,d_2,d_3,d_4')-\Delta_{n,4}(d_1,d_2,d_3',d_4')\\
    &+\Delta_{n,4}(d_1,d_2,d_3',d_4')-\Delta_{n,4}(d_1,d_2',d_3',d_4')\\
    &+\Delta_{n,4}(d_1,d_2',d_3',d_4')-\Delta_{n,4}(d_1',d_2',d_3',d_4')\\
    &:=\sum_{j=1}^4{Part_j}.
\end{align*}
Then, $Part_j, j=1,\ldots,4$ can be divided into two groups depending on whether the fourth coordinates are the same. In $Part_1$ all corresponding cuts are same except $d_4\neq d_4'$, while $d_4=d_4'$ in each $Part_j, j=2,3,4$. Then, we can use the same method to bound $Part_1$, which was used in the case $k=1, p=2$. Meanwhile, by the similarity of $Part_j, j=2,3,4$, we only  analyze $Part_2$. Without loss of generality, we just consider a case of $Part_2$ shown in Fig \ref{fig:3}. To simplify notations, let $A_b:=\triangle ABC$ and $A_a:= \triangle ADE$. Therefore, we have $\triangle ABC_L=A_{b,L}, \triangle ABC_R=A_{b,R}$ and $\triangle ADE_L=A_{a,L}, \triangle ADE_R=A_{a,R}$. 
    \begin{figure}[ht]
       \centering
       \includegraphics[width=0.5\linewidth]{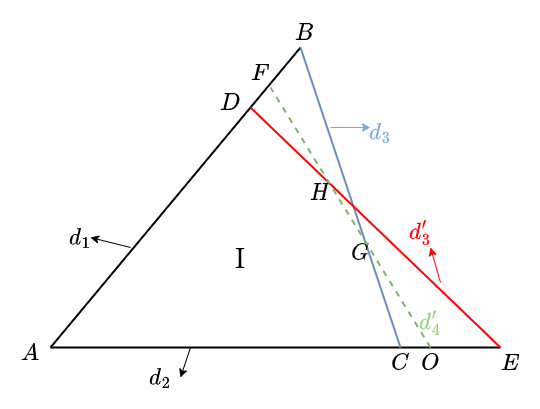}
       \caption{In this case, $A(x,(d_1,d_2,d_3))= \triangle ABC$ and $A(x,(d_1,d_2,d_3'))= \triangle ADE$. The cut $d_4'$ divides $\triangle ABC$ into two daughters, where $\triangle ABC_L=\Box AFGC$ and $\triangle ABC_R=\triangle GFB$. Meanwhile, the cut $d_4'$ divides $\triangle ADE$ into two daughters, where $\triangle ABC_L=\Box ADHO$ and $\triangle ABC_R=\triangle HOE$.}
       \label{fig:3}
   \end{figure}
   
Using the above notations, we have
   \begin{align}
       &\Delta_{n,4}(d_1,d_2,d_3,d_4')-\Delta_{n,4}(d_1,d_2,d_3',d_4')\nonumber\\
       &= \frac{1}{N(A_b)}\sum_{i:X_i\in A_b}{(Y_i-\bar{Y}_{A_b})^2}- \frac{1}{N(A_a)}\sum_{i:X_i\in A_a}{(Y_i-\bar{Y}_{A_a})^2}\nonumber\\
       &+\frac{1}{N(A_a)}\sum_{i:X_i\in A_{a,R}}{(Y_i-\bar{Y}_{A_{a,R}})^2}- \frac{1}{N(A_b)}\sum_{i:X_i\in A_{b,R}}{(Y_i-\bar{Y}_{A_{b,R}})^2}\nonumber\\
       &+\frac{1}{N(A_a)}\sum_{i:X_i\in A_{a,L}}{(Y_i-\bar{Y}_{A_{a,L}})^2}- \frac{1}{N(A_b)}\sum_{i:X_i\in A_{b,L}}{(Y_i-\bar{Y}_{A_{b,L}})^2}\nonumber\\
       &:=T_1+T_2+T_3. \label{hdasbkJHBNJKbnkjdfjklqwBernouli}
   \end{align}

Here, we assume $S(A_{b,L})>2 Area(\delta)$.  Otherwise, it is easy to prove the trivial case $S(A_{b,L})\le 2 Area(\delta)$. First, we bound $T_1$. By the assumption $\bm{d}_4\in \bar{A}^\xi_k(x)$, we know from \eqref{eq3} that $N(A_b)\ge 0.5\xi^2\cdot n$ with probability larger than $1-\rho$ when $n$ is large enough.  This means we can write $N(A_b)=n$ in order to simplify notation. With the same argument, we can also write $N(A_a)=n$. By symmetry of $ \frac{1}{N(A_b)}\sum_{i:X_i\in A_b}{(Y_i-\bar{Y}_{A_b})^2}$ and $\frac{1}{N(A_a)}\sum_{i:X_i\in A_a}{(Y_i-\bar{Y}_{A_a})^2}$, we only need to consider the case $T_1\ge 0$ and further assume $N(A_b)\ge N(A_a)$.  Therefore,
  \begin{align*}
    T_1 &\le \frac{1}{N(A_a)}\left[ \sum_{i:X_i\in I}{(Y_i-\bar{Y}_{A_b})^2}- \sum_{i:X_i\in I}{(Y_i-\bar{Y}_{A_a})^2}\right]
     \\
    &+ \frac{1}{N(A_a)}\left[ \sum_{i:X_i\in A_b\setminus I}{(Y_i-\bar{Y}_{A_b})^2}- \sum_{i:X_i\in A_a\setminus I}{(Y_i-\bar{Y}_{A_a})^2}\right]\\
     &:= T_{1,1}+T_{1,2},
  \end{align*}
  where $I$ denotes the pentagon $ADHGC$. Since $BC$ and $DE$ are contained in $\mathfrak{R}(d_3,d_3',\delta)$, we have $S(\triangle ADE)> Area(\delta)$. Note that
  $$
   T_{1,1}= \frac{1}{n}\sum_{i:X_i\in I}{(\bar{Y}_{A_a}-\bar{Y}_{A_b})\left(Y_i-\frac{\bar{Y}_{A_a}+\bar{Y}_{A_b}}{2}\right)}
  $$
  and $S(I)> Area(\delta)$. Thus, we can use the method which bounds $|J_1|$ above to find an upper bound of $T_{1,1}$. Next, we focus on $T_{1,2}$. By some calculations,
  \begin{align}
    T_{1,2}&\le \frac{2}{n}\sum_{i:X_i\in A_b\setminus I}{Y_i^2}+ \frac{2}{n}\sum_{i:X_i\in A_b\ I}{\bar{Y}_{A_b}^2} \nonumber\\
    &\le \frac{2}{n}\sum_{i:X_i\in A_b\ I}{(m(X_i)+\varepsilon_i)^2}+ 2\frac{N(A_b\setminus I)}{n}(\|m\|_\infty+\alpha)^2 \nonumber\\
    &\le 4\frac{N(A_b\setminus I)}{n}\|m\|^2_\infty+ \frac{4}{n}\cdot \sum_{i:X_i\in A_b\setminus I}{\varepsilon_i^2}+2\frac{N(A_b\setminus I)}{n}(\|m\|_\infty+\alpha)^2. \label{ghbhbBNKJNJvgvBer}
  \end{align}
At this point, we need three observations. By $\eqref{eq3}$, the first one is
 \begin{equation}\label{hbjsbfdcjksdBJKNKber}
    N(A_b\setminus I)/n\le 2Area(\delta)+Area^2(\delta)
 \end{equation}
 with probability larger than $1-\rho$ when $n$ is large. The second one relates to the chi-squared distribution. Recall that  $\mathscr{F}$ is the class of all subsets of $[0,1]^p$ and there are at most $n^2$ sets taking the form $\{i:X_i\in F\}$ for $F\in \mathscr{F}$. Since $\P(\chi^2(n)\ge 5n)\le \exp{(-n)}$, there exists $N_{12}\in\mathbb{Z}_+$ such that for any $F\in \mathscr{F}$ satisfying $N(F)\ge \sqrt{n}$ and $n>N_{12}$,
$$
  \frac{1}{N(F)}\sum_{i:X_i\in F}{\varepsilon_i^2}\le 5.
$$
with probability larger than $1-\rho$. The third observation is like \eqref{eq1}. There exists $N_1\in\mathbb{Z}_+, c(\rho)>0$ such that for all $n>N_1$,
 \begin{equation*}
        \max_{1\le i\le n}{|\varepsilon_i|^2}\leq c(\rho)\ln n
 \end{equation*}
with probability at least $1-\rho$. With the second and third observations, there exists $N_{13}\in\mathbb{Z}_+$ such that for all $n>N_{13}$,
\begin{equation}\label{GHBJBJBansjdnBJKbernnpu}
  \frac{4}{n}\cdot \sum_{i:X_i\in A_b\setminus I}{\varepsilon_i^2}\le \frac{\alpha}{6}
\end{equation}
with probability larger than $1-\rho$. By \eqref{ghbhbBNKJNJvgvBer}, \eqref{hbjsbfdcjksdBJKNKber} and \eqref{GHBJBJBansjdnBJKbernnpu}, there exits $\delta(\alpha)>0, N_{14}\in\mathbb{Z}_+$ such that if $n>N_{14}$,
\begin{equation}\label{babdJKkdjskajdsHJNNber}
  T_1\le \alpha
\end{equation}
with probability larger than $1-\rho$. With the same argument, there also exits $\delta(\alpha)>0, N_{15}\in\mathbb{Z}_+$ such that if $n>N_{15}$,
\begin{equation}\label{babdJKkdj2skajds2HJNNber}
  T_2\le \alpha
\end{equation}
with probability larger than $1-\rho$. Finally, consider $T_3$ by using the decomposition below.
\begin{align}
  T_3&:= \frac{1}{N(A_a)}\sum_{i:X_i\in A_{a,L}}{(Y_i-\bar{Y}_{A_{a,L}})^2}- \frac{1}{N(A_b)}\sum_{i:X_i\in A_{b,L}}{(Y_i-\bar{Y}_{A_{b,L}})^2} \nonumber\\
     &=  \frac{1}{N(A_a)}\left[ \sum_{i:X_i\in I}{(Y_i-\bar{Y}_{A_{a,L}})^2}+\sum_{i:X_i\in A_{a,L}\setminus I}{(Y_i-\bar{Y}_{A_{a,L}})^2}\right] \nonumber\\
     &-  \frac{1}{N(A_b)}\left[ \sum_{i:X_i\in I}{(Y_i-\bar{Y}_{A_{b,L}})^2}+\sum_{i:X_i\in A_{b,L}\setminus I}{(Y_i-\bar{Y}_{A_{b,L}})^2}\right]\nonumber\\
     &=  \left[ \frac{1}{N(A_a)}\sum_{i:X_i\in I}{(Y_i-\bar{Y}_{A_{a,L}})^2} -\frac{1}{N(A_b)}\sum_{i:X_i\in I}{(Y_i-\bar{Y}_{A_{b,L}})^2}\right]\nonumber\\
     &+ \left[ \frac{1}{N(A_a)}\sum_{i:X_i\in A_{a,L}\setminus I}{(Y_i-\bar{Y}_{A_{a,L}})^2} -\frac{1}{N(A_b)}\sum_{i:X_i\in A_{b,L}\setminus I}{(Y_i-\bar{Y}_{A_{b,L}})^2}\right] \nonumber\\
     &:= T_{3,1}+T_{3,2}. \nonumber
\end{align}
Note that $T_{3,1} $ and $ T_{3,2}$ are similar to $T_{1,1}$ and $ T_{1,2}$ respectively.  Therefore,  methodologies that are used to deal with $T_{1,1}, T_{1,2}$  can be employed again to bound terms $T_{3,1}, T_{3,2}$. In conclusion, by \eqref{hdasbkJHBNJKbnkjdfjklqwBernouli}, \eqref{babdJKkdjskajdsHJNNber} and \eqref{babdJKkdj2skajds2HJNNber} there exits $\delta(\alpha)>0$ with $ N_{15}(\delta(\alpha))\in\mathbb{Z}_+$ such that if $n>N_{15}(\delta(\alpha))$,
$$
 |\Delta_{n,4}(d_1,d_2,d_3,d_4')-\Delta_{n,4}(d_1,d_2,d_3',d_4')|\le 3\alpha
$$
with probability larger than $1-\rho$. This completes the proof. 	\hfill\(\Box\)

\section{Proof of Proposition \ref{proposition1}}
 At the beginning of the proof, we analyze the first tree $T_{\D_n,\L}$ (or its random version $T_{\D_n,\L}^r$) in the  boosting process. Let $\mathbb{A}_1,\mathbb{A}_2,\ldots,\mathbb{A}_t$ be $t$ leaves of $T_{\D_n,\L}$ (or its random version $T_{\D_n,\L}^r$). Then, we know each $\mathbb{A}_j$ is generated after performing $\mathcal{C}_j$ cuts in $[0,1]^p$ with $\mathcal{C}_j\le t$. Since each tree partition corresponds with a direction $\theta\in\R^p$ and a threshold $s\in\R$, we can denote each $\mathbb{A}_j$ by 
        $$
          \mathbb{A}_j=\tilde  A_{j.1}\cap\cdots\cap  \tilde A_{j.\mathcal{C}_j},
        $$        
        where $\tilde A_{j.\ell} = \{x\in [0,1]^p: \theta_{j,\ell}^Tx > s_\ell\} $  or $\tilde A_{j.\ell} = \{x\in [0,1]^p: \theta_{j,\ell}^Tx \le s_\ell\} $ for each $\ell=1,2,\ldots,\mathcal{C}_j$
and $\theta_{j,\ell}\in\R^p, s_\ell\in\R$. In Figure \ref{ODT11fig2}, We give an example of such representation of tree leaves.

    \begin{figure}[ht]
    \includegraphics[height=0.25\textheight]{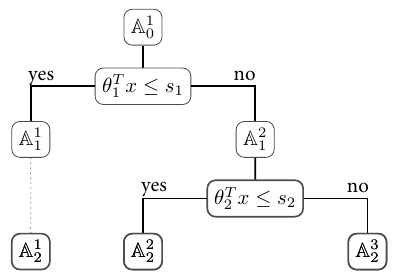}
    \caption{This ODT  has two layers and three leaves denoted by $\mathbb{A}_2^1, \mathbb{A}_2^2, \mathbb{A}_2^3$. Note that $\mathbb{A}_1^1$ is not partitioned anymore and thus $\mathbb{A}_1^1=\mathbb{A}_2^1$. Meanwhile, it can be seen that $\mathbb{A}_2^1=\{x:\theta_1^Tx\le s_1\}$, $\mathbb{A}_2^2=\{x:\theta_1^Tx> s_1\}\cap\{x:\theta_2^Tx\le s_2\}$ and $\mathbb{A}_2^3=\{x:\theta_1^Tx> s_1\}\cap\{x:\theta_2^Tx> s_2\}$.}
    \label{ODT11fig2}
    \end{figure}     
        
    Meanwhile, note that the following equation holds 
        \begin{equation}\label{asbdBerr}
         \mathbb{I}(x\in  \mathbb{A}_j)= \sigma_0\left(\sum_{\ell=1}^{\L_j}\sigma_0(s_\ell-\theta_{j,\ell}^Tx)-\mathcal{C}_j\right)
       \end{equation}
        if 
        \begin{equation}\label{asbdBerr2}
            \mathbb{A}_j=\{x\in [0,1]^p: \theta_{j,1}^Tx\le s_1\}\cap\cdots\cap  \{x\in [0,1]^p:\theta_{j,\mathcal{C}_j}^Tx\le s_{\mathcal{C}_j}\}.
        \end{equation}

        Since $\mathbb{I}(\{x\in[0,1]^p: \theta^\top x>s\})=\sigma_0(0)-\sigma_0(s-\theta^\top x)$, we can assume \eqref{asbdBerr} holds without loss of generality. This is because that if $\theta_{j,\ell}^\top x>s$ we only need to replace $\sigma_0(s_\ell-\theta_{j,\ell}^Tx)$ by $\sigma_0(0)-\sigma_0(s_\ell-\theta_{j,\ell}^\top x)$ in \eqref{asbdBerr}. Recall that $\bar{Y}_{\mathbb{A}_j}$ is the constant estimator in the region
         $\mathbb{A}_j$. Therefore, the first tree in the boosting process is equal to
        $$
            \sum_{j=1}^{t}{\bar{Y}_{\mathbb{A}_j} \sigma_0\left(\sum_{\ell=1}^{\mathcal{C}_j}\sigma_0(s_\ell-\theta_{j,\ell}^Tx)-\mathcal{C}_j\right)},
        $$
        which is a neural network with three layers. Therefore, $T_{\D_n,\L}$ (or its random version $T_{\D_n,\L}^r$) can be regarded as a neural network with  $\sum_{j=1}^t{\mathcal{C}_j}$ neurons in the first hidden layer and  $t$ neurons in the second hidden layer. 
        
        Since feed-forward neural networks have additive structures, we know 
        any boosting tree with $k$ ODTs is  in the following  neural network class
        $$
         \left\{ \sum_{i=1}^k\sum_{j=1}^ta_{i,j}\sigma_0\left(  \sum_{\ell=1}^t\sigma_0(\theta_{i,j,\ell}^Tx+s_{i,j,\ell})b_{i,j,\ell}+v_{i,j}\right): a_{i,j},b_{i,j,\ell},s_{i,j,\ell},v_{i,j}\in\R,\theta_{i,j,\ell}\in\R^p\right\},
        $$
        which has $kt^2(p+1)$ parameters ($\theta_{i,j,\ell},s_{i,j,\ell}$) in the first  hidden layer and  $kt(t+1)$ parameters ($b_{i,j,\ell},v_{i,j}$) in the second  hidden layer  and  $kt$ parameters ($a_{i,j}$) in the final  hidden layer. This completes the proof.   \hfill\(\Box\)

\section{Proof of Lemma \ref{Oracle inequality for boosting tree}}\label{S.8.}
First, we fix the data $\D_n$ and omit it when writing any conditional expectation during this proof. The boosting tree with random ODT is shown as follows.

	\begin{algorithm}[H]\label{Algorithm.rBODT}
		\caption{random ODT-based boosting tree}
		\KwIn{Data $\D_n=\{(X_i,Y_i)\}_{i=1}^n$, pre-specified number of leaves $t$, and maximum number of iterations $k_n$.}
		
		Initialize the number of iteration $j=1$, and data $ \D_n^1=\D_n$.
		
		\For{$j \in  \{1, 2,\cdots, k_n\}$}{

            Obtain the tree estimator $ m^r_{n, n, t,\L_j} $ by using Algorithm r1 (check it at the beginning of Section \ref{sec.ODRF}) with data $\D_n^j$;
            
			Estimate  $m(x)$ in the step $j$  by
			\begin{equation}\label{jhajkhdjkDSTEP323}
				m_{j,boost}(x):= \sum_{\ell=1}^j{a_{j,\ell}^*m^r_{n, n,t,\L_\ell}(x)},
			\end{equation}
			where
			\begin{equation}\label{jhajkhdjkDSTEP3}
				(a_{j,1}^*,a_{j,2}^*,\ldots,a_{j,j}^*):=arg\min_{\substack{a_{j,\ell}\in\R\\\ell=1,\ldots,j}}{\sum_{i=1}^n{\left(Y_i- \sum_{\ell=1}^j{a_{j,\ell}m^r_{n, n, t,\L_\ell}(X_i)}\right)^2}};
			\end{equation}


           Calculate residual $r_{j,i}=Y_i-m_{j,boost}(X_i)$ and update data $\D_n^{j+1}=\{(X_i,r_{j,i})\}_{i=1}^n$;

           Update $j\leftarrow j+1$;
		}
		\KwOut{The boosting estimator  $m_{k_n,boost}(x), x\in [0,1]^p$.}
	\end{algorithm} 


 For the activation function $ \sigma_0 $, we can assume $\theta_j\in\Theta^p$ since $\sigma_0(\theta_j^\top x+s_j)=\sigma_0(\theta_j^\top x/\|\theta_j\|_2+s_j/\|\theta_j\|_2)$ for any $x\in [0,1]^p$. Rewrite the function $h(x)$ as
    $$
        h(x)=\sum_{j=1}^\infty a_{j}\|\sigma_0(\theta_j^\top x+s_j)\|_n\cdot \frac{\sigma_0(\theta_j^\top x+s_j)}{\|\sigma_0(\theta_j^\top x+s_j)\|_n}\in \mathcal{G}
    $$
    and define 
    $$
       \|h\|_{EL^1}:= \sum_{j=1}^\infty |a_{j}|\|\sigma_0(\theta_j^\top x+s_j)\|_n\le \sum_{j=1}^\infty{|a_j|}<\infty
    $$
    since $\|\sigma_0(\theta_j^\top x+s_j)\|_n\le 1$ and  $h(x)\in\mathcal{G}$. Then, define the dictionary
$$
\mathcal{D}ic_n=\left\{\frac{\sigma_0(\theta^\top x+s)}{\|\sigma_0(\theta^\top x+s)\|_n}: \theta\in \Theta ^p,s\in\R\right\}.
$$

Denote by $r_{k_n}:=\mathbb{Y}-(m_{k_n,boost}(X_1),\ldots, m_{k_n,boost}(X_n))^\top\in\R^n$ the vector of training errors for each $k_n\geq 1$. Specially, we define $r_0=\mathbb{Y}$. For each $k_n\in\mathbb{Z}_+$, we have
 \begin{equation}\label{Math21}
     \|r_{k_n-1}\|_n^2=\langle r_{k_n-1},h+\mathbb{Y}-h\rangle_n \leq \|h\|_{EL^1}\cdot\sup_{g\in \mathcal{D}ic_{n}}|\langle r_{k_n-1},g\rangle|_n+\|r_{k_n-1}\|_n\|\mathbb{Y}-h\|_n,
 \end{equation}
 where in the first equation, we use the fact that the vector $r_{k_n-1}$ is orthogonal to $m_{k_n-1,boost}\in\R^n$ (recall the projection step \eqref{jhajkhdjkDSTEP3} in our boosting process); and in the last inequality, we use Cauchy-Schwarz inequality in $\R^n$. Note that $\|r_{k_n-1}\|_n \|\mathbb{Y}-h\|_n\leq \frac{1}{2}(\|r_{k_n-1}\|_n^2+\|\mathbb{Y}-h\|_n^2)$. Therefore, \eqref{Math21} implies that
 \begin{equation}\label{Math22}
    \sup_{g\in \mathcal{D}ic_n}|\langle r_{k_n-1},g\rangle|_n\geq \frac{\|r_{k_n-1}\|_n^2-\|\mathbb{Y}-h\|_n^2}{2\cdot\|h\|_{EL^1}}.
 \end{equation}

Before continuing our arguments, we need more notations on random seeds.    For any $1\le j\le k_n$, let $$\mathcal{S}_{j,1},\mathcal{S}_{j,2},\ldots,\mathcal{S}_{j,t-1}\subseteq\{1,2,\ldots,p\}$$ be a sequence of random indexes chosen for the CART splitting in the $j$-th random ODT (see the beginning of Section \ref{sec.ODRF}). For example, $\mathcal{S}_{j,1}$ is used in the   root node splitting of the $j$-th tree. Let $$\Phi_{k_n}:=(\mathcal{S}_{1,1},\mathcal{S}_{1,2},\ldots,\mathcal{S}_{1,t-1},\ldots,\mathcal{S}_{k_n,1},\mathcal{S}_{k_n,2},\ldots,\mathcal{S}_{k_n,t-1})$$ be the collection of all random indexes. Therefore, $\Phi_{k_n}$ is used in the construction of $m_{k_n,boost}$ in Algorithm \thealgocf.

Now, we study the relationship between$\|r_{k_n-1}\|_n^2$ and $\E_{\Phi_{k_n}|\Phi_{k_n-1}}\|r_{k_n}\|^2$. Recall the estimator corresponding with the boosting tree 
is $m_{k_n,boost}(x):= \sum_{\ell=1}^{k_n}{a_{k_n,\ell}^*m^r_{n, n,t,\L_\ell}(x)}$. From \eqref{jhajkhdjkDSTEP3}, we see $r_{k_n}\in\R^n$ is the residual vector obtained by projecting  $\mathbb{Y}$ onto the linear space $span\{v_\ell, \ell=1,\ldots,k_n\}$ and the vector $v_\ell:=(m^r_{n, n,t,\L_\ell}(X_1),m^r_{n, n,t,\L_\ell}(X_2),\ldots,m^r_{n, n,t,\L_\ell}(X_n))^\top\in\R^n$. Therefore, for any $a,b,s\in\R$ and $\theta\in\Theta^p$, we have
\begin{align}
    \E_{\Phi_{k_n}|\Phi_{k_n-1}}\|r_{k_n}\|_n^2&\le \E_{\Phi_{k_n}|\Phi_{k_n-1}}\|r_{k_n-1}- m^{r}_{n,n,t,\L_{k_n}}\|_n^2 \nonumber\\
    &\le \E_{\Phi_{k_n}|\Phi_{k_n-1}}\|r_{k_n-1}- a\cdot \mathbb{I}(-\theta^\top x_{\mathcal{S}_{k_n,1}}\le s)-b\cdot \mathbb{I}(-\theta^\top x_{\mathcal{S}_{k_n,1}}> s)\|_n^2 \label{dgshjfbbnqqw}\\
    &\le  \E_{\Phi_{k_n}|\Phi_{k_n-1}}\| r_{k_n-1}- a\cdot \sigma_0(\theta^\top x_{\mathcal{S}_{k_n,1}}+ s)\|_n^2 \nonumber\\
    &\le \E_{\Phi_{k_n}|\Phi_{k_n-1}}\left( \| r_{k_n-1}\|_n^2 -\left|\langle r_{k_n-1}, \frac{\sigma_0(\theta^\top x_{\mathcal{S}_{k_n,1}}+s)}{\|\sigma_0(\theta^\top x_{\mathcal{S}_{k_n,1}}+s)\|_n} \rangle_n\right|^2 \right) \nonumber\\
    &\le \|r_{k_n-1}\|_n^2- c(p)\cdot \left|\langle r_{k_n-1}, \frac{\sigma_0(\theta^\top x+s)}{\|\sigma_0(\theta^\top x+s)\|_n} \rangle_n\right|^2, \label{hadsb1231}
\end{align}
       where $c(p)>0$; and \eqref{dgshjfbbnqqw} holds since the training error  corresponding with the  layer $\L_{k_n}$ is no larger than the training error of the first layer; and \eqref{hadsb1231} holds because we have a positive probability to choose all features in the cut of root node of the $k_n$-th random ODT. Since \eqref{hadsb1231} is true for arbitrary $\theta\in \Theta^p$ and $s\in\R$, thus 
 \begin{equation}\label{Math23}
     \|r_{k-1}\|_n^2-\E_{\Phi_{k_n}|\Phi_{k_n-1}}\|r_{k_n}\|_n^2\geq c(p)\cdot\sup_{g\in \mathcal{D}ic_{n}}{|\langle r_{k_n-1},g\rangle|_n^2}.
 \end{equation}
  
  In conclusion, the combination of \eqref{Math22} and \eqref{Math23} implies that
  \begin{equation}\label{Math24}
     \|r_{k_n-1}\|_n^2-\E_{\Phi_{k_n}|\Phi_{k_n-1}}\|r_{k_n}\|_n^2\geq \frac{\max{(\|r_{k_n-1}\|_n^2-\|\mathbb{Y}-h\|_n^2,0})^2}{4\cdot\left(\|h\|_{EL^1}/ \sqrt{c(p)}\right)^2}.
 \end{equation}
Taking expectation w.r.t. $\Xi_{k_n-1}$ on both sides of \eqref{Math24} yields 
\begin{equation}\label{Math25}
    \E_{\Phi_{k_n}}\|r_{k_n}\|_n^2\leq  \E_{\Phi_{k_n-1}}\|r_{k_n-1}\|_n^2-\frac{ \left(\E_{\Phi_{k_n-1}}\max{(\|r_{k_n-1}\|_n^2-\|\mathbb{Y}-h\|_n^2,0})\right)^2}{4\cdot\left(\|h\|_{EL^1}/ \sqrt{c(p)}\right)^2},
\end{equation}
where we use the law of iterated expectation and Jensen's inequality.

If $\E_{\Phi_{k_n-1}} \left( \|r_{k_n-1}\|_n^2-\|\mathbb{Y}-h\|_n^2 \right)\leq 0$, then the inequality $\E_{\Phi_{k_n}} \left( \|r_{k_n}\|_n^2-\|\mathbb{Y}-h\|_n^2 \right)\leq 0$ also holds; and the proof is already completed. This argument is verified because 
$\E_{\Phi_{k_n}|\Phi_{k_n-1}}\|r_{k_n}\|_n^2\leq \|r_{k_n-1}\|_n^2$ almost surely and $\E_{\Phi_{k_n}}(\|r_{k_n}\|_n^2)=\E_{\Phi_{k_n-1}}\E_{\Phi_{k_n}|\Xi_{k_n-1}}\|r_{k_n}\|_n^2$. Therefore, without loss of generality, we can assume that  for each $1\leq j\leq k_n$,  $$
\E_{\Phi_{j-1}} \left( \|r_{j-1}\|_n^2-\|\mathbb{Y}-h\|_n^2 \right)> 0.$$ With similar arguments, \eqref{Math25} implies we have recurrences
\begin{equation}\label{mathematical induction method}
      a_j\leq  a_{j-1}\left(1-\frac{a_{j-1}}{M}\right)
\end{equation}
for each $1 \leq j\leq k_n$, where $a_j:=\E_{\Xi_{j}} \left( \|r_{j}\|_n^2-\|\mathbb{Y}-h\|_n^2 \right)\geq 0$ and $M=4\left(\|h\|_{EL^1}/ \sqrt{c(p)}\right)^2$.
If $a_0\leq M$, it is not difficult to know  $a_{k_n}\leq \frac{M}{k_n+1}$ by  mathematical induction and \eqref{mathematical induction method}. Otherwise $a_1<0$ by \eqref{mathematical induction method}, which implies that $a_{k_n}<0$ for all $k_n\geq 1$ because $\{a_{j}\}_{j=0}^\infty$ is  non-increasing. Therefore, this lemma is always true, regardless of the sign of $a_0$.  \hfill\(\Box\)

\section{Proof of Theorem \ref{Boosting tree consi}}

Recall Lemma \ref{agirov2009estimation} shows  that
\begin{equation}\label{Mainlemmaformula2}
    \begin{aligned}
\E\left(\int{|\hat{m}^1_{k_n,boost}(x)-m(x)|^2d\mu(x)}\right) &\leq 2 \E_{\mathcal{D}_{n},\Xi_{k_n}}\left( \| m^1_{k_n,boost}(X)-\mathbb{Y}\|_n^2-\|m(X)-\mathbb{Y}\|_n^2\right)\\
& \ \ \ +\frac{c\ln^2 n}{n}\cdot\sup_{z_1^n}\ln\left(\mathcal{N}_1(1/(80nt_n),\mathcal{NN}_3^{\beta_n},z_1^n)\right),
\end{aligned}
\end{equation}
where $\mathbb{Y}=(Y_1,\ldots,Y_n)^\top\in\R^n$  and the above $\mathcal{N}_1(\cdot,\cdot,\cdot)$ is related to the  covering number of $\mathcal{NN}_3^{\beta_n}$; see Definition \ref{coveringnumber}. To give an upper bound for the first part on the RHS of \eqref{Mainlemmaformula2}, let
\begin{align*}
  \text{I} &=\|m^1_{k_n,boost}(X)-\mathbb{Y}\|_n^2-\|m(X)-\mathbb{Y}\|_n^2 \\
   &= (\|\mathbb{Y}-m^1_{k_n,boost}\|_{n}^2-\|\mathbb{Y}-h_n\|_{n}^2) + (\|\mathbb{Y}-h\|_{n}^2-\|\mathbb{Y}-m\|_{n}^2) := \text{II} + \text{III},
\end{align*}
where $h= \sum_{j=1}^\infty a_{j}\sigma_0(\theta_{n,j}^\top x+s_{n,j})\in\mathcal{G}$. For $\text{II}$,
by Lemma \ref{Oracle inequality for boosting tree} we have
\begin{equation}\label{Sta1}
 \E_{\Phi_{k_n}|\mathcal{D}_{n}}(\text{II})\leq c(p) \left(\sum_{j=1}^\infty{|a_{j}|}\right)^2\cdot\frac{1}{k_n+1},
\end{equation}
where $\Phi_{k_n}$ is defined in Section \ref{S.8.}. Taking expectation w.r.t. $\mathcal{D}_n$ on both sides of \eqref{Sta1} implies that
\begin{equation}\label{Sta2}
     \E_{\Phi_{k_n}, \mathcal{D}_{n}}(\text{II})\leq c(p) \left(\sum_{j=1}^\infty{|a_{j}|}\right)^2\cdot\frac{1}{k_n+1}.
\end{equation}
Since $\text{III}$ is independent to the random seed $\Phi_{k_n}$, thus we further have
\begin{equation}\label{Sta3}
     \E_{\Phi_{k_n}, \mathcal{D}_{n}}(\text{III})=\E\left((Y-h_n(X))^2- (Y-m(X))^2\right)=\E\left(m(X)-h_n(X)\right)^2.
\end{equation}
The combination of \eqref{Sta2} and \eqref{Sta3} yields
\begin{equation}\label{Sta4}
     \E_{\Phi_{k_n}, \mathcal{D}_{n}}(\text{I})\leq
     c(p) \left(\sum_{j=1}^\infty{|a_{j}|}\right)^2\cdot\frac{1}{k_n+1} + \E\left(m(X)-h(X)\right)^2.
\end{equation}
For the second part on the RHS of \eqref{Mainlemmaformula2},  we can use \eqref{aihkdjbmaBGKOLqqb}. Therefore, it follows from \eqref{Mainlemmaformula2}, \eqref{Sta4} and \eqref{aihkdjbmaBGKOLqqb} that
    \begin{align}
      \E\left(\int{|\hat{m}^1_{k_n,boost}(x)-m(x)|^2d\mu(x)}\right)  &\leq  c(p,t)\frac{\ln^2 n}{n} k_n\ln(k_n) \ln\left(cn \beta^2_n \right) \nonumber \\
      &+c(p) \left(\sum_{j=1}^\infty{|a_{j}|}\right)^2\cdot\frac{1}{k_n+1}+\E\left(m(X)-h(X)\right)^2.\label{Sta5}
    \end{align}
    Next, we finish our arguments by showing the RHS of \eqref{Sta5} goes to $0$ as $n\to\infty$. Let $\varepsilon>0$ be any given number. Since $\sigma_0(v),v\in\R$ is a squashing function, by Lemma  16.1 in \cite{gyorfi2006distribution} there is $h_\varepsilon=\sum_{j=1}^\infty a_{\varepsilon,j}\sigma_0(\theta_{\varepsilon,j}^\top x+s_{\varepsilon,j}) \in\mathcal{G}$ such that
    \begin{equation}\label{dbkjfcjnbNBsVHBJHKMa}
        \E\left(m(X)-h_\varepsilon(X)\right)^2\le \frac{\varepsilon}{2}.
    \end{equation}
   Replace $h$ and $ \sum_{j=1}^\infty{|a_j|}$ in \eqref{Sta5} by $h_\varepsilon$ and $ \sum_{j=1}^\infty{|a_{\varepsilon,j}|}$  respectively.  By the assumptions in Theorem \ref{Boosting tree consi}, there is $N \in\mathbb{Z}_+$ such that
   \begin{equation}\label{f}
    c(p,t)\frac{\ln^2 n}{n} k_n\ln(k_n) \ln\left(c n \beta^2_n \right)   +c(p) \left(\sum_{j=1}^\infty{|a_{\varepsilon,j}|}\right)^2\cdot\frac{1}{k_n+1}\le \frac{\varepsilon}{2}
   \end{equation}
   for all $n\ge N$. Thus, the combination of \eqref{dbkjfcjnbNBsVHBJHKMa} and \eqref{f} guarantees the consistency of $\hat{m}^1_{k_n,boost}(x)$. Finally,  the consistency of $\hat{m}^{ens}_{k_n,boost}(x)$ follows from the Jensen's inequality.  \hfill\(\Box\)

\section{Proof of Theorem \ref{theorem_lowerbound}}\label{sec::last}

Without loss of generality, we assume $\theta_j\in\Theta^p$ since $\sigma_0(\theta_j^\top x+s_j)=\sigma_0(\theta_j^\top x/\|\theta_j\|_2+s_j/\|\theta_j\|_2)$ for any $x\in [0,1]^p$. Because $x$ is restricted in $[0,1]^p$, we can also assume $s_j\in [-M,M]$ for some large $M>0$.  Therefore,  $\mathcal{G}_1$ can be rewritten as
  $$
\mathcal{G}_1:=\left\{  \sum_{j=1}^\infty a_{j}\sigma_0(\theta_j^\top x+s_j): x\in [0,1]^p, \sum_{j=1}^\infty{|a_j|}\le 1, \theta_j\in\Theta^p, s_j\in [-M,M] \right\}.
$$
 The proof of this theorem is based on Theorem 1 in \cite{yang1999information}. For the completeness of this supplement, we list it below. 

 \begin{lemma}[\cite{yang1999information}]
     Consider the regression model
     $$
        Y_i = m(X_i)+ \varepsilon_i,\ \ i = 1,\ldots, n,
    $$
where $X_i\sim U[0,1]^p$ and $\varepsilon_i \sim N(0,1)$ are independent Gaussian noises and $m\in \mathcal{M}$ for some function class $ \mathcal{M}$. Recall that the $L^2$ distance between any functions $f,g$ is defined by $\sqrt{\E(f(X_1)-g(X_1))^2}$.
  Suppose that there exists $\delta, \xi > 0$ such that
  \begin{equation}\label{dshbafkdjbnkjxyqc}
   \ln \mathcal{N}(\xi, \mathcal{M})\le \frac{n\xi^2}{2}, \ \  \ln \mathcal{P}(\delta, \mathcal{M})\ge 2n\xi^2+\ln 2,    
  \end{equation}
  where $\mathcal{N}(\xi, \mathcal{M}) (\mathcal{P}(\delta, \mathcal{M}))$ denotes the cardinality of a minimal $\xi$ ($\delta$)-covering (packing) for the class $\mathcal{M}$ under the $L^2$ distance respectively.  Then we have
  \begin{equation}\label{bhsabddmnA45SBqber}
    \inf_{\hat{m}}\sup_{m(x)\in\mathcal{M}}\E\left(\int{|\hat{m}(x)-m(x)|^2dx }\right)\ge \frac{\delta^2}{8}.
  \end{equation}
 \end{lemma}

 Let $\mathcal{M}=\mathcal{G}_1$ in above lemma. Firstly, by \cite{siegel2022sharp}, we know 
 $$
     \ln \mathcal{N}(\xi, \mathcal{G}_1)\le c\cdot \left(\frac{1}{\xi}\right)^{\frac{2p}{p+1}}
 $$
 for some $c>0$. Then, by \cite{siegel2022sharp}, we have the lower bound about the packing number:
 $$
  \ln \mathcal{P}(\delta, \mathcal{G}_1)\ge  \ln \mathcal{P}(2\delta, \bar{\mathcal{G}}_1)\ge \ln \mathcal{N}(2\delta, \bar{\mathcal{G}}_1)\ge c\cdot \left(\frac{1}{\delta}\right)^{\frac{2p}{p+1}},
 $$
 where $\bar{\mathcal{G}}_1$ denotes the closure of $\mathcal{G}_1$ calculated by using norm of  $L^2$ space $\{f:\E(f^2(X))<\infty\}$. Choosing $\xi=\delta=c\cdot (1/n)^{\frac{p+1}{4p+2}}$ for some $c>0$,  it can be checked that the two equations in \eqref{dshbafkdjbnkjxyqc} are satisfied. Therefore, the result in \eqref{bhsabddmnA45SBqber} gives us the lower bound  as desired. This completes the proof. \hfill\(\Box\)

\

\section{Computational issues}\label{computation}

We have explained the specific parameter settings in Section 6. This section  presents the summary of datasets used in paper; see Table \ref{datasets} in this supplementary and will discuss other details of the calculation. More calculation settings and details can be found in our R-package “\textsf{ODRF}”, while the exact values are used as default values in this package.

Specifically, the depth of all tree methods is determined by their respective stopping criteria of CART or RF, while the number of iterations in our boosting method is determined by the BIC criterion. More precisely, our boosting method with bagging employs a BIC-type criterion to determine the number of iterations. The BIC includes both fit and out-of-bag errors and incorporates penalties related to sample size \( n \), dimensionality \( p \), and the number of iterations to prevent overfitting. Let \( n^{(oob)} \) denote the out-of-bag observations as \( \{(X_b^{(oob)}, y_b^{(oob)}) : b=1, \ldots, n^{(oob)}\} \) and let \( m_{\tau} \) represent the output function of the inner loop of ODBT. Our BIC for the iteration is defined
as follows:
	\begin{align*}
	\text{BIC}(\tau) = &\log\left(\frac{1}{n+n^{(oob)}} \Big(\sum_{i=1}^n (Y_i - m_{\tau}(X_i))^2 + \sum_{b=1}^{n^{(oob)}} (Y_b^{(oob)} - m_{\tau}(X_b^{(oob)}))^2\Big)\right) \\
 &+ \frac{\tau \log(p) \log(n)}{n}.    
	\end{align*}

About the computational time, we have compared the computational time of ODT with CART and other tree-based methods, including Extremely Randomized Trees (ERT), Evolutionary Learning of Globally Optimal Classification and Regression Trees (EVT), and Conditional Inference Trees (CT). Each method was implemented using its respective R package: ERT with the `RLT()' function from the `RLT' package, EVT with the `evtree()' function from the `evtree` package, and CT with the `ctree()' function from the `partykit' package.  All data sets and codes used in the paper are publicly accessible at \url{https://github.com/liuyu-star/ODBT}. The results presented in the paper can be reproduced by running the provided code.

All experimental settings were kept consistent with those in our paper, and the results of 100 repetitions are presented in Tables \ref{Table4} and \ref{Table5}. The computations were conducted in R version 4.2.2 on a Windows 11 operating system with an AMD Ryzen 7 5800H @ 3.2GHz processor, and no multi-threading was used.  As shown in the tables, ODT provides more accurate classification and regression predictions, but indeed at some computational expense. Specifically, in classification tasks, ODT incurs only a small increase in computation time compared to CART, while significantly reducing the misclassification rate. In contrast, other methods either require more computational time or offer lower accuracy, making it challenging to achieve an optimal balance between speed and performance.


\begin{table}[t]
	\centering
	\caption{Dataset summaries for and regression and classification experiments.}\label{datasets}
\begin{minipage}{0.5\textwidth}
\setlength{\tabcolsep}{0.mm}
{
\begin{tabular*}{0.9\textwidth}{@{\extracolsep{\fill}}llrrrll}
    \toprule%
    \multicolumn{5}{@{}c@{}}{Dataset with continuous responses} \\
    \cmidrule(lr){1-5}
    data & Name  & N     & p     & Link \\
    \midrule
    data.1 & Servo & 166   & 4     & B \\
    data.2 & Strike & 624   & 6     & B \\
    data.3 & Auto MPG & 391   & 7     & B \\
    data.4 & Low birth weight & 188   & 9     & B \\
    data.5 & Pharynx & 192   & 12    & B \\
    data.6 & Body fat & 251   & 14    & B \\
    data.7 & Paris housing price & 10000 & 16    & C \\
    data.8 & Parkinsons & 5875  & 19    & A \\
    data.9 & Auto 93 & 81    & 22    & B \\
    data.10 & Auto horsepower & 159   & 24    & B \\
    data.11 & Wave energy & 71998 & 32    & C \\
            & Converters-Adelaide &   \\
    data.12 & Baseball player statistics & 4535  & 74    & C \\
    data.13 & Year prediction MSD & 50000 & 90    & A \\
    data.14 & Residential  & 372   & 103   & A \\
        & building-Sales\\
    data.15 & Residential  & 372   & 103   & A \\
        & building-Cost \\
    data.16 & Geographical  & 1059  & 116   & A \\
             &  original-Latitude \\
    data.17 & Geographical & 13063 & 249   & A \\
                 &  original-Latitude \\
    data.18 & Credit score & 80000 & 259   & C \\
    data.19 & CT slices & 53500 & 380   & A \\
    data.20 & UJIndoor-Longitude & 19937 & 465   & A \\
    \bottomrule
\end{tabular*}
}
\end{minipage}%
\begin{minipage}{0.5\textwidth}
\setlength{\tabcolsep}{0.mm}
{
\begin{tabular*}{0.9\textwidth}{@{\extracolsep{\fill}}llrrrll}
    \toprule%
    \multicolumn{5}{@{}c@{}}{Dataset with binary categorical response (0 and 1)} \\
    \cmidrule(lr){1-5}
    data & Name  & N     & p     & Link \\
    \midrule
    data.21 & MAGIC & 19020 & 10    & B \\
    data.22 & EEG eye state & 14980 & 14    & A \\
    data.23 & Diabetic retinopathy  & 1151  & 19    & A \\
       &  debrecen \\ 
    data.24 & Parkinson multiple  & 1208  & 26    & C \\
      & sound \\
    data.25 & Pistachio & 2148  & 28    & C \\
    data.26 & Breast cancer & 569   & 30    & C \\
    data.27 & Waveform (2) & 5000  & 40    & B \\
    data.28 & QSAR biodegradation & 1055  & 41    & A \\
    data.29 & Spambase & 4601  & 57    & A \\
    data.30 & Mice protein expression & 1047  & 70    & A \\
    data.31 & Ozone level detection & 1847  & 72    & A \\
    data.32 & Company bankruptcy & 6819  & 94    & C \\
    data.33 & Hill valley & 1211  & 100   & B \\
    data.34 & Hill valley noisy & 1211  & 100   & B \\
    data.35 & Musk  & 6598  & 166   & A \\
    data.36 & ECG heartbeat  & 14550 & 186   & C \\
             &  categorization \\
    data.37 & Arrhythmia & 420   & 192   & A \\
    data.38 & Financial indicators  & 986   & 216   & C \\
     & of US stocks\\
    data.39 & Madelon & 2000  & 500   & A \\
    data.40 & Human activity  & 2633  & 561   & A \\
     & recognition\\
    \bottomrule
\end{tabular*}
}
\end{minipage}
 
	\end{table}

\begin{table}[t]
		\centering
		\caption{Regression: average RPE based on 100 random partitions of each data set into training and test sets}\label{Table4}%
		{ \small
					\begin{tabular}{lrrrrrrrrrr}
								\toprule%
								& \multicolumn{5}{c}{relative prediction error}& \multicolumn{5}{c}{Computational time (in seconds)}  \\
								\cmidrule(lr){2-6}\cmidrule(lr){7-11}%
								data & CART	&ERT&	EVT	&CT	&ODT& CART	&ERT&	EVT	&CT	&ODT \\
						\midrule
					data.1 & 0.298  & 0.860  & 0.254  & 0.300  & 0.313  & 0.007  & 0.001  & 0.109  & 0.029  & 0.021  \\
					data.2 & 1.003  & 1.279  & 0.959  & 0.915  & 0.900  & 0.009  & 0.001  & 0.218  & 0.047  & 0.063  \\
					data.3 & 0.213  & 0.302  & 0.222  & 0.192  & 0.177  & 0.008  & 0.001  & 0.329  & 0.073  & 0.049  \\
					data.4 & 0.399  & 0.743  & 0.383  & 0.393  & 0.439  & 0.007  & 0.002  & 0.054  & 0.024  & 0.024  \\
					data.5 & 0.377  & 1.129  & 0.420  & 0.334  & 0.478  & 0.010  & 0.001  & 0.095  & 0.042  & 0.027  \\
					data.6 & 0.061  & 0.373  & 0.067  & 0.049  & 0.063  & 0.010  & 0.001  & 0.232  & 0.090  & 0.038  \\
					data.7 & 0.018  & 0.316  & 0.004  & 0.000  & 0.000  & 0.047  & 0.015  & 4.436  & 1.025  & 0.538  \\
					data.8 & 0.295  & 1.020  & 0.436  & 0.439  & 0.538  & 0.086  & 0.012  & 3.742  & 0.620  & 0.575  \\
					data.9 & 0.620  & 0.929  & 0.685  & 0.602  & 0.595  & 0.009  & 0.002  & 0.047  & 0.044  & 0.014  \\
					data.10 & 0.236  & 0.346  & 0.250  & 0.238  & 0.207  & 0.013  & 0.001  & 0.136  & 0.084  & 0.030  \\
					data.11 & 0.543  & 0.721  & 0.554  & 0.488  & 0.494  & 0.159  & 0.014  & 3.794  & 0.838  & 0.864  \\
					data.12 & 0.039  & 0.168  & 0.015  & 0.008  & 0.003  & 0.182  & 0.027  & 4.242  & 0.952  & 1.498  \\
					data.13 & 0.969  & 1.682  & 0.928  & 0.916  & 1.338  & 0.393  & 0.040  & 2.488  & 0.851  & 1.771  \\
					data.14 & 0.060  & 0.255  & 0.058  & 0.040  & 0.046  & 0.048  & 0.004  & 0.367  & 0.623  & 0.209  \\
					data.15 & 0.110  & 0.214  & 0.138  & 0.094  & 0.090  & 0.048  & 0.004  & 0.346  & 0.614  & 0.194  \\
					data.16 & 1.043  & 1.573  & 0.910  & 0.907  & 1.263  & 0.163  & 0.015  & 0.467  & 0.540  & 0.664  \\
					data.17 & 0.630  & 1.089  & 0.673  & 0.633  & 0.918  & 0.816  & 0.105  & 2.876  & 2.147  & 4.751  \\
					data.18 & 0.274  & 0.297  & 0.268  & 0.234  & 0.232  & 0.746  & 0.092  & 4.487  & 6.541  & 5.494  \\
					data.19 & 0.224  & 0.211  & 0.233  & 0.187  & 0.154  & 1.099  & 0.169  & 5.531  & 13.215  & 8.190  \\
					data.20 & 0.087  & 0.118  & 0.140  & 0.064  & 0.028  & 1.029  & 0.156  & 5.933  & 9.681  & 9.526  \\
						\hline
						Average& 0.375  & 0.681  & 0.380  & 0.351  & 0.414  & 0.244  & 0.033  & 1.997  & 1.904  & 1.727  \\
						\bottomrule
					\end{tabular}%
				}
			\end{table}%

\begin{table}[t]
		\centering
		\caption{Classification: average MR (\%) based on 100 random partitions of each data set into training and test sets}\label{Table5}%
		{\small
					\begin{tabular}{lrrrrrrrrrr}
								\toprule%
								& \multicolumn{5}{c}{Miss-classification rate}& \multicolumn{5}{c}{Computational time (in seconds)}  \\
								\cmidrule(lr){2-6}\cmidrule(lr){7-11}%
								data & CART	&ERT&	EVT	&CT	&ODT& CART	&ERT&	EVT	&CT	&ODT \\
								\midrule
								data.21 & 8.00  & 17.22  & 8.70  & 8.88  & 10.68  & 0.06  & 0.01  & 2.49  & 0.19  & 0.26  \\
								data.22 & 28.90  & 33.07  & 28.05  & 32.41  & 21.79  & 0.08  & 0.01  & 3.60  & 0.25  & 0.33  \\
								data.23 & 36.06  & 39.98  & 35.02  & 36.72  & 29.37  & 0.04  & 0.00  & 0.86  & 0.08  & 0.15  \\
								data.24 & 34.10  & 38.98  & 34.28  & 36.93  & 34.81  & 0.06  & 0.00  & 0.98  & 0.13  & 0.21  \\
								data.25 & 13.76  & 17.36  & 13.85  & 13.85  & 9.92  & 0.09  & 0.00  & 1.42  & 0.27  & 0.17  \\
								data.26 & 8.02  & 9.30  & 7.18  & 6.84  & 4.16  & 0.02  & 0.00  & 0.21  & 0.10  & 0.03  \\
								data.27 & 26.36  & 91.07  & 25.08  & 24.65  & 17.78  & 0.11  & 0.01  & 3.34  & 0.65  & 0.35  \\
								data.28 & 18.15  & 20.94  & 18.03  & 19.75  & 17.15  & 0.06  & 0.00  & 0.83  & 0.28  & 0.14  \\
								data.29 & 10.47  & 11.11  & 9.79  & 10.45  & 8.06  & 0.20  & 0.01  & 2.90  & 0.71  & 0.46  \\
								data.30 & 13.86  & 18.24  & 16.48  & 14.17  & 6.01  & 0.11  & 0.00  & 0.91  & 0.57  & 0.12  \\
								data.31 & 7.78  & 9.58  & 6.68  & 7.43  & 8.29  & 0.16  & 0.00  & 0.55  & 0.34  & 0.20  \\
								data.32 & 3.74  & 4.72  & 3.28  & 3.43  & 4.14  & 0.45  & 0.03  & 1.41  & 0.56  & 0.45  \\
								data.33 & 47.92  & 46.31  & 48.45  & 49.50  & 0.04  & 0.30  & 0.01  & 1.17  & 0.11  & 0.07  \\
								data.34 & 47.55  & 75.91  & 48.34  & 50.12  & 18.10  & 0.38  & 0.01  & 1.02  & 0.10  & 0.29  \\
								data.35 & 6.74  & 9.78  & 10.37  & 7.61  & 7.48  & 0.71  & 0.04  & 1.98  & 1.78  & 0.87  \\
								data.36 & 22.67  & 24.27  & 35.25  & 24.77  & 19.04  & 0.73  & 0.06  & 3.28  & 2.06  & 1.29  \\
								data.37 & 25.00  & 36.66  & 23.09  & 31.81  & 29.75  & 0.12  & 0.01  & 0.39  & 0.45  & 0.17  \\
								data.38 & 0.04  & 16.99  & 0.25  & 0.13  & 1.27  & 0.17  & 0.01  & 0.45  & 0.36  & 0.24  \\
								data.39 & 22.86  & 44.36  & 34.09  & 32.78  & 44.66  & 1.40  & 0.06  & 2.10  & 3.27  & 1.79  \\
								data.40 & 0.00  & 0.34  & 0.17  & 0.07  & 0.08  & 1.10  & 0.05  & 3.99  & 1.54  & 1.80  \\
								\hline
								Average &  19.10  & 28.31  & 20.32  & 20.62  & 14.63  & 0.32  & 0.02  & 1.69  & 0.69  & 0.47  \\
								\bottomrule
							\end{tabular}%
						}
					\end{table}%


\bibliographystyle{imsart-nameyear}
\bibliography{ODTref2}

\end{document}